\documentclass[11pt]{amsart}
\usepackage{graphicx}
\usepackage{pgf,tikz,pgfplots}
\usepackage{mathrsfs}
\usetikzlibrary{arrows}
\usepackage{mathpple}
\usepackage{float}
\usepackage{yhmath}
\usepackage{graphicx}
\usepackage{amsmath, amssymb,color}
\usepackage[all,cmtip]{xy}

\usepackage{hyperref,graphicx}

\newtheorem{thm}{Theorem}[section]

\newtheorem{lem}[thm]{Lemma}
\newtheorem{prop}[thm]{Proposition}
\theoremstyle{definition}
\newtheorem{defn}[thm]{Definition}
\theoremstyle{remark}
\newtheorem{rem}[thm]{Remark}
\numberwithin{equation}{section}

\newcommand{\abracket}[1]{\left\langle#1\right\rangle}
\newcommand{\bbracket}[1]{\left[#1\right]}
\newcommand{\fbracket}[1]{\left\{#1\right\}}
\newcommand{\bracket}[1]{\left(#1\right)}

\newcommand{\LS}[1]{(\!(#1)\!)}
\newcommand{\PS}[1]{[\![#1]\!]}

\newcommand{\C}{\mathbb{C}}

\newcommand{\R}{\mathbb{R}}
\newcommand{\Q}{\mathbb{Q}}
\newcommand{\Z}{\mathbb{Z}}
\newcommand{\E}{\mathbb{E}}

\newcommand{\mc}{\mathcal}

\newcommand{\g}{\mathfrak{g}}
\newcommand{\lk}{\mathfrak{k}}

\newcommand{\co}{\mathcal{O}}

\newcommand{\HC}{Harish-Chandra }
\renewcommand{\*}{\bullet}
\newcommand{\pa}{\partial}
\newcommand{\OO}{{\mathcal O}}

\newcommand{\CE}{Chevalley--Eilenberg }

\newcommand{\into}{\hookrightarrow}

\newcommand{\W}{\mathcal W}

\newcommand{\iso}{\cong}

\renewcommand{\sp}{\mathfrak{sp}}

\newcommand{\hOmega}{\hat{\Omega}}

\renewcommand{\d}{d_{2n}}

\newcommand{\un}[1]{\underline{#1}}

\newcommand{\uTr}{\widehat{\Tr \ }}

\DeclareMathOperator{\Lie}{Lie}

\DeclareMathOperator{\Hom}{Hom}
\DeclareMathOperator{\Sp}{Sp}

\DeclareMathOperator{\Aut}{Aut}
\DeclareMathOperator{\End}{End}
\DeclareMathOperator{\GL}{GL}
\DeclareMathOperator{\PGL}{PGL}

\DeclareMathOperator{\gl}{gl}
\DeclareMathOperator{\Tr}{Tr}
\DeclareMathOperator{\tr}{tr}

\DeclareMathOperator{\pr}{pr}
\DeclareMathOperator{\Ch}{Ch}
\DeclareMathOperator{\trace}{tr}
\DeclareMathOperator{\Id}{Id}

\DeclareMathOperator{\Obs}{Obs}

\DeclareMathOperator{\desc}{desc}

\DeclareMathOperator{\wt}{wt}
\DeclareMathOperator{\Sym}{Sym}
\DeclareMathOperator{\Conf}{Conf}
\DeclareMathOperator{\Cyc}{Cyc}

\newtheorem{definition}{Definition}
\newtheorem{proposition}{Proposition}
\newtheorem{lemma}{Lemma}

\newtheorem{theorem}{Theorem}
\newcommand{\bd}{\begin{definition}}
	\newcommand{\ed}{\end{definition}}
\newcommand{\bp}{\begin{proposition}}
	\newcommand{\ep}{\end{proposition}}
\newcommand{\bl}{\begin{lemma}}
	\newcommand{\el}{\end{lemma}}
\newcommand{\bpf}{\begin{proof}}
	\newcommand{\epf}{\end{proof}}

\newcommand{\bt}{\begin{theorem}}
	\newcommand{\et}{\end{theorem}}
\newcommand{\bi}{\begin{itemize}}
	\newcommand{\ei}{\end{itemize}}

\newcommand{\si}[1]{(\textcolor{blue}{Si: #1})}

\begin{document}

\title[]{Geometry of Localized Effective Theories, Exact Semi-classical Approximation and the Algebraic Index}%
 \author{Zhengping Gui, Si Li, and Kai Xu}

  \address{
Z. Gui: Yau Mathematical Sciences Center, Tsinghua University, Beijing, China;
}
\email{gzp16@mails.tsinghua.edu.cn}
  \address{
S. Li:
\begin{tabular}{l}
Yau Mathematical Sciences Center, Tsinghua University, Beijing, China\\ Institute for Advanced Study, New Jersey, USA
\end{tabular}
}

\email{sili@mail.tsinghua.edu.cn}
  \address{
K. Xu: Department of Mathematics, Harvard University, Massachusetts, USA;
}
\email{kaixu@math.harvard.edu}

\thanks{}%

\begin{abstract}
In this paper we propose a general framework to study the quantum geometry of $\sigma$-models when they are effectively localized to small quantum fluctuations around constant maps. Such effective theories have surprising \emph{exact} descriptions at all loops in terms of target geometry and can be rigorously formulated. We illustrate how to turn the physics idea of exact semi-classical approximation into a geometric set-up in this framework, using Gauss-Manin connection.  As an application, we  carry out this program in details by the example of topological quantum mechanics, and explain how to implement the idea of exact semi-classical approximation into a proof of the algebraic index theorem. The proof resembles much of the physics derivation of Atiyah-Singer index theorem and clarifies the geometric face of many other mathematical constructions.

\end{abstract}
\maketitle
\tableofcontents
\section{Introduction}

The art of using quantum field theory to derive mathematical results often lies in a mysterious transition between infinite dimensional geometry and finite dimensional geometry.  Typically, one starts with a path integral in quantum field theory
$$
   \int_{\mathcal E} e^{{i\over \hbar}S}
$$
which is usually an infinite dimensional monster and difficult to study rigorously. Despite the challenging problem in defining this integral, the simplicity and symmetric beauty of the action $S$ sometimes lead to exact formal computations that could be translated into meaningful and surprising  mathematical results. This happens in the wonderland of many supersymmetric field theories. There the above path integral is often localized effectively to an equivalent integral on a finite dimensional moduli space $\mathcal M\subset \mathcal E$
$$
   \int_{\mathcal E} e^{iS/\hbar}= \int_{\mathcal M} \bracket{-}.
$$
We refer to \cite{pestun2017localization} for a comprehensive review.

We will be mainly interested in $\sigma$-models, which study the space of maps
$$
    \varphi: \Sigma\to X.
$$
For example, when $\Sigma$ is one-dimensional, this describes quantum mechanical models; when $\Sigma$ is two-dimensional, this falls into the regime of string theory. Physical models often ask for understanding the integral over the space of arbitrary maps, which is huge. Nevertheless, we expect to learn novel geometry of the target $X$ from the response to randomly throwing ``stones" $\Sigma$ into it.

In versions of topological $\sigma$-models, the path integral over the space of arbitrary maps will be localized to a finite dimensional subspace. The extreme case is when it is localized to constant maps. The path integral will be captured exactly by an effective theory in the formal neighborhood of constant maps inside the full mapping space. Exact semi-classical approximation in $\hbar\to 0$ allows us to reduce the path integral into a meaningful integral on the moduli space of constant maps, i.e., $X$. For example, in topological quantum mechanics, we find a formal result (when $\Sigma=S^1$)
$$
   \int_{\text{Map}(S^1, X)}(-) \stackrel{\hbar\to 0}{=}\int_X(-)^\prime.
$$
The LHS usually gives a physics presentation of the analytic index of certain elliptic operator. The RHS will end up with integrals of various curvature forms representing the topological index. This is the physics ``derivation" of Atiyah-Singer type index theorem \cite{alvarez1983supersymmetry,friedan1984supersymmetric,witten1982supersymmetry,witten1999index}, where we only need to compute certain one-loop Feynman diagrams as the result of exact semi-classical approximation.

Another well-studied cousin is the two-dimensional topological A-model and B-model \cite{witten1991mirror}. In the topological A-model, it is localized to holomorphic maps from a Riemann surface $\Sigma$ to a complex manifold $X$.  Such localized geometry triggers the fascinating development of Gromov-Witten (GW) theory.  In the topological B-model, it is localized to constant maps, again. Despite the simplicity of this localized data, its geometry is not very well understood. Nevertheless, throwing a sphere into a Calabi-Yau manifold in the B-model and feeling its response around the fluctuation of constant maps, we end up with period map and variation of Hodge structures (VHS). If we throw a higher genus curve, it is expected to give a certain quantization of VHS on Calabi-Yau manifolds. Instead, this picture is fairly understood if we replace $\sigma$-model by a string field theory description for the B-model \cite{bershadsky1994kodaira,costello2012quantum}.  Mirror symmetry links these two models by a version of infinite dimensional Fourier transformation \cite{hori2003mirror}.  After we translate such simple link of two monsters to localized geometry, we end up with the highly unexpected prediction on the mirror equivalence between GW and VHS.

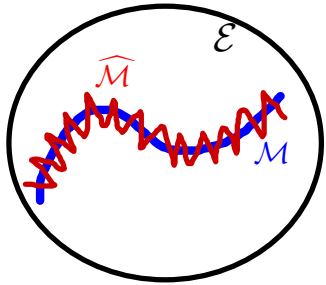
\begin{center}
\vskip -30 pt
\definecolor{ccqqqq}{rgb}{0.8,0,0}
\definecolor{qqqqff}{rgb}{0,0,1}
\begin{tikzpicture}[line cap=round,line join=round,>=triangle 45,x=1cm,y=1cm]
\clip(-5.514319618561127,-4.573416654075767) rectangle (6.582147638072778,3.3950736585167274);
\draw [rotate around={0:(0,0)},line width=2pt] (0,0) ellipse (2.0769431803137888cm and 1.8203551780495901cm);
\draw [line width=3.2pt,color=qqqqff] (-1.6623077069904444,-0.7673989145835366)-- (-1.6623077069904444,-0.7099061994854666)-- (-1.6623077069904444,-0.6179178553285547)-- (-1.6623077069904444,-0.5259295111716428)-- (-1.6508091639708304,-0.46843679607357286)-- (-1.6278120779316025,-0.4224426239951169)-- (-1.6163135349119886,-0.37644845191666093)-- (-1.5933164488727607,-0.34195282285781897)-- (-1.5588208198139186,-0.26146302172052105)-- (-1.5128266477354628,-0.16947467756360915)-- (-1.4553339326373929,-0.054489247367469254)-- (-1.397841217539323,0.02600055376982867)-- (-1.363345588480481,0.07199472584828462)-- (-1.351847045460867,0.10649035490712659)-- (-1.3173514164020252,0.15248452698558254)-- (-1.2943543303627971,0.18698015604442453)-- (-1.2598587013039553,0.20997724208365248)-- (-1.2253630722451134,0.24447287114249447)-- (-1.1908674431862714,0.27896850020133646)-- (-1.1678703571470435,0.3134641292601784)-- (-1.0873805560097456,0.3709568443582483)-- (-1.0183892978920617,0.4169510164367043)-- (-0.9838936688332198,0.4284495594563183)-- (-0.9378994967547639,0.4399481024759323)-- (-0.8229140665586242,0.4399481024759323)-- (-0.7194271793820983,0.4399481024759323)-- (-0.6619344642840285,0.4169510164367043)-- (-0.6044417491859586,0.39395393039747634)-- (-0.5469490340878888,0.3594583013386344)-- (-0.48945631898981884,0.32496267227979236)-- (-0.44346214691136293,0.2904670432209504)-- (-0.397467974832907,0.24447287114249447)-- (-0.36297234577406506,0.20997724208365248)-- (-0.32847671671522316,0.18698015604442453)-- (-0.2939810876563812,0.15248452698558254)-- (-0.2594854585975393,0.12948744094635456)-- (-0.23648837255831134,0.0949918118875126)-- (-0.19049420047985544,0.07199472584828462)-- (-0.13300148538178555,0.014502010750214678)-- (-0.08700731330332963,-0.019993618308627288)-- (-0.029514598205259743,-0.054489247367469254)-- (0.016479573873196166,-0.06598779038708324)-- (0.07397228897126605,-0.08898487642631123)-- (0.13146500406933592,-0.08898487642631123)-- (0.16596063312817785,-0.10048341944592522)-- (0.3499373214420015,-0.10048341944592522)-- (0.4534242086185273,-0.10048341944592522)-- (0.4994183806969832,-0.07748633340669724)-- (0.5684096388146671,-0.06598779038708324)-- (0.625902353912737,-0.04299070434785527)-- (0.6833950690108068,-0.03149216132824128)-- (0.7408877841088767,-0.0084950752890133)-- (0.7983804992069466,0.003003467730600689)-- (0.8328761282657885,0.02600055376982867)-- (0.8673717573246305,0.048997639809056644)-- (0.9018673863834724,0.060496182828670635)-- (0.9363630154423144,0.07199472584828462)-- (1.0053542735599983,0.12948744094635456)-- (1.074345531677682,0.16398307000519655)-- (1.120339703756138,0.2214757851032665)-- (1.1778324188542078,0.26746995718172245)-- (1.2238265909326638,0.32496267227979236)-- (1.2813193060307337,0.3709568443582483)-- (1.3273134781091895,0.4169510164367043)-- (1.3963047362268735,0.47444373153477426)-- (1.4537974513249432,0.5204379036132302)-- (1.4767945373641713,0.5549335326720721)-- (1.4997916234033992,0.5894291617309141)-- (1.534287252462241,0.6239247907897562);
\draw [line width=2pt,color=ccqqqq] (-1.8577829383238822,-0.5489265972108708)-- (-1.8117887662454262,-0.5489265972108708)-- (-1.7657945941669704,-0.5489265972108708)-- (-1.7312989651081283,-0.5604251402304847)-- (-1.6853047930296725,-0.5719236832500988)-- (-1.5818179058531467,-0.5719236832500988)-- (-1.5013281047158489,-0.5719236832500988)-- (-1.4668324756570068,-0.5604251402304847)-- (-1.5013281047158489,-0.5374280541912568)-- (-1.6048149918923746,-0.38794699493627494)-- (-1.6393106209512165,-0.330454279838205)-- (-1.6738062500100586,-0.307457193798977)-- (-1.7083018790689004,-0.28446010775974906)-- (-1.7427975081277423,-0.24996447870090707)-- (-1.7772931371865843,-0.2269673926616791)-- (-1.7427975081277423,-0.2154688496420651)-- (-1.6968033360492865,-0.2154688496420651)-- (-1.6508091639708304,-0.2269673926616791)-- (-1.5818179058531467,-0.24996447870090707)-- (-1.5358237337746907,-0.24996447870090707)-- (-1.5013281047158489,-0.26146302172052105)-- (-1.4668324756570068,-0.28446010775974906)-- (-1.443835389617779,-0.318955736818591)-- (-1.409339760558937,-0.34195282285781897)-- (-1.374844131500095,-0.353451365877433)-- (-1.3403485024412531,-0.34195282285781897)-- (-1.351847045460867,-0.295958650779363)-- (-1.363345588480481,-0.26146302172052105)-- (-1.386342674519709,-0.2269673926616791)-- (-1.397841217539323,-0.19247176360283713)-- (-1.420838303578551,-0.15797613454399517)-- (-1.443835389617779,-0.10048341944592522)-- (-1.4553339326373929,-0.054489247367469254)-- (-1.489829561696235,-0.03149216132824128)-- (-1.5128266477354628,0.003003467730600689)-- (-1.5358237337746907,0.03749909678944266)-- (-1.5703193628335326,0.060496182828670635)-- (-1.5703193628335326,0.10649035490712659)-- (-1.5243251907550768,0.0949918118875126)-- (-1.4783310186766208,0.060496182828670635)-- (-1.420838303578551,0.02600055376982867)-- (-1.386342674519709,0.014502010750214678)-- (-1.351847045460867,0.003003467730600689)-- (-1.3173514164020252,-0.03149216132824128)-- (-1.2828557873431832,-0.054489247367469254)-- (-1.2483601582843413,-0.08898487642631123)-- (-1.2138645292254995,-0.11198196246553921)-- (-1.2138645292254995,-0.06598779038708324)-- (-1.2368616152647274,-0.03149216132824128)-- (-1.2598587013039553,0.02600055376982867)-- (-1.2943543303627971,0.11798889792674058)-- (-1.3058528733824113,0.17548161302481052)-- (-1.3058528733824113,0.2214757851032665)-- (-1.3173514164020252,0.25597141416210845)-- (-1.3173514164020252,0.3019655862405644)-- (-1.3173514164020252,0.3479597583190204)-- (-1.3288499594216392,0.38245538737786233)-- (-1.3058528733824113,0.3479597583190204)-- (-1.2828557873431832,0.3134641292601784)-- (-1.2483601582843413,0.27896850020133646)-- (-1.2138645292254995,0.24447287114249447)-- (-1.1908674431862714,0.20997724208365248)-- (-1.1563718141274295,0.17548161302481052)-- (-1.1333747280882016,0.14098598396596856)-- (-1.0988790990293595,0.11798889792674058)-- (-1.0758820129901316,0.08349326886789861)-- (-1.0528849269509037,0.11798889792674058)-- (-1.0413863839312898,0.15248452698558254)-- (-1.0298878409116758,0.1984786990640385)-- (-1.0298878409116758,0.4399481024759323)-- (-1.0298878409116758,0.5204379036132302)-- (-1.0413863839312898,0.5549335326720721)-- (-1.0413863839312898,0.6009277047505281)-- (-1.0413863839312898,0.646921876828984)-- (-1.0183892978920617,0.6124262477701421)-- (-0.9953922118528338,0.5779306187113001)-- (-0.9838936688332198,0.5434349896524582)-- (-0.9608965827939919,0.5089393605936162)-- (-0.9493980397743779,0.47444373153477426)-- (-0.9378994967547639,0.4399481024759323)-- (-0.914902410715536,0.4054524734170903)-- (-0.891905324676308,0.3709568443582483)-- (-0.857409695617466,0.3479597583190204)-- (-0.8459111525978521,0.3134641292601784)-- (-0.8459111525978521,0.26746995718172245)-- (-0.8344126095782382,0.23297432812288046)-- (-0.7884184374997822,0.23297432812288046)-- (-0.7654213514605542,0.26746995718172245)-- (-0.7539228084409403,0.32496267227979236)-- (-0.7309257224017124,0.3709568443582483)-- (-0.7079286363624844,0.4054524734170903)-- (-0.7079286363624844,0.45144664549554625)-- (-0.7079286363624844,0.4974408175740022)-- (-0.6964300933428704,0.5319364466328442)-- (-0.6849315503232565,0.5664320756916862)-- (-0.6619344642840285,0.5319364466328442)-- (-0.6619344642840285,0.46294518851516026)-- (-0.6619344642840285,0.39395393039747634)-- (-0.6504359212644145,0.3594583013386344)-- (-0.6504359212644145,0.3134641292601784)-- (-0.6389373782448006,0.27896850020133646)-- (-0.6389373782448006,0.23297432812288046)-- (-0.6274388352251865,0.1984786990640385)-- (-0.5929432061663447,0.17548161302481052)-- (-0.5699461201271167,0.14098598396596856)-- (-0.5354504910682747,0.17548161302481052)-- (-0.48945631898981884,0.25597141416210845)-- (-0.44346214691136293,0.3134641292601784)-- (-0.420465060872135,0.3479597583190204)-- (-0.385969431813293,0.3594583013386344)-- (-0.3514738027544511,0.38245538737786233)-- (-0.32847671671522316,0.4169510164367043)-- (-0.31697817369560916,0.38245538737786233)-- (-0.31697817369560916,0.10649035490712659)-- (-0.31697817369560916,0.014502010750214678)-- (-0.32847671671522316,-0.019993618308627288)-- (-0.32847671671522316,-0.06598779038708324)-- (-0.2939810876563812,-0.054489247367469254)-- (-0.28248254463676725,-0.019993618308627288)-- (-0.24798691557792532,0.014502010750214678)-- (-0.23648837255831134,0.048997639809056644)-- (-0.2019927434994694,0.07199472584828462)-- (-0.16749711444062748,0.0949918118875126)-- (-0.13300148538178555,0.12948744094635456)-- (-0.11000439934255758,0.16398307000519655)-- (-0.07550877028371565,0.18698015604442453)-- (-0.0525116842444877,0.2214757851032665)-- (-0.018016055185645766,0.20997724208365248)-- (-0.0065175121660317895,0.16398307000519655)-- (0.016479573873196166,0.060496182828670635)-- (0.08547083199088003,-0.14647759152438117)-- (0.10846791803010798,-0.2154688496420651)-- (0.10846791803010798,-0.26146302172052105)-- (0.11996646104972196,-0.295958650779363)-- (0.13146500406933592,-0.26146302172052105)-- (0.13146500406933592,-0.2154688496420651)-- (0.13146500406933592,-0.16947467756360915)-- (0.1429635470889499,-0.11198196246553921)-- (0.16596063312817785,-0.054489247367469254)-- (0.16596063312817785,-0.0084950752890133)-- (0.17745917614779183,0.02600055376982867)-- (0.1889577191674058,0.060496182828670635)-- (0.1889577191674058,0.10649035490712659)-- (0.21195480520663376,0.07199472584828462)-- (0.22345334822624774,0.02600055376982867)-- (0.2464504342654757,-0.0084950752890133)-- (0.2579489772850897,-0.04299070434785527)-- (0.26944752030470365,-0.07748633340669724)-- (0.26944752030470365,-0.12348050548515319)-- (0.26944752030470365,-0.16947467756360915)-- (0.28094606332431765,-0.2154688496420651)-- (0.3039431493635456,-0.27296156474013505)-- (0.3499373214420015,-0.27296156474013505)-- (0.36143586446161546,-0.2384659356812931)-- (0.3959314935204574,-0.18097322058322313)-- (0.4304271225792993,-0.12348050548515319)-- (0.44192566559891333,-0.06598779038708324)-- (0.4649227516381413,-0.019993618308627288)-- (0.47642129465775523,0.03749909678944266)-- (0.4994183806969832,0.08349326886789861)-- (0.4994183806969832,0.12948744094635456)-- (0.5224154667362112,0.0949918118875126)-- (0.5339140097558251,0.060496182828670635)-- (0.556911095795053,0.02600055376982867)-- (0.5684096388146671,-0.0084950752890133)-- (0.591406724853895,-0.04299070434785527)-- (0.602905267873509,-0.07748633340669724)-- (0.614403810893123,-0.11198196246553921)-- (0.6374008969323509,-0.14647759152438117)-- (0.6488994399519649,-0.19247176360283713)-- (0.6718965259911929,-0.24996447870090707)-- (0.6718965259911929,-0.2039703066224511)-- (0.6718965259911929,-0.14647759152438117)-- (0.6833950690108068,-0.11198196246553921)-- (0.6833950690108068,-0.06598779038708324)-- (0.7063921550500348,-0.03149216132824128)-- (0.7178906980696488,0.003003467730600689)-- (0.7293892410892627,0.03749909678944266)-- (0.7293892410892627,0.08349326886789861)-- (0.7408877841088767,0.14098598396596856)-- (0.7638848701481047,0.1984786990640385)-- (0.8098790422265606,0.18698015604442453)-- (0.8213775852461745,0.15248452698558254)-- (0.8558732143050165,0.11798889792674058)-- (0.8903688433638585,0.08349326886789861)-- (0.9018673863834724,0.048997639809056644)-- (0.9133659294030864,0.014502010750214678)-- (0.9133659294030864,-0.03149216132824128)-- (0.9248644724227003,-0.07748633340669724)-- (0.9478615584619283,-0.1349790485047672)-- (0.9593601014815423,-0.10048341944592522)-- (0.9593601014815423,0.15248452698558254)-- (0.9593601014815423,0.23297432812288046)-- (0.9708586445011562,0.26746995718172245)-- (0.9708586445011562,0.3134641292601784)-- (0.9823571875207703,0.3479597583190204)-- (0.9823571875207703,0.39395393039747634)-- (1.0053542735599983,0.4284495594563183)-- (1.0398499026188401,0.4399481024759323)-- (1.074345531677682,0.4284495594563183)-- (1.09734261771691,0.39395393039747634)-- (1.120339703756138,0.3594583013386344)-- (1.131838246775752,0.32496267227979236)-- (1.1663338758345938,0.2904670432209504)-- (1.189330961873822,0.25597141416210845)-- (1.2468236769718917,0.15248452698558254)-- (1.2928178490503477,0.08349326886789861)-- (1.3158149350895756,0.11798889792674058)-- (1.3158149350895756,0.45144664549554625)-- (1.3158149350895756,0.5319364466328442)-- (1.3043163920699616,0.6009277047505281)-- (1.3043163920699616,0.6699189628682121)-- (1.3273134781091895,0.6354233338093701)-- (1.3618091071680314,0.6009277047505281)-- (1.3848061932072595,0.5664320756916862)-- (1.4308003652857153,0.5319364466328442)-- (1.4882930803837853,0.4974408175740022)-- (1.4997916234033992,0.46294518851516026)-- (1.5112901664230132,0.4284495594563183)-- (1.534287252462241,0.39395393039747634)-- (1.5572843385014692,0.3594583013386344)-- (1.591779967560311,0.33646121529940637);
\draw (0.4994183806969832,1.716286377653085) node[anchor=north west] {\Large$\mathbf{\mathcal E}$};
\draw (1.0398499026188401,0.13646121529940637) node[anchor=north west] {${\color{blue}\mathcal M}$};
\draw (-1.0758820129901316,1.3013009474569453) node[anchor=north west] {${\color{red}\widehat{\mathcal M}}$};
\end{tikzpicture}
\end{center}
\vskip -70 pt

Let  ${\widehat{\mathcal M}}$ be the {formal neighborhood} of $\mathcal M$ inside $\mathcal E$. The main idea in this paper is to study the effective theory on $\widehat{\mathcal M}$ that would sit in between the original physics and the localized integral
$$
   \int_{\mathcal E} e^{iS/\hbar}"="  {\int_{\widehat{\mathcal M}} e^{iS^{eff}/\hbar}}= \int_{\mathcal M} \bracket{-}.
$$
The pair ${(\widehat{\mathcal M},S^{eff})}$ will be called the \emph{localized effective theory}.

In most physics applications, the effective theory in the small neighborhood of the localized geometry has complicated quantum corrections and its precise form  is not much used.  Once the exact semi-classical approximation is available, we can  derive an explicit quantity on the localized moduli itself after integrating out the small quantum fluctuations around it, without knowing much detailed information. Miraculous cancellation happens along integration, usually by supersymmetry.

\noindent \textbf{Q}: How useful is the effective theory around the localized geometry? Can we actually use it to prove interesting mathematical results along this conceptually clean physics journey, in other words,  to turn this beautiful ``physics argument" of \emph{exact semi-classical approximation} into actual theorems?

This question is the main motivation of the current work. Such effective theory sits in between the full infinite dimensional geometry and the localized finite dimensional geometry, like a bridge. To start with, we will present in this paper a geometric way to study the effective theory in the formal neighborhood of constant maps. Our main targets will be to understand index type problems (the usual Atiyah-Singer type index as well as many other unknown index from higher dimensional worldsheet), in terms of a rigorous formulation of the method of exact semi-classical approximation.

We show that this effective theory has an \emph{exact} description in terms of geometry and captures many useful informations about the infinite dimensional monster. We illustrate this  in this paper by a detailed study of the example of topological quantum mechanics. The effective interaction at all loops is exactly identified \cite{BVQandindex} with Fedosov's flat connection \cite{fedosov1994} on the Weyl bundle. Using this, we explain how to implement the idea of exact semi-classical approximation into a proof of the algebraic index theorem. The proof resembles much of the physics derivation of Atiyah-Singer index theorem described above and clarifies the geometric face of many other mathematical constructions.

The rest of the introduction is mainly about a journey which explains the key idea and the physics model behind the main construction in this paper. It is independent from the remainder of
the paper and can be skipped by readers less interested in motivations by physics, and not familiar with quantum field theory. Nevertheless, we feel it helpful to explain this story since it would illustrate how to generalize the construction to many other models, including the two dimensional $\sigma$-models that have attracted most interests in the last few decades. The  mathematical exposition starts from Section \ref{section:formal},  and will be presented with precise definitions in a purely finite dimensional context.

\subsection{The geometry of localized effective theory}

Let us now explain the basic idea/philosophy of our approach to the effective theory around the localized geometry of constant maps.

\vskip 10pt

\noindent \textbf{(1) Local model}. Let us assume that $X$ is locally modeled by a geometry $T^{Model}$ that will be relevant for our $\sigma$-model. $X$ is built up from gluing pieces of $T^{Model}$ via its automorphism groups $\Aut(T^{Model})$. For example, in quantum mechanics, $X$ is a symplectic manifold and locally modeled (via the choice of Darboux coordinates) by
the standard linear symplectic structure $
T^{Model}=(\R^{2n},\omega).
$

 Intuitively, a constant map specifies a point on the target $X$. Small quantum fluctuations around this map produce  certain geometric structures that depend on $\Sigma$ and the local model $T^{Model}$ (and of course, on the $\sigma$-model itself). Let us vaguely call it $\mathcal F^\hbar(\Sigma, T^{Model})$, where $\hbar$ is the quantization parameter.  It exists locally around any point in $X$.

 Usually, the $\sigma$-model action has interaction terms that depend on the geometry of $X$. It often becomes free when $X$ is flat and we assume this is the case. The interplay between the interacting data and the gluing data on the target will play an important role later. In this case,  $\mathcal F^\hbar(\Sigma, T^{Model})$ captures structures of a free quantum field theory.  A free theory is not trivial, it allows rich algebraic structures of observables that we can put on $\Sigma$. As developed systematically in \cite{Costello_2011,Factorization1}, they form a factorization algebra on $\Sigma$. This fully characterizes what is happening locally on $X$.

\vskip 20pt

\noindent \textbf{(2) Gluing}. If the theory is consistent at the quantum level, then the existence of effective theory tells us that we can glue $\mathcal F^\hbar(\Sigma, T^{Model})$ consistently to form a sheaf on $X$. The gluing is represented by how $\Aut(T^{Model})$ acts on $\mathcal F^\hbar(\Sigma, T^{Model})$. This can be implemented by the method of coupling with background symmetry, and is exactly how the interaction terms will be introduced.

Let us illustrate this by a concrete topological quantum mechanics example following \cite{BVQandindex}. It studies maps from $S^1$ to a symplectic manifold $(X, \omega)$.  At each point $p\in X$, the tangent plane $T_pX$ is a linear symplectic space. Quantum fluctuations deform the algebra of functions on $T_pX$ to the associated Weyl algebra. These pointwise Weyl algebras form a vector bundle, the Weyl bundle $\W$ on $X$. Fedosov \cite{fedosov1994} shows $\W$ carries a flat connection of the form $D=\nabla+{1\over \hbar}[\gamma, -]$ where $\nabla$ is a symplectic connection, and $\gamma$ is a $\W$-valued 1-form. As shown in \cite{BVQandindex}, the connection $\gamma$ precisely produces the interaction term in the effective topological quantum mechanical model. The flatness of $D$ has the interpretation of quantum master equation, which is a quantum gauge consistency condition for gluing. Quantum observables are given by $D$-flat sections \cite{fedosov1994}. These are objects that can be consistently glued on $X$, leading to a deformation quantization of the Poisson algebra $C^\infty(X)$.

Typically, the gluing symmetry will be represented by a \HC pair $(\g, K)$. $K$ represents the linearized symmetry that allows us to talk about $K$-compatible connections. It sets the kind of geometry we study. $\g$ is the Lie algebra of the full  (formal) non-linear symmetries. In the topological quantum mechanics example above, $K$ is the symplectic group, and $\g$ is the associated Lie algebra of the Weyl algebra. The geometry of $X$ itself that reflects the outcome of  gluing local pieces will lead to (see Section \ref{section:formal} for a detailed discussion) a principal $K$-bundle $P$
$$
P\to X
$$
together with a flat $(\g,K)$-connection $A$. $(\g,K)$-connection generalizes the notation of connection on $P$, by allowing it to take values in the bigger Lie algebra $\g$.  Flatness means the curvature $dA+{1\over 2}[A,A]=0$ vanishes. Chern-Weil type formalism says we can transfer Lie algebraic constructions to geometric objects on $X$. This leads to the descent map (see Section \ref{desc})
$$
\desc: C^\*(\g, K; -)\to \Omega^\*(X, P\times_K -)
$$
which sends the Lie algebra cochain complex for a $(\g,K)$-module to the de Rham complex of the associated flat vector bundle. It is compatible with differentials and descents to cohomologies.

The complex $C^\*(\g, K; -)$ is our algebraic framework of describing how to glue objects under the
local symmetry transformation. A consistent gluing will end up with a cocycle in this complex.

Let us explain this point. Let $(V, \delta)$ be some BRST complex constructed from the local model $\mathcal F^\hbar(\Sigma, T^{Model})$. Assume $V$ carries a structure of $(\g,K)$-module, telling how it is transferred under a gluing symmetry. We are interested in constructing some element $\alpha_0\in V$ that is BRST closed $\delta\alpha_0=0$ and can be glued globally on $X$. This naive requirement is unrealistic unless $\alpha_0$ is stable under the $\g$-action. However, if the choice of $\alpha_0$ is natural, the gluing transformation will modify it by BRST exact terms. Therefore we are led to look for $\alpha_1\in C^1(\g,K;V)$ such that
$$
   \pa_{\Lie}\alpha_0+\delta \alpha_1=0.
$$
Here $\pa_{\Lie}$ is the \CE differential. This equation tells exactly how to correct $\alpha_0$ by BRST exact term under the gluing transformation. Iterating this process, we need to find
$$
\alpha= \alpha_0+\alpha_1+\cdots, \quad \alpha_i\in C^i(\g,K;V)
$$
satisfying
$$
  (\pa_{\Lie}+\delta) \alpha=0.
$$
Such $\alpha$ can be glued globally on $X$, by the method of descent  described above.

This set-up follows closely the philosophy of Gelfand-Kazhdan formal geometry. Our presentation here follows Costello's idea \cite{costello2010geometric} on using the factorization/BV machine to bridge the algebra and the path integral in the formal neighborhood of constant maps. This is used in \cite{Grady:2011jc} to study the effective theory of topological quantum mechanics when the target is a cotangent bundle, and generalized in \cite{BVQandindex} to arbitrary symplectic targets to understand quantum corrections at all loops. As another solid example, \cite{gorbounov2016chiral} studies the gluing of BV quantization of $\beta-\gamma$ system  in this framework.  There the Witten genus appears as a universal object computed in the local model coupled with background symmetry, which descents to the geometric construction described in \cite{costello2010geometric}. Such gluing can be viewed as the analogue of Fedosov connection for a vertex algebra bundle on the target, which is equivalent to the quantum master equation by the method described in \cite{Li:2016gcb}.

\noindent \textbf{(3) Exact semi-classical approximation}.  Armed with the above local story and gluing,  we now come to the main message in this paper on  the strategy to implement the method of exact semi-classical approximation from the viewpoint of localized effective theory. This will allow us to prove interesting results by exactly the same way as what physicists would do.

So far we have been talking about algebraic structures. Another important object is the expectation value map, or correlation function
$$
\abracket{-}: (\Obs^\hbar_\Sigma, \delta) \to \C\LS{\hbar}.
$$
Here $\Obs^\hbar_\Sigma$ is a suitable version of quantum observables on $\Sigma$, and $\delta$ plays the role of BRST differential. Mathematically, the chain complex $(\Obs^\hbar_\Sigma, \delta)$ models the factorization homology of the algebra of local observables on $\Sigma$ (see \cite{Factorization1} for a precise formulation in perturbation theory). A consistent theory will imply that $\abracket{-}$ is a cochain map, hence it descents to an evaluation map on $\delta$-cohomologies. Here we view the coefficient ring $\C\LS{\hbar}$ as a complex with zero differential. In general, $\abracket{-}$ could be valued in some differential graded algebra.

We are mainly interested in calculating expectation values of observables. This is usually very difficult. Physics principle tells us that we should explore its dependence on $\hbar$. Thinking about $\hbar$ as a deformation parameter for a family of algebraic structures,  we expect to have a version of Gauss-Manin connection  available to compute $\hbar$-variations. For example, for a family of (homotopy) associative algebras, we have Getzler's Gauss-Manin connection \cite{getzler1993cartan} on the periodic cyclic homologies. In good situations, we can find nice quantities that are invariant under $\hbar$-variations. In this case, we can take $\hbar\to 0$ and the semi-classical approximation becomes exact. Standard technique in quantum field theory says that this can be computed by one-loop Feynman diagrams. As we will see by a concrete example, this turns intuitive physics reasoning into interesting mathematical proof.

In some special examples of field theory, the full quantization is one-loop exact by nature. BF theory is of this type, as well as its various chiral and homotopic generalizations, for example \cite{costello2010geometric,gorbounov2016chiral,BVQandindex,Grady:2011jc,williams2018renormalization}. In such case, higher loops are forbidden and the above analysis is greatly simplified.

\subsection{A case study of topological quantum mechanics} The main body of this paper is to carry out the above idea through the concrete example of topological quantum mechanics. An an application, we give an explicit construction of the universal trace map on periodic cyclic tensors of matrix Weyl algebras. This leads to a simple geometric proof of the universal algebraic index theorem via one-loop Feynman diagram computations, so physics works.

One way to formulate topological quantum mechanics is to apply AKSZ \cite{Alexandrov_1997} construction to maps
$$
       S^1_{dR}\to (X, \omega).
$$
Here $(X, \omega)$ is a symplectic manifold. $S^1_{dR}$ is the supermanifold with underlying topology $S^1$ and with structure ring the sheaf of de Rham complex $\Omega_{S^1}^\bullet$ on $S^1$.

Motivated by the discussion above, we first study the local model
$$
S^1_{dR}\to \R^{2n}, \quad \omega=\sum_{i=1}^n dp^i \wedge d q^i.
$$
Here $p^1,...,p^n,q^1,...,q^n$ are linear coordinates on $\R^{2n}$. The fields can be organized into superfields
$$
   \mathbb{P}^i(t)=p^i(t)+\eta^i(t)dt,\quad \mathbb{Q}^i(t)=q^i(t)+\xi^i(t)dt.
$$
Here $q^i(t), p^i(t)$ are bosons, and their anti-fields $ \eta^i(t), \xi^i(t)$ are fermions. $t$ is the coordinate on $S^1$.

The action takes the form
$$
S=\sum_{i}\int_{S^1} \mathbb{P}^i d\Q^i= \sum_{i}\int_{S^1} p^i dq^i.
$$
The nontrivial BRST transformation is
$$
\delta \eta^i=\pa_t p^i, \quad \delta \xi^i=\pa_t q^i.
$$

The BRST cohomology of local observables at a point $t\in S^1$ is generated by $p^i(t), q^i(t)$. The theory is topological on $S^1$, since $t$-variation of $p^i(t), q^i(t)$'s are BRST exact. The fusion of local observables at the quantum level turns out to be the Moyal-Weyl product. Therefore we can identify local observables with the Weyl algebra  $\mathcal{W}_{2n}$ (we use a formally completed version, see Section \ref{Weyl}). As a general feature of topological field theory \cite{witten1991introduction}, we can construct non-local observables by topological descent. Let $\OO^{(0)}(t)$ be a local observable representing a BRST class.  Its topological descendent is a 1-form valued observable $\OO^{(1)}(t) dt$ such that
$$
\pa_t \OO^{(0)}(t)+ \delta \OO^{(1)}(t)=0.
$$

Explicitly in this example, given $\OO(p^i,q^i) \in \mathcal{W}_{2n}$, we consider the following observable $\OO(\mathbb{P}^i,\mathbb{Q}^i)$ which is valued in $\Omega_{S^1}^\bullet$. If we decompose it into $0$-form and $1$-form on $S^1$
$$
\OO(\mathbb{P}^i,\mathbb{Q}^i)=\OO^{(0)}+  \OO^{(1)}dt,
$$
then $\OO^{(1)}$ is the topological descent of $\OO^{(0)}$.

In applications, we are also interested in coupling this system with a rank $r$ vector bundle $E$ on $X$. In the local model, this will allow local observables to be matrix valued. In this general setting, we will be interested in the expectation value of the following observable in our free theory (which will be defined in Section 3 Definition \ref{defn-free-expectation})
$$
\abracket{ \int_{t_0=0<t_1<\cdots< t_m<1} dt_1\cdots dt_m \OO^{(0)}_0(t_0) \OO_1^{(1)}(t_1) \cdots \OO_m^{(1)}(t_m) }_{free}, \quad \OO_i\in \gl_r(\mathcal{W}_{2n}).
$$

\begin{figure}[h]

\definecolor{rvwvcq}{rgb}{0.08235294117647059,0.396078431372549,0.7529411764705882}
\begin{tikzpicture}[line cap=round,line join=round,>=triangle 45,x=1cm,y=1cm, scale=0.8,transform shape]
\clip(-6.87,-3.54) rectangle (7.47,4.3);
\draw [line width=2pt] (-0.01,0) circle (2.285607140345865cm);
\draw (-0.45,3.3) node[anchor=north west] {{$\mathcal{O}^{(0)}_0(t_0)$}};
\draw (-4.45,0.18) node[anchor=north west] {\Huge{$\int_{}$}};
\draw (1.83,2.6) node[anchor=north west] {{$\mathcal{O}^{(1)}_m(t_m)dt_m$}};
\draw (2.63,0.94) node[anchor=north west] {{}};
\draw (-4,2.62) node[anchor=north west] {{$\mathcal{O}^{(1)}_1(t_1)dt_1$}};
\draw (-3.29,0.86) node[anchor=north west] {{}};
\draw (-5.2,-1) node[anchor=north west] {{\tiny${t_0=0<t_1<\cdots< t_m<1} $}};
\draw (2.27,-0.28) node[anchor=north west] {{$\cdot\cdot\cdot$}};
\draw (2.05,-1.18) node[anchor=north west] {{$\mathcal{O}^{(1)}_k(t_k)dt_k$}};
\draw (-0.13,-2.26) node[anchor=north west] {{}};
\draw (-2.41,-1.54) node[anchor=north west] {{}};
\begin{scriptsize}
\draw [fill=rvwvcq] (-2.29,0.16) circle (2.5pt);
\draw [fill=rvwvcq] (-0.1476741265598887,2.281456954421006) circle (2.5pt);
\draw [fill=rvwvcq] (-1.7250258764177986,1.5108561292252032) circle (2.5pt);
\draw [fill=rvwvcq] (1.4922446916995062,1.7225739131487674) circle (2.5pt);
\draw [fill=rvwvcq] (2.25417072637637,0.31229941053467203) circle (2.5pt);
\draw [fill=rvwvcq] (1.8207956964465248,-1.3682788889231925) circle (2.5pt);
\draw [fill=rvwvcq] (0.2272102048344094,-2.273264462996424) circle (2.5pt);
\draw [fill=rvwvcq] (-1.5107707124277856,-1.7238582507616451) circle (2.5pt);
\end{scriptsize}
\end{tikzpicture}

  \caption{}\label{Freeobservable}
\end{figure}

We have put $\OO_0$ at a fixed point. This is related to the usual gauge fixing of the extra $S^1$-rotation symmetry. For simplicity, we will just write the above expression as
$$
\abracket{\OO_0\otimes \OO_1\otimes \cdots \otimes \OO_m}_{free}.
$$
This tensor notation is to match with cyclic chain convention below. This expectation value can be rigorously defined using Feynman diagrams, where the propagator is given by the integral kernel of the inversion of the operator ${d\over dt}$ on $S^1$.  See Definition \ref{defn-free-expectation} for the precise formula.

We could also consider the more familiar expression
$
\abracket{ \OO^{(0)}_0(t_0) \int_{S^1}dt_1\OO_1^{(1)}(t_1) \cdots \int_{S^1}dt_m\OO_m^{(1)}(t_m) }_{free}
$.
This is related to the above correlation by
$$
\abracket{ \OO^{(0)}_0(t_0) \int_{S^1}dt_1\OO_1^{(1)}(t_1) \cdots \int_{S^1}dt_m\OO_m^{(1)}(t_m) }_{free}=\sum_{\sigma\in S_m} \pm \abracket{\OO_0\otimes \OO_{\sigma(1)}\otimes \cdots \otimes \OO_{\sigma(m)}}_{free}.
$$
The LHS is related to the study of Lie algebra chains, and terms in the RHS are related to the study of cyclic chains. We will focus on the cyclic case which is relevant for the universal trace and index.

This model has zero modes,  given by
$$
  H^\bullet(S^1)\otimes \R^{2n}.
$$
These are the configurations when $p^i, \eta^i, q^i, \xi^i$ are constants. The algebra of functions on $H^\bullet(S^1)\otimes \R^{2n}$ is the de Rham forms on $\R^{2n}$, denoted by $\hOmega^{-\*}_{2n}$ (we work with a formal version, see definition \ref{de Rham}). Our expectation $\abracket{\OO_0\otimes \OO_1\otimes \cdots \otimes \OO_m}_{free}$ should be viewed as valued in $\hOmega^{-\*}_{2n}$. This gives
$$
\abracket{-}_{free}: C_{-\bullet}(\gl_r(\mathcal{W}_{2n}))\to \hOmega^{-\*}_{2n}
$$
from cyclic tensors of $\gl_r(\mathcal{W}_{2n})$ to $\hOmega^{-\*}_{2n}$. See Convention below for the definition of cyclic tensors and related homological operators.

This model comes from a BV theory, so the zero modes carry a BV structure.  Let $\Pi$ denote the Poisson tensor. Then the BV operator $\Delta$ acting on $\hOmega^{-\*}_{2n}$ is given by the Lie derivative $\Delta=\mathcal L_{\Pi}$.

The crucial property of $\abracket{-}_{free}$ is that it intertwines the Connes operator $B$ with the de Rham differential $\d$ on $\hOmega^{-\*}_{2n}$, and intertwines the Hochschild differential $b$ with the BV operator $\Delta$
$$
\abracket{B(-)}_{free}=\d\abracket{-}_{free}, \quad \abracket{b(-)}_{free}=\hbar \Delta\abracket{-}_{free}.
$$
This is nothing but an explicit construction of the deformed HKR map between the noncommutative differential forms of the Weyl algebra and the commutative counterpart.

We can further integrate out zero modes by performing Berezin integration over the fermionic BV lagrangian $H^1(S^1)\otimes \R^{2n}\into H^\bullet(S^1)\otimes \R^{2n}$. This will allow us to land on actual numbers $\C\LS{\hbar}$. This is most elegantly described by working with periodic cyclic tensors as follows. Let us $\C[u,u^{-1}]$-linearly extend the expectation map $\abracket{-}_{free}$ to
$$
\abracket{-}_{free}: CC^{per}_{-\bullet}(\gl_r(\mathcal{W}_{2n}))\to \hOmega^{-\*}_{2n}[u,u^{-1}].
$$
Here $u$ is a variable of degree $2$ representing the cyclic parameter.  Then $\abracket{-}_{free}$ intertwines
$$
(\hbar \Delta+ u\d)\abracket{-}_{free}=\abracket{(b+u B)(-)}_{free}.
$$
The composition of $\abracket{-}_{free}$ with the Berezin integration (see Definition \ref{BV-integration})
$$
\int_{BV}: \hOmega^{-\*}_{2n}[u,u^{-1}]\to \C\LS{\hbar}[u,u^{-1}]
$$
leads to a trace map
$$
\Tr=\int_{BV}\circ \abracket{-}_{free}: CC^{per}_{-\bullet}(\gl_r(\mathcal{W}_{2n}))\to \C\LS{\hbar}[u,u^{-1}].
$$
By construction, $\Tr ((b+uB)(-))=0$ and descents to periodic cyclic homologies. This is basically the same trace map that Feigin-Felder-Shoikhet \cite{feigin2005hochschild} used in their study of algebraic index.

This completes the description of our local model. The next step is to do gluing and our goal is to glue the trace map $\Tr$. The relevant \HC pair $(\g, K)$ turns out to be (see Section \ref{sec: GF})
$$
\mathfrak{g}=\mathcal{W}^+_{2n}\cdot\text{Id}+\hbar\gl_r(\mathcal{W}^+_{2n}), \quad K=\Sp_{2n}\times \GL_r.
$$
$\Sp_{2n}$ indicates an underlying symplectic geometry, and $\GL_r$ indicates a rank $r$ vector bundle.

Taking into account of the central extensions (see Section \ref{sec: GF}), it turns out that the relevant gluing construction is modeled by the following Lie algebra cochain complex
$$
  C^\bullet(\g,\mathfrak{h}; -).
$$
Here $\mathfrak{h}=\Lie(K)+ Z(\g)$ and $Z(\g)$ is the center of $\g$.

As we briefly mentioned above, the gluing data can be equivalently described by turning on interactions. It turns out that the relevant interaction is very simple from this viewpoint: it is the identity map $\g\to \g$, viewed as a 1-cochain $\widehat\Theta\in C^1_{\Lie}(\g;\g)$.

Intuitively, the interaction represents a flat connection that will be used for gluing in geometry. $\widehat\Theta$ indeed can be viewed as a universal flat connection as follows.  Let $(C_{\Lie}^\bullet(\mathfrak{g}),\pa)$ be the \CE complex and $M$ be an arbitrary $\g$-module.  We consider the tensor product
$$
C^\bullet_{\Lie}(\g)\otimes M.
$$
It carries a differential $\pa$ which acts on the $C^\bullet_{\Lie}(\g)$-factor. Let us view $\widehat\Theta\in C^1_{\Lie}(\g)\otimes \g$, so it defines
$$
\widehat\Theta: C^\bullet_{\Lie}(\g)\otimes M \to C^{\bullet+1}_{\Lie}(\g)\otimes M
$$
by combining the wedge product on the $C^1_{\Lie}(\g)$-factor and the application of $\g$-factor to $M$. Then  $\pa+\widehat\Theta$ defines a differential, which is precisely the \CE differential under the vector space identification $C^\bullet_{\Lie}(\g;M) \iso C^\bullet_{\Lie}(\g)\otimes M$. In this sense, $\widehat\Theta$ is the universal flat connection.

Guided by this universal property, we can form an expectation map $\abracket{-}_{int}$ by inserting $\widehat\Theta$ as the interaction to our free theory. Here we further view $\widehat\Theta$ as an element of $C^1_{\Lie}(\g;\gl_r(\W_{2n}))$ using the embedding $\g\into \gl_r(\W_{2n})$, hence it is valued in local observables. We then add the following ($\g^*$-valued) interaction term into our free theory
$$
   \int_{S^1} \widehat\Theta(\mathbb{P}^i,\mathbb{Q}^i).
$$

Intuitively, the expectation map in our interacting theory would be
$$
\abracket{\OO_0\otimes \OO_1\otimes \cdots \otimes \OO_m}_{int}
"="\abracket{ \int_{t_0=0<t_1<\cdots< t_m<1} dt_1\cdots dt_m \OO^{(0)}_0(t_0) \OO_1^{(1)}(t_1) \cdots \OO_m^{(1)}(t_m)  e^{{1\over \hbar}\int_{S^1}\hat \Theta}}_{free}
$$
However, since we work with matrices, we need to modify this formula to take care of the order of integration on $S^1$. The precise expression is Definition \ref{defn:interaction}. This leads to the gluing of trace map
$$
\uTr := \int_{BV}\circ \abracket{-}_{int}
$$
which turns out to lie in the correct Lie algebra cochain
$$
\uTr  \in C^\*(\mathfrak{g},\mathfrak{h};\Hom_{\mathbb{K}}(CC^{per}_{-\bullet}(\gl_r(\mathcal{W}_{2n})),\mathbb{K})),  \quad \mathbb{K}=\C\LS{\hbar}[u,u^{-1}]
$$
and satisfies the cocycle condition
$$
 (\pa_{\Lie}+ b+ uB) \uTr=0.
$$
This is our explicit construction of the universal trace map formulated in \cite{bressler2002riemannI, bressler2002riemannII}.

The universal index arises by computing its value on the cyclic tensor $1$ (i.e. the partition function)
$$
\text{Index}=\uTr(1) \in H^\bullet(\g,\mathfrak{h};\mathbb{K}).
$$
With the underlying topological field theory at hands, this index can be computed exactly as what physicists would do. It is invariant under $\hbar$-variation (up to a well-controlled central term), and so computed via one-loop Feynman diagrams by the exact semi-classical approximation $\hbar\to 0$. This proves the universal algebraic index theorem \cite{feigin1989riemann} (see Theorem \ref{index thm} for a precise statement in our context), which expresses the universal partition function ``$\uTr(1)$" in terms of a universal Lie algebra cohomology. By the descent construction, this universal index implies the usual algebraic index theorem \cite{Fedosov:1996fu,fedosov-bundle,nest1995} for a rank $r$ bundle $E$ on $X$
$$
\uTr(1)=\int_X e^{-\omega_{\hbar}/\hbar}(\hat{A}(X)\cdot Ch(E))\in H^{\bullet}(X,\mathbb{C}\LS{\hbar}).
$$
 $\Tr$ is now the normalized trace operator on the quantum deformation of the algebra $\Gamma(X, \End(E))$, and $\omega_{\hbar}$ is the 2-form  characteristic class of the deformation (see Section \ref{sec:Fedosov}).

This topological quantum mechanics example presented along the idea of localized effective theory is closely related to the works of Nest-Tsygan \cite{nest1995}, Feigin-Felder-Shoikhet \cite{feigin2005hochschild}, Willwacher \cite{WILLWACHER2015277}, Cattaneo-Felder-Willwacher \cite{CATTANEO20111966} and Pflaum-Posthuma-Tang \cite{PFLAUM20101958}. It can be viewed as a quantum field theoretic interpretation of them. The propagator of this example is used in \cite{feigin2005hochschild},  and our expectation value map in the free theory coincides with their trace map.  To derive the algebraic index, they use the relation between Lie algebra cohomology and Hochschild homology to reduce the computation to tree and one loop diagram. In \cite{WILLWACHER2015277} the author extended the result \cite{feigin2005hochschild} to cyclic complex, and use the method in \cite{feigin2005hochschild} to reduce the computation of index to certain Lie subalgebra of the matrix valued Weyl algebra. In \cite{PFLAUM20101958} the authors also work with cyclic complex and they use large N techniques to reduce the computations of Lie algebra cohomology to invariant polynomials. Our method is different since we directly work in periodic cyclic complex and  show that the higher loop contribution in the Feynman diagram is exact with respect to the Connes's operator $B$. The cochain map in \cite[Section 5.2]{WILLWACHER2015277} can be recovered by our universal trace using the descent method. This leads to a simple calculation of the algebraic index and is closer to the physics argument. In particular, we do not need to compute the Lie algebra cohomology or use the large N technique, but in fact be able to compute the index manifestly from our explicit cocycle. An algebraic version of such semi-classical analysis of the periodic cyclic homology is implicitly formulated in \cite{nest1995} to compute the algebraic index. There is a closely related construction \cite{CATTANEO20111966,Moshayedi:2019rnu} using 2d Poisson sigma model, where the Gauss-Manin connection is used to obtain algebraic index on Poisson manifold. Our localized effective theory could be applied to many other topological models. For example, it would be interesting to combine techniques in \cite{CATTANEO20111966,Moshayedi:2019rnu} and our exact semi-classical method to construct topological B-models coupled with gravity, which are closely related \cite{Alexandrov_1997} to Poisson sigma models.

\noindent{\bf Acknowledgements.}
We would like to thank Vladimir Baranovsky, Owen Gwilliam, Weiqiang He, Qin Li, Yifan Li,  Ryszard Nest, Xiang Tang, Mich\`{e}le Vergne, Brian Williams, Edward Witten, Jie Zhou for helpful discussions. We are specially grateful to Mich\`{e}le Vergne for providing  numerous valuable suggestions on the earlier version of this manuscript. S.L. is partially supported by grant 11801300 of National Natural Science Foundation of China  and grant Z180003 of Beijing Municipal Natural Science Foundation.  Part of this work was done in Fall 2019 while S.L. was visiting Institute for Advanced Study at Princeton and Z.G. was visiting Center of Mathematical Sciences and Applications at Harvard. We thank for their hospitality and provision of excellent working enviroment.

\noindent \textbf{Conventions}. In this paper, we will work with various versions of cyclic tensors. We set up our notations here, which will be used throughout this paper.

Let $A$ be a $\mathbb{C}$-algebra with unit, not necessarily commutative. Let $\overline{A}:=A/\mathbb{C}\cdot 1$. Let
$$
C_{-p}(A):=A\otimes \overline{A}^{\otimes {p}}
$$
be the cyclic $p$-chains. The Hochschild differential
$$
b:C_{-p}(A)\rightarrow C_{-p+1}(A), p\geq 1
$$
and Connes operator
$$
B:C_{-p}(A)\rightarrow C_{-p-1}(A)
$$
are defined by
\begin{align*}
b(a_0\otimes\cdot\cdot\cdot\otimes a_p)&=(-1)^pa_pa_0\otimes\cdot\cdot\cdot\otimes a_{p-1}+a_0a_1\otimes\cdot\cdot\cdot\otimes a_{p}\\
&+\sum_{i=1}^{p-1}(-1)^ia_0\otimes\cdot\cdot\cdot \otimes a_i a_{i+1}\otimes\cdot\cdot\cdot \otimes a_p.
\end{align*}

\begin{align*}
B(a_0\otimes\cdot\cdot\cdot\otimes a_p)&=1\otimes a_0\otimes\cdot\cdot\cdot \otimes \cdot\cdot\cdot \otimes a_{p}+\sum_{i=1}^{p}(-1)^{pi}1\otimes a_i\otimes\cdot\cdot\cdot \otimes a_p\otimes a_0\otimes\cdot\cdot\cdot \otimes a_{i-1}.
\end{align*}

The periodic cyclic complex is defined by
$$
CC^{per}_{-\bullet}(A):=(C_{-\bullet}(A)[u,u^{-1}],b+uB)
$$
where $u$ is a formal variable of degree $2$, and will be called the cyclic parameter.

\section{Formal geometry and algebraic index}\label{section:formal}
\subsection{Harish-Chandra pair}\label{sec:HC}\label{desc}
In this section we review some basics on Gelfand-Kazhdan formal geometry, which can be viewed as a universal treatment of differential geometric manipulations in terms of symmetry. Our presentation  follows \cite{gorbounov2016chiral,gorokhovsky2017equivariant}.

For our purpose, this will reduce the problem of establishing algebraic index to universal Lie algebra cohomology computations. In our later topological quantum mechanics interpretation, the Gelfand-Kazhdan formal geometry captures precisely the gauge consistency (under target coordinate transformations) of low energy effective theory on the formal loop space.

All Lie algebras and Lie groups will be defined over $\mathbb{C}$. For $G$ a Lie group, we use $\text{Lie}(G)$ to denote its associated Lie algebra.
	
	\newcommand{\gog}{\mathfrak{g}}
	
	\begin{defn}
	A Harish-Chandra pair is a pair $(\gog,K)$ where $\gog$ is a Lie algebra and $K$ is a Lie group together with
\begin{itemize}
\item an action $\rho_K:K\longrightarrow \text{Aut}(\gog)$ of $K$ on $\gog$  and
\item an injective Lie algebra map $i:\text{Lie}(K)\longrightarrow \gog$
\end{itemize}
such that the action of $\text{Lie}(K)$ on $\gog$ induced by $\rho_K$\[
	\text{Lie}(\rho_K):\text{Lie}(K)\longrightarrow \text{Der}(\gog)
	\]
	is the adjoint action induced by the embedding $i:\text{Lie}(K)\longrightarrow\gog.$
	\end{defn}
	 We have the notion of $(\mathfrak{g},K)-$modules.
\begin{defn}
  A $(\mathfrak{g},K)-$module is a vector space $V$ together with
\begin{itemize}
\item a Lie algebra map $\rho_{\mathfrak{g}}:\mathfrak{g}\rightarrow \text{End}(V)$ and
\item a Lie group map $\rho_K:K\rightarrow \text{GL}(V)$
 \end{itemize}
  such that the composition
  $\text{Lie}(K)\xrightarrow{i}\mathfrak{g}\xrightarrow{\rho_{\mathfrak{g}}}\text{End}(V)$
  equals $\text{Lie}(\rho_K).$
\end{defn}

Notions of morphisms between Harish-Chandra pairs and Harish-Chandra modules are defined in the obvious way, forming into natural categories.
	
	\begin{defn}
	A flat $(\gog,K)$-principal bundle (or just flat $(\gog,K)$-bundle) over X is\bi
	\item a principal $K$ bundle $\pi: P\longrightarrow X$
	\item a $K$-equivariant $\gog$-valued 1-form $\gamma \in\Omega^1(P,\gog)$ on $P$
	\ei
	such that\bi
	\item for all $a\in \Lie(K)$ we have $\gamma(\xi_a)=a$ where $\xi_a$ denotes the vector field on $P$ generated by $a$
	\item $\gamma$ satisfies the Maurer-Cartan equation \[
	d\gamma+\frac{1}{2}[\gamma,\gamma]=0
	\]
	where $d$ is the de Rham differential on $P$ and the bracket is taken in $\gog$.
	\ei
	\end{defn}
\begin{rem}\label{rmk-projective}
If instead, we replace the last equation by requiring $\gamma$ to satisfy
	$$
	d\gamma+\frac{1}{2}[\gamma,\gamma]=\pi^*\omega
	$$
	where $\omega$ is a $Z(\g)$-valued 2-form on $X$, then we say it is projective flat. Here $Z(\g)$ is the center of $\g$. Let $H$ be a subgroup of the center of $K$ such that $\Lie(H)=Z(\g)\cap \Lie(K)$, then we can form a  \HC pair $(\g/Z(\g), K/H)$. $P/H$ gives rise to a principal $K/H$-bundle,  and $\gamma$ induces a flat $(\g/Z(\g), K/H)$ -principal bundle structure on $P/H$.

\end{rem}
	
Given a flat $(\mathfrak{g},K)$-bundle $P\rightarrow X$ and a $(\mathfrak{g},K)$-module $V$, let
$$
\Omega^\bullet(P,{V}):=\Omega^\*(P)\otimes V
$$
denote differential forms on $P$ valued in $V$. It carries a differential which gives a flat connection by
$$\nabla^{\gamma}:=d+\rho_{\mathfrak{g}}(\gamma):\Omega^\*(P;{V})\rightarrow \Omega^{\*+1}(P;{V}).
$$
The group $K$ acts on $\Omega^\*(P)$ and $V$, hence inducing a natural action on $\Omega^\*(P, {V})$.

\begin{defn}
We define basic forms $\Omega^\*(P; {V})_{bas}\subset \Omega^\*(P, {V})$ by
$$
\Omega^\*(P;{V})_{bas}=\fbracket{\left.  \alpha \in  \Omega^\*(P, {V})   \right | \iota_{\xi_a} \alpha=0,  \forall a\in \Lie(K) \quad  \text{and}\quad  \text{$\alpha$ is $K$-invariant}.}
$$
Here $\iota_{\xi_a}$ is the contraction with the vector field $\xi_a$ generated by $a$.
\end{defn}
 It is standard to see that
$$
\nabla^{\gamma}: \Omega^\*(P; {V})_{bas}\to \Omega^{\*+1}(P; {V})_{bas}
$$
preserves basic forms. Therefore we obtain a cochain complex $\bracket{\Omega^\bullet(P;{V})_{bas}, \nabla^{\gamma}}$.

Let
$$
{V}_P:= P\times_{K} V
$$
be the vector bundle on $X$ associated to the $K$-representation $V$. Let
$
  \Omega^\*(X;{V_P})
$
be differential forms on $X$ valued in the bundle ${V_P}$. There is a canonical identification
$$
   \Omega^\*(P; {V})_{bas}\simeq   \Omega^\*(X;{V_P}).
$$
The flat connection $\nabla^{\gamma}$ induces a flat connection on ${V_P}$, still denoted by $\nabla^{\gamma}$. Equivalently, ${V_P}$ carries a structure of $D_X$-module. Its de Rham cohomology is denoted by
$
  H^\bullet(X; V_P).
$

Next we discuss how to descent Lie algebra cohomologies to geometric objects on $X$.

\begin{defn} Let $(\g, K)$ be a \HC pair, and $V$ be a \HC module.  The $(\g, K)$ relative Lie algebra cochain complex $(C^\*_{\Lie}(\mathfrak{g},K;V),\partial_{\Lie})$ is defined by
$$
C^p_{\Lie}(\mathfrak{g},K;V)=\Hom_K(\wedge^p(\g/\Lie(K)),V).
$$
Here $\Hom_K$ means $K$-equivariant linear maps.  $\pa_{\Lie}$ is the Chevalley-Eilenberg differential if we view $C^p_{\Lie}(\mathfrak{g},K;V)$ as a subspace of the Lie algebra cochains $C^p_{\Lie}(\mathfrak{g};V)$. Explicitly, for $\alpha \in C^p_{\Lie}(\mathfrak{g},K;V)$,
\begin{align*}
(\partial_{\Lie}\alpha)(a_1\wedge\cdots\wedge a_{n+1}):=\sum_{i=1}^{n+1}(-1)^{i-1}\rho_{\g}(a_i) \alpha(a_1\wedge\cdots\hat{a_i}\cdots\wedge a_{n+1})\\
+\sum_{i<j}(-1)^{i+j}\alpha([a_i,a_j]\wedge \cdots \hat{a_i}\cdots\hat{a_j}\cdots\wedge a_{n+1}).
\end{align*}
$\pa_{\Lie}$ preserves the $(\g, K)$-cochains, and we have a subcomplex
$$
   \bracket{C^\*_{\Lie}(\mathfrak{g},K;V),\partial_{\Lie}}\into   \bracket{C^\*_{\Lie}(\mathfrak{g};V),\partial_{\Lie}}.
$$
\end{defn}

\begin{rem}
	Note that we have an analogue definition of relative Lie algebra cochain $(C^\*_{\Lie}(\mathfrak{g},\mathfrak{k};V),\partial_{\Lie})$ for a Lie algebra $\mathfrak{g}$ and a sub-algebra $\mathfrak{k}$: we simply replace $K$-equivariance in the definition above by $\mathfrak{k}$. When $K$ is a connected, we can identify $C^\*_{\Lie}(\mathfrak{g},K;V)$ with $C^\*_{\Lie}(\mathfrak{g},\Lie(K);V)$.
\end{rem}

Note that an element of $C^p_{\Lie}(\mathfrak{g},K;V)$ is the same as an element in $C^p_{\Lie}(\mathfrak{g};V)$ satisfying two conditions: 1) contraction with a vector in $\Lie(K)$ is zero; 2) $K$-equivariance. It is analogous to the definition of basic forms. This is precisely described  by the following descent construction.

\begin{defn}We define the descent map  from the $(\g,K)$ relative Lie algebra cochain complex to $V$-valued de Rham complex on $P$ by
 \begin{align*}
\desc: (C^\*_{\Lie}(\mathfrak{g},K;V),\partial_{\Lie})&\rightarrow (\Omega^\*(P;{V})_{bas},\nabla^{\gamma})\\
	\alpha &\to \alpha(\gamma,...,\gamma).
\end{align*}	
Here if view $\alpha$ as an element in $\Hom(\wedge^k \g, V)$ and multi-linearly extend it over $\Omega^\*(P)$, then $\alpha(\gamma,...,\gamma)$ produces a $V$-valued $k$-form on $P$.
\end{defn}

 It is straight-forward to check $\desc$ is compatible with the differential and its image lies in basic forms. By abuse of notation, we will also write this map as
 $$
 \desc:(C^\*_{\Lie}(\mathfrak{g},K;V),\partial_{\Lie}) \rightarrow (\Omega^\*(X;{V_P}),\nabla^{\gamma}).
 $$
 Passing to cohomology, we obtain a descent map (again denoted by $\desc$)
 $$
  \desc: H^\*_{\Lie}(\mathfrak{g},K;V) \rightarrow H^\*(X;{V_P}).
 $$

\subsection{Fedosov connection}\label{sec:Fedosov}
We review Fedosov's geometric approach to deformation quantization. We explain Fedosov's flat connection on quantized algebra as a projective flat \HC bundle. This will play a fundamental role in the gluing of topological quantum mechanical models.

\subsubsection{Weyl Algebra}\label{Weyl}

\begin{defn}
The polynomial Weyl algebra $\mathcal{W}^{pol}_{2n}$ over $k=\mathbb{C}[\hbar,\hbar^{-1}]$  is the space of polynomials $k[p^1,...,p^n,q^1,...,q^n]$ in $2n$ variables with the Moyal product associated to the bivector
$$\hat\Pi=\frac{1}{2}\sum_{i=1}^{n}(\frac{\partial}{\partial p^i}\otimes\frac{\partial}{\partial q^i}-\frac{\partial}{\partial q^i}\otimes \frac{\partial}{\partial p^i})\in \End_{k}(\mathcal{W}^{pol}_{2n}\otimes\mathcal{W}^{pol}_{2n}).$$
The Moyal product is given by the formula
$$f\star g=m(e^{\hbar\hat\Pi}(f\otimes g)),$$
where $m(f\otimes g)=fg$ is the standard commutative product on polynomials. The Weyl algebra is $\mathbb{Z}$-graded for the assignment of weights
$$\wt(q^i)=\wt(p^i)=1,\quad \wt(\hbar)=2.$$
\end{defn}
\begin{defn}
We define the completed Weyl algebras $\mathcal{W}_{2n}$ and $\mathcal{W}^+_{2n}$  by
$$
\mathcal{W}_{2n}:=\mathbb{C}[[p^i,q^i]]\LS{\hbar}, \quad \mathcal{W}^+_{2n}=\mathbb{C}[[p^i,q^i]]\PS{\hbar}.
$$
An element of $\mathcal{W}_{2n}$ is represented by a formal sum $\alpha=\sum_{k\in \Z}\alpha_k \hbar^k$ where $\alpha_k \in \mathbb{C}[[p^i,q^i]]$ and vanishes for $k \ll 0$. $\alpha$ lies in $\mathcal{W}^+_{2n}$ if $\alpha_k=0$ for all $k<0$. The Moyal product $\star$ is well-defined on the completed Weyl algebras, and $\mathcal{W}^{pol}_{2n}, \mathcal{W}^+_{2n}$ are subalgebras of $\mathcal{W}_{2n}$.  We use the same $\Z$-grading of weights on $\mathcal{W}^{pol}_{2n}$ when talking about homogenous elements.
\end{defn}

$\mathcal{W}_{2n}$ has an induced Lie algebra structure with Lie bracket defined by
$$[f,g]:=\frac{1}{\hbar}[f,g]_{\star}:=\frac{1}{\hbar}(f\star g-g\star f).$$
Let $\Sp_{2n}$ be the symplectic group of linear transformations preserving the tensor $\Pi$. It acts on Weyl algebras by inner automorphisms. The Lie algebra $\mathfrak{sp}_{2n}$ of $\Sp_{2n}$ can be identified with quadratic polynomials in $\C[p^1,...,p^n,q^1,...,q^n]$. This defines a natural sequence of embeddings
$$
\mathfrak{sp}_{2n}\into \mathcal{W}^{pol}_{2n}\into \mathcal{W}_{2n}.
$$
It acts on $\mathcal{W}_{2n}$ by the corresponding inner derivations.

\subsubsection{Fedosov's connection and deformation quantization} \label{sec: Fed-conn}
\begin{defn}
Let $(M,\omega)$ be a symplectic manifold and $F_{\Sp}(M)$ be its symplectic frame bundle. The Weyl bundles $\mathcal{W}_M^+, \mathcal{W}_M$ of $M$ are defined to be
$$
\mathcal{W}_M^+:=F_{\Sp}(M)\times_{\Sp_{2n}}\mathcal{W}^+_{2n}, \quad \mathcal{W}_M:=F_{\Sp}(M)\times_{\Sp_{2n}}\mathcal{W}_{2n}.
$$
They are both bundles of algebras, since the Moyal product is $\Sp_{2n}$-equivariant. We denote Weyl-algebra valued differential forms by
$$
\Omega^\bullet_M(\mathcal{W}^+):=\Omega^\bullet_M\otimes_{C^{\infty}(M)}\Gamma(M,\mathcal{W}_M^+)=\Gamma(M, \wedge^\bullet T_M^\vee\otimes \mc W_M^+).
$$
Similarly for $\Omega^\bullet_M(\mathcal{W})$. We can also identify $\Omega^\bullet_M(\mathcal{W})=\Omega^\bullet_M(\mathcal{W}^+)[\hbar^{-1}]$.
\end{defn}

\begin{defn}
  A \textit{symplectic connection} on $M$ is a torsion free connection $\nabla$ on the tangent bundle $T_M$ that is compatible with the symplectic form: $\nabla \omega=0$.
\end{defn}

Symplectic connections exist on any symplectic manifold and are not unique \cite{Fedosov:1996fu}. Such a connection $\nabla$ induces a connection $\nabla^{\mc W}$ on $\mathcal{W}_M$. Let
$$
R_{\nabla}=\nabla^2\in \Omega^2_M(\End(T_M))
$$
denote the curvature of $\nabla$. If we identify
$
T_M\iso F_{\Sp}(M)\times_{\Sp_{2n}}\R^{2n}
$, then
$$
R_{\nabla}\in \Omega^2_M(F_{\Sp}(M)\times_{\Sp_{2n}}\mathfrak{sp}_{2n}).
$$
Under the embedding $\mathfrak{sp}_{2n}\into  \mathcal{W}^+_{2n}$, this curvature form is mapped to an element
$$
R_\nabla^{\mc W}\in \Omega^2_M(\mathcal{W}^+).
$$
\begin{lem}The curvature of $\nabla^{\mc W}$ on the Weyl bundle $\mc W_M$ is given by
$$
(\nabla^{\mc W})^2= {1\over \hbar}[R_\nabla^{\mc W},-]_{\star}.
$$

\end{lem}

Given a symplectic connection $\nabla$ and any sequences $\{\omega_k\}_{k\geq 1}$ of closed 2-forms on $M$, Fedosov\cite{fedosov1994} proved that there exists a unique element $\gamma\in \Omega^1_M(\mathcal{W})$ with certain prescribed gauge fixing condition such that the following equation holds
  $$\nabla^{\mc W}\gamma+\frac{1}{2\hbar}[\gamma,\gamma]_{\star}+R_{\nabla}^{\mc W}=\omega_{\hbar}.$$
 Here $\omega_{\hbar}=-\omega+\sum_{k\geq 1}\hbar^k\omega_k$ is the characteristic class of the deformation quantization (the original proof in \cite{fedosov1994} is for $\omega_{\hbar}=-\omega$ but the same method works in this generality). Define the modified connection on $\mc W_M$ by $D=\nabla^{\mc W}+\frac{1}{\hbar}[\gamma,-]_{\star}$. Then
$$
D^2={1\over \hbar}[\omega_{\hbar},-]_\star=0
$$
since $\omega_{\hbar}$ is an central element in $\Omega_M^\bullet(\mc W)$. $D$ is Fedosov's flat connection on the Weyl bundle. Moreover, if we consider flat sections
$$
\mathcal{W}^+_D(M):=\fbracket{s\in \Gamma(M, \mc W^+_M)| D s=0},
$$
then $\mathcal{W}^+_D(M)$ is a subalgebra of $\Gamma(M, \mc W^+_M)$ which defines a deformation quantization on $(M,\omega)$ (under a natural isomorphism $\mathcal{W}^+_D(M)\iso C^\infty(M)\PS{\hbar}$ via the symbol map \cite{fedosov1994}, which means the map by setting $p^i=q^i=0$ in the fiber $\mathcal{W}^+_{2n}$).

Fedosov connection can be generalized to the bundle case \cite{fedosov-bundle}. Let $E$ be a vector bundle equipped with a connection $\nabla^E$. We consider the following bundle of algebras
$$
  \mc W^+_M \otimes \End(E)
$$
which has an induced connection from $\nabla^{\mc W}$ and $\nabla^E$. Then similarly we can find
$$
\gamma_E\in \Omega^1_M(  \mc W^+_M \otimes \End(E))
$$
which satisfies Fedosov's equation:
\begin{equation}\label{FEDOSOV}
   \nabla^{\mc W}\gamma_E+\frac{1}{2\hbar}[\gamma_E,\gamma_E]_{\star}+R_{\nabla}^{\mc W}=\omega_{\hbar}. \tag{F}
\end{equation}
It follows that the following modified connection on $\mc W^+_M \otimes \End(E)$ is flat
$$
  \nabla^{\mc W}\otimes 1+ 1\otimes \nabla^E+{1\over \hbar}[\gamma_E,-]_\star.
$$
The flat sections give a quantum deformation of the algebra $\Gamma(M, \End(E))$.

\subsection{Gelfand-Fuks map}\label{sec: GF}
\subsubsection{Fedosov connection and Gelfand-Fuks map}

Set $\mathfrak{g}=\mathcal{W}^+_{2n}$, regarded as a Lie algebra with bracket $[-,-]=\frac{1}{\hbar}[-,-]_{\star}$, and we have $\mathfrak{sp}_{2n}\subset \mathfrak{g}$. Then $(\mathfrak{g},\text{Sp}_{2n})$ forms a \HC pair.


The symplectic connection $\nabla$ can be equivalently described by a connection 1-form
$$
A\in \Omega^1(F_{\Sp}(M), \sp_{2n})
$$
on the principal $\Sp_{2n}$-bundle $F_{\Sp}(M)$. We view $A$ as an element
$$
A \in \Omega^1(F_{\Sp}(M),\g)
$$
under the natural embedding $\sp_{2n}\into\mathcal{W}^+_{2n}$.  The correction term $\gamma$ in Fedosov's connection can be viewed as a $\Sp_{2n}$-invariant element of
$$
\gamma \in \Omega^1(F_{\Sp}(M), \mathcal{W}^+_{2n}).
$$

\begin{prop}\label{prop-flat} $A+\gamma$ defines a projective-flat $(\g, \Sp_{2n})$-principal bundle structure on $F_{Sp}(M)$.
\end{prop}
\begin{proof}Fedosov's equation is equivalent to
$$
  d(A+\gamma)+{1\over 2}[A+\gamma, A+\gamma]=\omega_\hbar.
$$
This says that $A+\gamma$ is a projective flat connection.
\end{proof}

Now we apply the descent construction in Section \ref{sec:HC} to the pair $(\mathfrak{g}',K):=(\mathfrak{g}/Z(\mathfrak{g}),\Sp_{2n})$, where $Z(\mathfrak{g})$ is the center of $\mathfrak{g}$. In our case $\mathfrak{g}=\mathcal{W}^+_{2n}$, $Z(\g)=\C\PS{\hbar}$ and $Z(\g)\cap \sp_{2n}=0$.

There is a natural isomorphism
$$
C^\bullet_{\Lie}(\mathfrak{g}', \Sp_{2n};\mathbb{C}\LS{\hbar})\iso C^\bullet_{\Lie}(\mathfrak{g},\sp_{2n}\oplus Z(\mathfrak{g});\mathbb{C}\LS{\hbar}).
$$

We get a map of cochain complexes  via descent construction

\begin{align*}
\desc:(C^\*_{\Lie}(\mathfrak{g}, \sp_{2n}\oplus \,Z(\mathfrak{g});\mathbb{C}\LS{\hbar}),\partial_{\Lie})&\rightarrow(\Omega^\*_M\LS{\hbar},d)\\
\alpha\quad \quad &\mapsto \alpha(\gamma, \dots, \gamma).
\end{align*}
Note we do not insert $A$ since it lies in $\sp_{2n}$. This is the Gelfand-Fuks map discussed in \cite{feigin2005hochschild,Pflaum2007An,gorokhovsky2017equivariant,feigin1989riemann}.
\subsubsection{Twist by vector bundle}

The Gelfand-Fuks map can be extended to the bundle case as well. Let $E$ be a vector bundle equipped with a connection $\nabla^E$. Let
$$
\gamma_E\in \Omega^1_M(  \mc W^+_M \otimes \End(E))
$$
such that   $\nabla^{\mc W}\otimes 1+ 1\otimes \nabla^E+{1\over \hbar}[\gamma_E,-]_\star$ defines a flat Fedosov connection on the bundle of algebras $
  \mc W^+_M \otimes \End(E)
$. Let $Fr(E)$ be the frame bundle of $E$, which is a principal $\GL_r$-bundle on $X$ and
$$
 E\iso Fr(E)\times_{\GL_r} \C^r.
$$
We can view $\nabla^E$ as coming from a connection on the principal $\GL_r$-bundle $Fr(E)$. Then the symplectic connection together with $\nabla^E$ gives a connection 1-form
$$
A\in \Omega^1(F_{\Sp}(M)\times_M Fr(E), \sp_{2n}\oplus \gl_r)
$$
on the principal $\Sp_{2n}\times \GL_r$-bundle $F_{\Sp}(M)\times_M Fr(E)$.

We consider the following \HC pair $(\g,K)$. Let $\gl_r(\mathcal{W}_{2n})$ be the Lie algebra of $\mathcal{W}_{2n}$-valued $r\times r$ matrices with bracket $[x,y]:=\frac{1}{\hbar}(x\star y-y\star x)$. Here $\star$ is the matrix multiplication using the Moyal product. Define $\g$ to be the Lie subalgebra spanned by
$$
\mathfrak{g}:=\mathcal{W}^+_{2n}\cdot\text{Id}+\hbar\gl_r(\mathcal{W}^+_{2n}).
$$
$\g$ has a Lie subalgebra $\sp_{2n}\oplus \hbar \gl_r$, which is isomorphic to the direct sum $\sp_{2n}\oplus \gl_r$ under the usual bracket. This gives an embedding $\sp_{2n}\oplus \gl_r\into \g$ that allows us to view
$$
A\in \Omega^1(F_{\Sp}(M)\times_M Fr(E), \g).
$$
Again, $\gamma_E$ can be viewed as a $\Sp_{2n}\times \GL_r$-invariant element
$$
\gamma_E\in \Omega^1(F_{\Sp}(M)\times_M Fr(E), \g).
$$

The analogue of Proposition \ref{prop-flat} is the following
\begin{prop}$A+\gamma_E$ defines a projective-flat $(\g, \Sp_{2n}\times \GL_r)$-bundle structure on $F_{Sp}(M)\times_{M} Fr(E)$.

\end{prop}

In our case, the center $Z(\mathfrak{g})=\C\PS{\hbar}\Id\iso \C\PS{\hbar}$. By Remark \ref{rmk-projective}, we can form the \HC pair $(\g/Z(\g), \PGL_r)$ and $A+\gamma_E$ induces a flat $(\g/Z(\g), \PGL_r)$-principal bundle structure on $F_{\Sp}(M)\times_M PFr(E)$. Here $PFr(E)=Fr(E)/\C^*$ is the projective frame bundle.

Note that there is a natural isomorphism
$$
C^\bullet_{Lie}(\mathfrak{g}, \sp_{2n}+ \hbar\gl_r+Z(\mathfrak{g});\mathbb{C}\LS{\hbar})\iso C^\bullet_{Lie}(\mathfrak{g}/Z(\g),\Sp_{2n}\times \text{PGL}_r ;\mathbb{C}\LS{\hbar}).
$$

As a corollary, we obtain the Gelfand-Fuks map of cochain complexes by descent
\begin{align*}
\desc:
\bracket{C^\bullet_{Lie}(\mathfrak{g}, \sp_{2n}+ \hbar\gl_r+Z(\mathfrak{g});\mathbb{C}\LS{\hbar}),\partial_{Lie}}&\rightarrow \bracket{\Omega_M^\bullet\LS{\hbar},d}\\
\alpha\quad \quad &\to \alpha(\gamma, \dots, \gamma).
\end{align*}

\subsubsection{Characteristic classes}\label{sec:Char}
We review the Chern-Weil construction of characteristic classes in Lie algebra cohomology.  They will descent to the
usual characteristic forms via the Gelfand-Fuks map.

Let us first recall the construction of curvature form in Lie algebra cohomology. Let $\g$ be a Lie algebra and $\mathfrak{h}$ be a Lie subalgebra. Suppose there exists an $\mathfrak{h}$-equivariant splitting
$$
\pr: \g\to \mathfrak{h}
$$
of the embedding $\mathfrak{h}\into \g$. In general $\pr$ is not a Lie algebra homomorphism from $\g$ to $\mathfrak{h}$. The failure defines a linear map
$
R\in \Hom(\wedge^2\g,\mathfrak{h})
$
by
$$
R(\alpha,\beta):= [\pr(\alpha),\pr(\beta)]_{\mathfrak{h}}-\pr[\alpha,\beta]_{\g}, \quad \alpha, \beta\in \g.
$$
$\mathfrak{h}$-equivariance implies that $R(\alpha,-)=0$ if $\alpha\in \mathfrak{h}$ and $R$ is $\mathfrak{h}$ equivariant. $R$ is called the curvature form. Let $\Sym^m(\mathfrak{h}^\vee)^{\mathfrak{h}}$ be invariant polynomials on $\mathfrak{h}$ of degree $m$. Given $P\in \Sym^m(\mathfrak{h}^\vee)^{\mathfrak{h}}$, we can associate a cochain $P(R)$ in $C^{2m}(\g, \mathfrak{h};\C)$ given by the composition
$$
 P(R):  \wedge^{2m}\g\stackrel{\wedge^m R}{\to} \Sym^{m}(\mathfrak{h})\stackrel{P}{\to} \C.
$$
It can be checked that $\pa_{\Lie}P(R)=0$, defining a cohomology class $[P(R)]$ in $H^{2m}(\g,\mathfrak{h};\C)$ which does not depend on the choice of $\pr$. Therefore we have the analogue of Chern-Weil characteristic map
\begin{align*}
\chi: \Sym^\bullet(\mathfrak{h}^\vee)^{\mathfrak{h}}&\to H^{\*}(\g,\mathfrak{h};\C)\\
    P &\to \chi(P):=[P(R)].
\end{align*}

Now we apply the above construction to our situation where
$$
\mathfrak{g}=\mathcal{W}_{2n}^+\cdot \Id+\hbar(\gl_r(\mathcal{W}^+_{2n})), \quad \mathfrak{h}=\sp_{2n}+ \hbar\gl_r+Z(\mathfrak{g}).
$$
Any element in $\mathfrak{g}=\mathcal{W}_{2n}^+\cdot \Id+\hbar(\gl_r(\mathcal{W}^+_{2n}))$ can be uniquely written as
$$
f\cdot \text{Id}+\hbar A, \quad f\in \mathbb{C}[[y^1,...,y^{2n}]], A\in \gl_r(\mathcal{W}^+_{2n}).
$$
Here we write uniformly $y^1=p^1,\dots, y^n=p^n, y^{n+1}=q^1,\dots, y^{2n}=q^n$. We also identify
$$
\mathfrak{h}=\mathfrak{sp}_{2n}\oplus \hbar\gl_r\oplus \mathbb{C}\oplus_{i>1} \hbar^i\mathbb{C}.
$$

 We define the $\mathfrak{h}-$equivariant projection: $\pr:\mathfrak{g}\rightarrow \mathfrak{h}$ as follows
 $$\pr(f\cdot\text{Id}+\hbar A):=\bracket{\frac{1}{2}\partial_i\partial_jf(0)y^i y^j,\hbar A_1(0),f(0),\bigoplus\limits_{i>1}\frac{1}{r}\trace(\hbar^i A_i(0))}.
 $$
 Here we write
  $\hbar A=\hbar A_1+\hbar^2 A_2+...$ . We can write $\pr=\pr_1+\pr_2+\pr_3$ where
  \begin{align*}
  \pr_1(f\cdot\text{Id}+\hbar A)&=\frac{1}{2}\partial_i\partial_jf(0)y^i y^j\in \mathfrak{sp}_{2n}\\\pr_2(f\cdot\text{Id}+\hbar A)&=\hbar A_1(0)\in\hbar\gl_r\\
  \pr_3(f\cdot\text{Id}+\hbar A)&=f(0)+\sum_{i>1}\frac{1}{r}\trace(\hbar^iA_i(0))\in \mathbb{C}\oplus_{i>1} \hbar^i\mathbb{C}
  \end{align*}

  The corresponding curvature is given by
 $$R:=[\pr(-),\pr(-)]-\pr([(-),(-)])\in \Hom(\wedge^2 \g,\mathfrak{h}).
 $$
It can be decomposed into three terms $R=R_1+R_2+R_3,$ where
\begin{align*}
R_1&:=[\pr_1(-),\pr_1(-)]-\pr_1[-,-]&&\in \Hom(\wedge^2 \g, \sp_{2n})\\
R_2&:=[\pr_2(-),\pr_2(-)]-\pr_2[-,-]&&\in \Hom(\wedge^2 \g, \gl_r)\\
R_3&:=-\pr_3[-,-]&&\in \Hom(\wedge^2 \g, \C\oplus \oplus_{i>1}\hbar^i \C).
\end{align*}

\begin{rem} It is worthwhile to point out that all the $\Hom$'s here are only $\C$-linear map, but not $\C\PS{\hbar}$-linear, although $\g$ is a $\C\PS{\hbar}$-module.

\end{rem}

We now define the $\hat A$ genus
$$\hat{A}(\mathfrak{sp}_{2n}):=\bbracket{\det\bracket{\frac{R_1/2}{\sinh(R_1/2)}}^{1/2}}\in H^\bullet(\mathfrak{g},\mathfrak{h};\mathbb{C}) ,$$
and the Chern Character:
$$\Ch(\gl_r):=[\Tr(e^{R_2})]\in H^\bullet(\mathfrak{g},\mathfrak{h};\mathbb{C}).$$

\begin{prop}\label{prop-char-desc}
Under the descent map $\desc: H^\*(\g,\mathfrak{h};\C)\to H^\*(M)\LS{\hbar}$ via the Fedosov connection,
  \begin{align*}
  &\desc(\hat{A}(\mathfrak{sp}_{2n}))=\hat{A}(M)\in H^\bullet(M,\mathbb{C})\\
&  \desc(\Ch(\gl_r))=e^{\omega_1}\Ch(E)\in H^\bullet(M,\mathbb{C})\\
&  \desc(R_3)=\omega_{\hbar}-\hbar \omega_1\in H^\*(M,\mathbb{C})\PS{\hbar}.
  \end{align*}
  Here $\omega_1$ is the $\hbar^1$ term of $\omega_{\hbar}.$ $\hat A(M)$ is the $\hat A$-genus of $M$, and $\Ch(E)$ is the Chern character of the bundle $E$.

\end{prop}
\begin{proof}
  For the first two identities, we write down the first few terms of Fedosov connection $\gamma$ \cite{fedosov-bundle}:
  $$\gamma=\omega_{ij}y^idx^j+\frac{1}{8}R_{ijkl}y^iy^jy^kdx^l+\hbar((R_E)_{ij}y^idx^j+(\omega_1)_{ij}y^idx^j)+...$$
Here we omit the terms with weight bigger than 3.  Those term will not appear in our computation of curvature, the Lie bracket of higher weight term will be projected out by $\pr$.

The descent of $R_1$ is given by
$$
R_1(\gamma,\gamma)=\frac{1}{4}R_{ijkl}y^iy^jdx^k\wedge dx^l.
$$
This is the curvature two form of our symplectic manifold. Applying the invariant $\hat{A}$ polynomial, we get $\hat{A}-$genus of $M$. The descent of $R_2$ is
$$
R_2(\gamma,\gamma)=\hbar(R^E_{kl}dx^k\wedge dx^l+\omega_1).
$$
Applying the invariant polynomial $\Tr(e^{R_2})$, we get $e^{\omega_1}\Ch(E).$

For the last identity we use the Fedosov equation \eqref{FEDOSOV} and apply $\pr_3$ to both side
$$
\pr_3(d(A+\gamma_E)+\frac{1}{2}[A+\gamma,A+\gamma])=\pr_3(\omega_{\hbar}),
$$
we get $\desc(R_3)=-\hbar \omega_1+\omega_{\hbar}$.
\end{proof}

\subsection{Universal Algebraic Index Theorem}\label{normalized}
Now we fix the \HC pair $(\g, K)$
$$
\mathfrak{g}=\mathcal{W}^+_{2n}\cdot \text{Id}+\hbar\gl_r(\mathcal{W}^+_{2n})), \quad K=\Sp_{2n}\times \GL_r.
$$
Let $C_{-\bullet}(\gl_r(\mathcal{W}_{2n}))$ denote the cyclic chains of the matrix algebra $\gl_r(\mathcal{W}_{2n})$.

We can write down a trace map (see Section \ref{trace-map}),
$$\text{Tr}:(CC^{per}_{-\bullet}(\gl_r(\mathcal{W}_{2n})),b+uB)\rightarrow (\mathbb{C}\LS{\hbar}[u,u^{-1}],0),$$
which is basically a composition of trace of matrices and the higher trace for Weyl algebras. This is indeed a cochain map, but not a morphism of $(\mathfrak{g},K)$-modules. The observation is that the $\mathfrak{g}$-action on $(CC^{per}_{-\bullet}(\gl_r(\mathcal{W}_{2n})),b+uB)$ is homotopy trivial. This suggests that $\text{Tr}$ could be lifted to Lie algebra cochains  valued in the $(\mathfrak{g}, K)$-module that is the periodic
cyclic chains. This is indeed the case.
\begin{thm}[see Theorem \ref{TRACE}]There is a canonical cocycle in the Lie algebra cochain complex (of cohomology degree $2n$)
$$\uTr\in C^\*_{\Lie}(\mathfrak{g},\mathfrak{h};\Hom_{\mathbb{K}}(CC_{-\*}^{per}(\gl_r(\mathcal{W}_{2n}),\mathbb{K}))$$
whose restriction to
$$C^0_{\Lie}(\mathfrak{g},\mathfrak{h};\Hom_{\mathbb{K}}(CC_{-\*}^{per}(\gl_r(\mathcal{W}_{2n}),\mathbb{K}))$$
coincides with $\Tr$. Here $\mathbb{K}:=\mathbb{C}\LS{\hbar}[u,u^{-1}]$.
\end{thm}

In the next section, we will give an explicit formula of $\uTr$ in terms of a topological quantum mechanical model. We will call $\uTr$ the universal trace map. When the bundle is not at present so $\g=\mathcal{W}_{2n}^+$, this becomes the universal trace as studied in \cite{nest1995,feigin2005hochschild}.

One application of our topological quantum mechanical interpretation of $\uTr$ is that we can implement the analogue of exact semi-classical approximation to give a simple geometric proof of the following universal algebraic index theorem \cite{feigin1989riemann, feigin2005hochschild}.

\begin{thm}[See Theorem \ref{index thm}]\label{thm-index}  Let $\uTr$ be the universal trace map, which induces a cohomology element
$$
 \uTr\in H^\*_{\Lie}(\mathfrak{g},\mathfrak{h};\Hom_{\mathbb{K}}(CC_{-\*}^{per}(\gl_r(\mathcal{W}_{2n}),\mathbb{K})),\quad \mathbb{K}:=\mathbb{C}\LS{\hbar}[u,u^{-1}].
 $$
Then we have the following index formula
$$
 \uTr(1)= u^ne^{-R_3/u\hbar}\hat{A}(\sp_{2n})_u\cdot \text{Ch}(\GL_r)_u\in H^\*_{Lie}(\mathfrak{g},\mathfrak{h};\mathbb{K}).$$
 Here for a cochain of even degrees $A=\sum_{p\ even}A_p,  A_p\in H^p(\g,\mathfrak{h};\mathbb{K})$, we denote
 $
 A_u:=\sum_{p} u^{-p/2}A_p.
 $
\end{thm}

Now we explain why this formula leads to the usual algebraic index formula for vector bundles on manifolds, by descent construction. We explain the version in \cite{fedosov-bundle}.

The quantum observables $W_D$ is ($D$ is the Fedosov connection)
$$
W_D=\fbracket{s\in \Gamma(M,   \mc W^+_M \otimes \End(E))|  Ds=0 }.
$$
We apply the Gelfand-Fuks descent to the universal trace map:
$$\desc(\uTr(-))\in \Omega^\*(M,E^{per}).$$
Here
$$E^{per}:=(F_{\Sp}(M)\times_M Fr(E))\times_{(\Sp_{2n}\times \GL_r)}\Hom_{\mathbb{K}}(CC_{-\*}^{per}(\gl_r(\mathcal{W}_{2n})),\mathbb{K}).$$
$\uTr$ is $D+b+uB$ closed by construction. Restrict to flat sections of $D$, we get a cochain map
$$\desc(\uTr):(CC_{-\*}^{per}(W_D),b+uB)\rightarrow (\Omega^\*(M)\LS{\hbar}[u,u^{-1}],d_{dR})$$
of degree $2n$. We can apply it to $f\in W_D$ and integrate over $M$, to get a trace
$$
  \Tr(f)= \int_M \desc(\uTr(f))\in \C\LS{\hbar}.
 $$
The value does not depend on $u$, by the degree reason. Using the explicit formula of $\uTr$ in Theorem \ref{TRACE}, we can see that it satisfies the normalized property
   $$ \Tr(f)=\frac{(-1)^n}{\hbar^n}\bracket{\int_M\text{tr}(f)\frac{\omega^n}{n!}+O(\hbar)},\forall f\in \Gamma(M,\End(E)).$$
Therefore
 $$\Tr:W_D\rightarrow \mathbb{C}\LS{h}$$
is indeed the canonical trace in \cite{fedosov-bundle,nest1995}, and the universal formula implies the geometric one by descent. See Section \ref{sec: GM} for further details.

\section{Topological Quantum Mechanics and the Algebraic Index}\label{section:BV}
Throughout this section we fix the notation
$$\mathfrak{g}:=\mathcal{W}^+_{2n}\cdot\text{Id}+\hbar \gl_r(\mathcal{W}_{2n}^+), \quad \mathfrak{h}:=\mathfrak{sp}_{2n}\oplus\hbar\gl_r\oplus\mathbb{C}\oplus_{i>1}\hbar^i\mathbb{C},$$
$$\mathbb{K}:=\mathbb{C}\LS{\hbar}[u,u^{-1}].$$
The main goal of this section is to give a natural and explicit construction of the universal trace
$$
\uTr\in C_{\Lie}^\*(\mathfrak{g},\mathfrak{h};\Hom_{\mathbb{K}}(CC_{-\*}^{per}(\gl_r(\mathcal{W}_{2n})),\mathbb{K})).
$$
This leads to a simple geometric proof of the universal algebraic index theorem (Theorem \ref{thm-index}) by a one-loop Feynman diagram computation in terms of a version of exact semi-classical approximation.

\subsection{Some Notations} We set-up some basic notations in this section and describe a simple BV algebra that will be used to construct the universal trace. We will write
$$
y^1=p^1,\dots, y^n=p^n, y^{n+1}=q^1,\dots, y^{2n}=q^n
$$
so that $
  \C[p^i,q^i]=\C[y^j].
$
The Poisson tensor is expressed in terms of the skew-symmetric matrix $\hat\omega^{ij}$
$$
\Pi=\frac{1}{2}\sum_{i=1}^{n}\bracket{\frac{\partial}{\partial {p^i}}\otimes\frac{\partial}{\partial {q^i}}-\frac{\partial}{\partial {q^i}}\otimes\frac{\partial}{\partial {p^i}}}=\frac{1}{2}\sum_{i,j=1}^{2n}\hat\omega^{ij}\partial_{y^i}\otimes \partial_{y^j}.
$$


\begin{defn}\label{de Rham}
We denote $\hOmega^{-\*}_{2n}$ to be the following formal de Rham algebra
$$\hOmega^{-\*}_{{2n}}:=\C\PS{y^i,dy^i}\LS{\hbar}.
$$
Here the 1-forms $dy^i$ have cohomology degree $-1$ and anticommute $dy^i\cdot dy^j=-dy^j\cdot dy^i$. We have put a $\hbar$ coefficient for convenience. With respect to the natural decomposition
$$\hOmega^{-\*}_{{2n}}=\bigoplus\limits_p \hOmega^{-p}_{{2n}},
$$
our degree assignment is such that p-forms $\hOmega^{-p}_{{2n}}$ sits in degree $-p$.
\end{defn}

Let $\d$ be the de Rham differential of degree -1
$$\d:\hOmega^{-\*}_{{2n}}\rightarrow \hOmega^{-\*-1}_{{2n}}.$$
The constant bivector $\Pi$ defines a map of degree 2 via the contraction
$$\iota_{\Pi}:\hOmega^{-\*}_{{2n}}\rightarrow \hOmega^{-\*+2}_{{2n}}.$$

  We define the BV operator $\Delta$ by the Lie derivative with respect to $\Pi$
  $$\Delta:=\mathcal{L}_{\Pi}=[\d,\iota_{\Pi}]:\hOmega^{-\*}_{{2n}}\rightarrow \hOmega^{-\*+1}_{{2n}}.$$
In terms of $y^i$,
  $$\iota_{\Pi}=\frac{1}{2}\hat\omega^{ij}\iota_{\partial_{y^i}}\iota_{\partial_{y^j}},   \quad \Delta=\hat\omega^{ij}\mathcal{L}_{\partial_{y^i}}\iota_{\partial_{y^j}}.$$


We denote the k-th tensor product of $\hOmega^{-\*}_{2n}$ by
  $$(\hOmega^{-\*}_{{2n}})^{\otimes k}:=\hOmega^{-\*}_{{2n}}\otimes_{\C\LS{\hbar}}\cdot\cdot\cdot\otimes_{\mathbb{C}\LS{\hbar}}\hOmega^{-\*}_{{2n}}.$$
We extend the definition of $\d,\iota_{\Pi},\Delta$ to $(\hOmega^{-\*}_{{2n}})^{\otimes k}$ by declaring
   \begin{align*}\d  (a_1\otimes\cdot\cdot\cdot\otimes a_k)&:=\sum_{1\leq \alpha\leq k}\pm a_1\otimes\cdot\cdot\cdot\otimes \d a_{\alpha}\otimes\cdot\cdot\cdot\otimes a_k\\
    \iota_{\Pi}(a_1\otimes\cdot\cdot\cdot\otimes a_k)&:=\frac{1}{2}\sum_{1\leq\alpha,\beta\leq k}\pm \hat\omega^{ij}a_1\otimes\cdot\cdot\cdot\otimes\iota_{\partial_{y^i}}a_{\alpha}\otimes\cdot\cdot\cdot \otimes \iota_{\partial_{y^j}}a_{\beta}\otimes\cdot\cdot\cdot\otimes a_k \\
    \Delta(a_1\otimes\cdot\cdot\cdot\otimes a_k)&:=\sum_{1\leq\alpha,\beta\leq k}\pm\hat\omega^{ij}a_1\otimes\cdot\cdot\cdot\mathcal{L}_{\partial_{y^i}}a_{\alpha}\otimes\cdot\cdot\cdot\otimes\iota_{\partial_{y^j}}a_{\beta}\otimes\cdot\cdot\cdot\otimes a_k.
    \end{align*}
    Here $\pm$ are Koszul signs by passing odd operators through the graded objects (see e.g. \cite{QFT}).

\subsection{Expectation values} In this section, we construct versions of expectation value maps that allows us to evaluate tensors in $\gl_r(\mathcal{W}_{2n})$. They arise naturally from topological quantum mechanics, which can be viewed as the process of integrating out massive modes in quantum field theory.

\subsubsection{Free expectation value map}
Let
$$
\Conf_{S^1}[m+1]=\{(p_0,p_1, \dots, p_m)\in (S^1)^{m+1}| p_i\neq p_j\} \subset (S^1)^{m+1}
$$
be the configuration space of $m$ ordered points on the circle $S^1$. Let
$$
 \Cyc_{S^1}[m+1]=\{(p_0,p_1,\dots,p_m)\in \Conf_{S^1}[m+1]| p_0, \dots, p_m\ \text{are anti-clockwise cyclic ordered} \}
$$
be the connected component of $\Conf_{S^1}[m+1]$ where points have the prescribed cyclic order.

\begin{figure}[h]

\definecolor{rvwvcq}{rgb}{0.08235294117647059,0.396078431372549,0.7529411764705882}
\begin{tikzpicture}[line cap=round,line join=round,>=triangle 45,x=1cm,y=1cm,scale=0.8,transform shape]
\clip(-6.87,-3.54) rectangle (7.47,4.3);
\draw [line width=2pt] (-0.01,0) circle (2.285607140345865cm);
\draw (-0.45,3.3) node[anchor=north west] {{$p_0$}};
\draw (-4.45,0.18) node[anchor=north west] {{}};
\draw (1.53,2.3) node[anchor=north west] {{$p_m$}};
\draw (2.63,0.94) node[anchor=north west] {{}};
\draw (-2.5,2.2) node[anchor=north west] {{$p_1$}};
\draw (-3.29,-1) node[anchor=north west] {{$\cdots$}};
\draw (-5.2,-1) node[anchor=north west] {{}};
\draw (2.27,-0.28) node[anchor=north west] {{$\cdot\cdot\cdot$}};
\draw (2.05,-1.18) node[anchor=north west] {{$p_k$}};
\draw (-0.13,-2.26) node[anchor=north west] {{}};
\draw (-2.41,-1.54) node[anchor=north west] {{}};
\begin{scriptsize}
\draw [fill=rvwvcq] (-2.29,0.16) circle (2.5pt);
\draw [fill=rvwvcq] (-0.1476741265598887,2.281456954421006) circle (2.5pt);
\draw [fill=rvwvcq] (-1.7250258764177986,1.5108561292252032) circle (2.5pt);
\draw [fill=rvwvcq] (1.4922446916995062,1.7225739131487674) circle (2.5pt);
\draw [fill=rvwvcq] (2.25417072637637,0.31229941053467203) circle (2.5pt);
\draw [fill=rvwvcq] (1.8207956964465248,-1.3682788889231925) circle (2.5pt);
\draw [fill=rvwvcq] (0.2272102048344094,-2.273264462996424) circle (2.5pt);
\draw [fill=rvwvcq] (-1.5107707124277856,-1.7238582507616451) circle (2.5pt);
\end{scriptsize}
\end{tikzpicture}

  \caption{}\label{CYCLICS}
\end{figure}

$\Cyc_{S^1}[m+1]$ has a natural compactification as follows. Let us identify $S^1=\R/\Z$ so that the total length of $S^1$ is 1. Given a cyclic ordered points $(p_0,\dots, p_m)$ on $S^1$, let $u_{i,i+1}$ denote the oriented distance by traveling  from $p_i$ to $p_{i+1}$ anti-clockwise ($p_{m+1}\equiv p_0$). Let $\Delta_m$ denote the standard simplex
$$
\Delta_m=\{(\lambda_0,\dots, \lambda_m)\in \R^{n+1}| \lambda_i\geq 0, \sum_{i=0}^m \lambda_i=1\}.
$$
Let $\Delta_{m}^o$ be the interior of $\Delta_m$. Then we have a natural identification
\begin{align*}
\Cyc_{S^1}[m+1] & \iso S^1\times \Delta_{m}^o\\
(p_0,\dots,p_m) & \to \{p_0\}\times (u_{0,1},u_{1,2},\dots, u_{m,0}).
\end{align*}
This allows us to compactify $\Cyc_{S^1}[m+1]$ by  $S^1\times \Delta_{m}$, which will be denoted by $S^1_{cyc}[m+1]$.

Similarly, we can compactify the whole space $\Conf_{S^1}[m]$, denoted by ${S^1}[m]$. ${S^1}[m]$ is a manifold with corners. Alternately, it could be constructed via successive real-oriented blow ups of diagonals in $(S^1)^m$ as described in \cite{axelrod1993chern,kontsevich1994feynman,getzler1994operads}. In particular, it carries a natural blow-down map
$$\pi:S^1[m]\rightarrow (S^1)^m.
$$

For example,  $S^1[2]$ is parametrized as a cylinder
$$S^1[2]=\{(e^{2\pi i\theta},u)|0\leq\theta<1,0\leq u\leq 1\}.$$
With this parametrization, the blow down map is
$$\pi:S^1[2]\rightarrow (S^1)^2,  \quad (e^{2\pi i\theta},u)\rightarrow (e^{2\pi i\theta},e^{2\pi i\theta+u}).$$
If we denote an element of $(S^1)^2$ by two ordered points $(p_0, p_1)$, then $u$ is the oriented distance by traveling from $p_0$ to $p_1$ anti-clockwise.

\begin{defn}We define the \textit{propagator} $P^{S^1}$ to be the following function on $S^1[2]$:
$$P^{S^1}(e^{2\pi i\theta},u)=u-\frac{1}{2}.$$
\end{defn}
Geometrically, $P^{S^1}$ represents the integral kernel of $``d^{-1}"$ on the circle. It is precisely the propagator for the free topological quantum mechanics.

Now we can formulate Feynman graph integrals on $S^1[k]$ as in \cite{axelrod1993chern,kontsevich1994feynman,getzler1994operads}.

\begin{defn} Let $\Omega^\*_{S^1[k]}$ be smooth differential forms on $S^1[k]$.
  We define the $\Omega^\*_{S^1[k]}-$linear operator
  $$\partial_P:\Omega^\*_{S^1[k]}\otimes (\hOmega^{-\*}_{{2n}})^{\otimes k}\rightarrow \Omega^\*_{S^1[k]}\otimes (\hOmega^{-\*}_{{2n}})^{\otimes k}$$
  by
  \begin{align*}
  \partial_{P}(a_1\otimes\cdots\otimes a_k)&:=\frac{1}{2}\sum_{1\leq\alpha\neq\beta\leq k}\pi^*_{\alpha\beta}(P^{S^1})\otimes \bracket{\hat\omega^{ij}a_1\otimes\cdots\otimes\mathcal{L}_{\partial_{y^i}}a_{\alpha}\otimes\cdots\otimes\mathcal{L}_{\partial_{y^j}}a_{\beta}\otimes...\otimes a_k}
  \end{align*}
  Here $a_i\in \hat\Omega^{-\*}_{{2n}}$ and $\pi_{\alpha\beta}:S^1[m]\rightarrow S^1[2]$ is the forgetful map to the two points indexed by $\alpha,\beta$. $\theta_{\alpha}\in [0,1)$ is the parameter on the $S^1$ indexed by $\alpha$ and $d\theta_\alpha$ is viewed as a 1-form in $S^1[k]$ via the pull-back $\pi_{\alpha}:S^1[m]\rightarrow S^1$.
\end{defn}

We also denote
$$
\int_{S^1[k]}: \Omega^\*_{S^1[k]}\otimes (\hOmega^{-\*}_{{2n}})^{\otimes k}\to (\hOmega^{-\*}_{{2n}})^{\otimes k}
$$
by integrating out differential forms on $S^1[k]$.

We define expectation value map for free topological quantum mechanics as follows.
\begin{defn}\label{defn-free-expectation}
We define the free expectation value map
$$\langle-\rangle_{free} : C_{-\*}(\gl_r(\mathcal{W}_{2n}))\rightarrow \hOmega^{-\*}_{{2n}}$$
by
$$\langle \hat{\mathcal{O}}_0\otimes \hat{\mathcal{O}}_1\otimes ...\otimes \hat{\mathcal{O}}_m\rangle_{free}=\tr{\int_{S^1_{cyc}[m+1]} d\theta_0 \cdots d\theta_{m}e^{\hbar\partial_P}\bracket{ \hat{\mathcal{O}}_0\otimes \d \hat{\mathcal{O}}_1\otimes \d\hat{\mathcal{O}}_2\otimes \cdots\otimes \d \hat{\mathcal{O}}_m}}\in \hOmega^{-m}_{2n}. $$
Here $\tr$ is the composition
$$
\tr:   \gl_r(\hOmega^{-\*}_{2n})\otimes \cdots \otimes \gl_r(\hOmega^{-\*}_{2n})\stackrel{\text{multiplying}}{\to} \gl_r(\hOmega^{-\*}_{2n})\stackrel{\text{trace}}{\to} \hOmega^{-\*}_{2n}
$$
of multiplying $r\times r$ matrices valued in $\hOmega^{-\*}_{2n}$ with an ordinary trace of the matrix.

\end{defn}

\begin{lem}
The expectation value map intertwines the de Rham differential $\d$ with the Connes operator $B$
$$\d \langle \hat{\mathcal{O}}\rangle_{free}=\langle B(\hat{\mathcal{O}})\rangle_{free}, \quad \quad \hat{\mathcal{O}}=\hat{\mathcal{O}}_0\otimes \hat{\mathcal{O}}_1\otimes\cdots\otimes \hat{\mathcal{O}}_m.$$
\end{lem}
\begin{proof}
We use the explicit formula of $B:\gl_r(\mathcal{W}_{2n})^{\otimes m+1}\rightarrow \gl_r(\mathcal{W}_{2n})^{\otimes m+2}$
$$B(\hat{\mathcal{O}}_0\otimes \hat{\mathcal{O}}_1\otimes ...\otimes \hat{\mathcal{O}}_m)=\sum_{i=0}^{m}(-1)^{mi}(1\otimes \hat{\mathcal{O}}_i\otimes \cdots \otimes \hat{\mathcal{O}}_m\otimes \hat{\mathcal{O}}_0\otimes...\otimes \hat{\mathcal{O}}_{i-1}).$$
Then
$$
\langle B(\hat{\mathcal{O}})\rangle_{free}=\tr \int_{S^1_{cyc}[n+2]}d\theta_0\cdots d\theta_{m+1}e^{\hbar\partial_P}\bracket{\sum_{i=0}^{m}(-1)^{mi}(1\otimes \d \hat{\mathcal{O}}_i\otimes\cdots\otimes \d \hat{\mathcal{O}}_m\otimes \d \hat{\mathcal{O}}_0\otimes...\otimes \d \hat{\mathcal{O}}_{i-1})}.
$$
Since the propagator $\pa_P$ can not be applied the the $1$-factor, the integrand does not depend on $\theta_0$. Using the cyclic symmetry of the trace, the above arbitrary insertion of $\hat \OO_0$ is equivalent to allowing $\theta_0$ to run over the whole $S^1$ while preserving the cyclic order of $\hat \OO_0, \hat\OO_1, \cdots, \hat\OO_m$. We can integrate out $\theta_0$, and it follows that

\begin{align*}\langle B(\hat{\mathcal{O}})\rangle_{free}
&=\tr \int_{S^1_{cyc}[m+1]}d\theta_0 \cdots d\theta_{m} e^{\hbar\partial_P}( \d \hat{\mathcal{O}}_0\otimes \d \hat{\mathcal{O}}_1\otimes\cdots\otimes  \d \hat{\mathcal{O}}_m)\\
&=\d \tr \int_{S^1_{cyc}[m+1]}d\theta_0 \cdots d\theta_{m} e^{\hbar\partial_P}( \hat{\mathcal{O}}_0\otimes \d \hat{\mathcal{O}}_1\otimes\cdots\otimes  \d \hat{\mathcal{O}}_m)\\
&=\d \langle \hat{\mathcal{O}}\rangle_{free}.
\end{align*}

\end{proof}

\begin{lem}
The expectation value map intertwining the operator $\hbar\Delta$ and the Hochschild boundary operator
$$\hbar\Delta\langle \hat{\mathcal{O}}\rangle_{free}=\langle b(\hat{\mathcal{O}})\rangle_{free}.$$
\end{lem}

\begin{proof}  Since $\Delta$ commutes with $\pa_P$, we have
$$
\hbar\Delta\langle \hat{\mathcal{O}}\rangle_{free}=\tr{\int_{S^1_{cyc}[m+1]} d\theta_0 \cdots d\theta_{m}e^{\hbar\partial_P}\hbar\Delta\bracket{ \hat{\mathcal{O}}_0\otimes \d \hat{\mathcal{O}}_1\otimes \d\hat{\mathcal{O}}_2\otimes\cdots\otimes \d \hat{\mathcal{O}}_m}}.
$$
The key observation is that the integrand is in fact an exact-form, leading to a boundary contribution. To see this, introduce the following $\Omega^\*_{S^1[m+1]}-$linear operator on $\Omega^\*_{S^1[m+1]}\otimes (\hOmega^{-\*}_{{2n}})^{\otimes (m+1)}$
$$
 D(a_0\otimes\cdots\otimes a_m)=\sum_{\alpha}\pm d\theta_{\alpha}\otimes \bracket{ a_0\otimes\cdots\otimes \d a_{\alpha}\otimes\cdots\otimes a_m}, \quad a_i \in \hOmega^{-\*}_{{2n}}.
$$
Then the expectation value map can be written as
$$
\abracket{\hat \OO}_{free}=\tr{\int_{S^1_{cyc}[m+1]} e^{\hbar\partial_P+D}\bracket{ d\theta_0\hat{\mathcal{O}}_0\otimes  \hat{\mathcal{O}}_1\otimes \hat{\mathcal{O}}_2\otimes\cdots\otimes  \hat{\mathcal{O}}_m}}.
$$
Let $d$ be the de Rham differential on ${S^1_{cyc}[m+1]}$. It is easy to check that $d-\hbar \Delta$ anticommutes with $\hbar\partial_P+D$ (this observation comes from \cite[Lemma 2.33]{BVQandindex}). Since
$$
  (d-\hbar \Delta)\bracket{ d\theta_0\hat{\mathcal{O}}_0\otimes  \hat{\mathcal{O}}_1\otimes \hat{\mathcal{O}}_2\otimes\cdots\otimes  \hat{\mathcal{O}}_m}=0,
$$
if follows that
\begin{align*}
\hbar\Delta\langle \hat{\mathcal{O}}\rangle_{free}&=\tr{\int_{S^1_{cyc}[m+1]} (\hbar \Delta)e^{\hbar\partial_P+D}\bracket{ d\theta_0\hat{\mathcal{O}}_0\otimes  \hat{\mathcal{O}}_1\otimes \hat{\mathcal{O}}_2\otimes\cdots\otimes  \hat{\mathcal{O}}_m}}\\
&=\tr{\int_{S^1_{cyc}[m+1]} d e^{\hbar\partial_P+D}\bracket{ d\theta_0\hat{\mathcal{O}}_0\otimes  \hat{\mathcal{O}}_1\otimes \hat{\mathcal{O}}_2\otimes\cdots\otimes  \hat{\mathcal{O}}_m}}\\
&=\tr{\int_{\pa S^1_{cyc}[m+1]} e^{\hbar\partial_P+D}\bracket{ d\theta_0\hat{\mathcal{O}}_0\otimes  \hat{\mathcal{O}}_1\otimes \hat{\mathcal{O}}_2\otimes\cdots\otimes  \hat{\mathcal{O}}_m}}.
\end{align*}

The boundary $\partial S^1_{cyc}[m+1]$ has $m+1$ components, each one corresponding to the collision of two adjacent points. Summing all boundary components we get Hochschild differential (Fig.\ref{Ho}). The fact that we get a Moyal product is due to the arbitrary distribution of the propagator $\pa_P$ and the property that $P^{S^1}$ becomes $\pm {1\over 2}$ when points collide (the sign is the direction of the collision).

\end{proof}

\begin{figure}

\definecolor{ccqqqq}{rgb}{0.8,0,0}
\definecolor{xdxdff}{rgb}{0.49019607843137253,0.49019607843137253,1}
\begin{tikzpicture}[line cap=round,line join=round,x=1cm,y=1cm]
\clip(-5.383780267054994,-2.411015216167219) rectangle (8.609640785576588,4.477142678569612);
\draw [line width=2pt] (-1.32,0) circle (0.7602631123499294cm);
\draw [line width=2pt] (0.72,0) circle (0.7717512552629899cm);
\draw [line width=2pt] (4,0) circle (0.7871467461661769cm);
\draw [line width=2pt] (-4,0) circle (0.82462cm);
\draw [->,line width=1pt] (-1.4362958317551557,0.588023172783776) -- (-1.088579890842029,0.4638389081719452);
\draw [->,line width=1pt] (-0.8774666410019162,0.2899809377153821) -- (-1.088579890842029,0.4638389081719452);
\draw [->,line width=1pt] (0.9480420487919997,0.5259310404778605) -- (1.1715737250932956,0.3644914964824805);
\draw [->,line width=1pt] (1.2460842838603943,0.11612296725881892) -- (1.1715737250932956,0.3644914964824805);
\draw [->,line width=1pt] (3.456564193950986,0.30239936417656516) -- (3.754606429019381,0.5010941875554944);
\draw [->,line width=1pt] (4.015393384704226,0.6004415992449591) -- (3.7345458939667004,0.4877204975203742);
\draw [->,line width=1pt] (-2.857464477581309,0) -- (-2.3048328986339404,0);
\draw (-0.5285171091602562,0.2) node[anchor=north west] {$\pm$};
\draw (1.5043776276818495,0.2) node[anchor=north west] {$\pm$};
\draw (2.688588153997639,0.2) node[anchor=north west] {$\pm$};
\draw (1.9977986803134284,0.2) node[anchor=north west] {$\cdot\cdot\cdot$};
\draw (-4.120622372318151,1.3) node[anchor=north west] {\tiny$\hat{\mathcal{O}}_0$};
\draw (-3.429832898633941,1.0) node[anchor=north west] {\tiny$\hat{\mathcal{O}}_m$};
\draw (-3.192990793370783,0.6) node[anchor=north west] {\tiny$\hat{\mathcal{O}}_{m-1}$};
\draw (-5.245622372318151,0.9) node[anchor=north west] {\tiny$\hat{\mathcal{O}}_1$};
\draw (-4.929832898633941,-0.654436268798804) node[anchor=north west] {\tiny$\hat{\mathcal{O}}_k$};
\draw (-1.61404342494973,1.3) node[anchor=north west] {\tiny$\hat{\mathcal{O}}_0\star\hat{\mathcal{O}}_m$};
\draw (-2.482464477581309,1.0) node[anchor=north west] {\tiny$\hat{\mathcal{O}}_1$};
\draw (-0.76,0.7) node[anchor=north west] {\tiny$\hat{\mathcal{O}}_{m-1}$};
\draw (-2.2258855302128877,-0.5557520582724883) node[anchor=north west] {\tiny$\hat{\mathcal{O}}_k$};
\draw (0.08332499610290182,-0.654436268798804) node[anchor=north west] {\tiny$\hat{\mathcal{O}}_k$};
\draw (3.1425355224186915,-0.6149625845882777) node[anchor=north west] {\tiny$\hat{\mathcal{O}}_k$};
\draw (4.4,1.1) node[anchor=north west] {\tiny$\hat{\mathcal{O}}_m$};
\draw (4.7,0.6) node[anchor=north west] {\tiny$\hat{\mathcal{O}}_{m-1}$};
\draw (0.43858815399763873,1.3) node[anchor=north west] {\tiny$\hat{\mathcal{O}}_0$};
\draw (-0.33114868810762454,1.0) node[anchor=north west] {\tiny$\hat{\mathcal{O}}_1$};
\draw (1.1096407855765862,1.0) node[anchor=north west] {\tiny$\hat{\mathcal{O}}_m\star\hat{\mathcal{O}}_{m-1}$};
\draw (2.6491144697871127,1.2) node[anchor=north west] {\tiny$\hat{\mathcal{O}}_1\star\hat{\mathcal{O}}_0$};
\draw (-5.5219381617918355,0.4126689943590897) node[anchor=north west] {$\cdot\cdot\cdot$};
\draw (-3.567990793370783,-0.4978573214303832) node[anchor=north west] {$\cdot\cdot\cdot$};
\draw (-2.6403592144234143,0.27451109962224776) node[anchor=north west] {$\cdot\cdot\cdot$};
\draw (-0.9627276354760457,-0.4386467951145938) node[anchor=north west] {$\cdot\cdot\cdot$};
\draw (1.2675355224186915,-0.24127837406196242) node[anchor=north west] {$\cdot\cdot\cdot$};
\draw (-0.39035921442341404,-0.1031204793251204685) node[anchor=north west] {$\cdot\cdot\cdot$};
\draw (2.6293776276818495,-0.1031204793251204685) node[anchor=north west] {$\cdot\cdot\cdot$};
\draw (4.563588153997639,-0.22154153195669928) node[anchor=north west] {$\cdot\cdot\cdot$};
\begin{scriptsize}
\draw [fill=xdxdff] (-4.056093858635844,0.8227099266590515) circle (2.5pt);
\draw [fill=xdxdff] (-4.680087905809231,0.4663459925549004) circle (2.5pt);
\draw [fill=xdxdff] (-3.4483583030094955,0.6129352188783382) circle (2.5pt);
\draw [fill=xdxdff] (-4.334909901876994,-0.7535472792232364) circle (2.5pt);
\draw [fill=xdxdff] (-1.9057023843149676,0.48471921460549094) circle (2.5pt);
\draw [fill=ccqqqq] (-0.98,0.68) circle (2.5pt);
\draw [fill=xdxdff] (-1.6124088893653576,-0.701781334476858) circle (2.5pt);
\draw [fill=xdxdff] (0.6823983080120576,0.7708346857528172) circle (2.5pt);
\draw [fill=xdxdff] (0.11925875967183541,0.4844687422001327) circle (2.5pt);
\draw [fill=xdxdff] (-3.18329231518371,0.11395921183483103) circle (2.5pt);
\draw [fill=xdxdff] (-0.5778409873888207,0.1649242250247065) circle (2.5pt);
\draw [fill=ccqqqq] (1.351143130657498,0.4441377586108319) circle (2.5pt);
\draw [fill=xdxdff] (0.5686471265297801,-0.7567643673510983) circle (2.5pt);
\draw [fill=ccqqqq] (3.6242715544557664,0.691684997096882) circle (2.5pt);
\draw [fill=xdxdff] (4.5656453376796176,0.547398713883501) circle (2.5pt);
\draw [fill=xdxdff] (4.783038972549198,0.08031168949222557) circle (2.5pt);
\draw [fill=xdxdff] (3.7147016048717862,-0.7336244446154077) circle (2.5pt);
\end{scriptsize}
\end{tikzpicture}

  \caption{Hochschild differential}\label{Ho}
\end{figure}
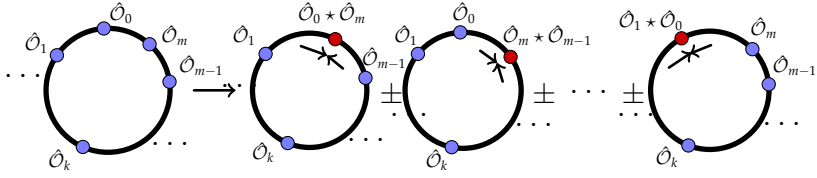

\subsubsection{Interacting expectation map} Recall (Section \ref{sec:Fedosov}) we are interested in the following Lie algebra
$$\mathfrak{g}:=\mathcal{W}^+_{2n}\cdot\text{Id}+\hbar \gl_r(\mathcal{W}_{2n}^+), \quad [-,-]={1\over  \hbar}[-,-]_\star.$$
Here $[-,-]_\star$ is the commutator with respect to the Moyal product.

The free expectation value map is an element
$$
\abracket{-}_{free}\in \Hom_{\mathbb{C}\LS{\hbar}}(C_{-\*}(\gl_r(\mathcal{W}_{2n})),\Omega^{-\bullet}_{{2n}})
$$
satisfying the following property
$$
(b-\hbar \Delta)\abracket{-}_{free}=0, \quad (B-\d)\abracket{-}_{free}=0.
$$
Now we are going to lift $\abracket{-}_{free}$ to Lie algebra cochains of $\g$ that will capture the interacting theory.

Let $\Id:\mathfrak{g}\rightarrow \mathfrak{g}$ be the identity map, which is viewed as a 1-chain in $C^1_{\Lie}(\mathfrak{g};\mathfrak{g})$ denoted by $\widehat{\Theta}$.

Let $(C^\*_{\Lie}(\mathfrak{g}),\pa)$ be the dg algebra with $\pa$ the \CE differential with respect to the trivial representation. We have an induced dg Lie algebra
$$
(C^\*_{\Lie}(\mathfrak{g})\otimes \g, \pa, [-,-])
$$
where $[-,-]$ is the Lie bracket on the $\g$-factor. Then $\widehat{\Theta}$ solves the Maurer-Cartan equation
$$
    \pa \widehat{\Theta}+{1\over 2}[\widehat{\Theta},\widehat{\Theta}]=0.
$$
As vector spaces, we can identify
$$
C^\*_{\Lie}(\mathfrak{g};\mathfrak{g})\iso C^\*_{\Lie}(\mathfrak{g})\otimes \g.
$$
Under this identification, the usual \CE differential $\pa_{\Lie}$ on $C^\*_{\Lie}(\mathfrak{g};\mathfrak{g})$ is
$$\partial_{\Lie}=\partial+[\widehat{\Theta},-].$$

As explained in the introduction, $\widehat{\Theta}$ can be viewed as a universal flat connection. Our goal is to turn this universal connection into an interaction added to our free theory.

Let us equip $\hOmega^{-\bullet}_{{2n}}$ with a trivial $\g$-module structure. Then
 $\Hom_{\mathbb{C}\LS{\hbar}}(C_{-\*}(\gl_r(\mathcal{W}_{2n})),\hOmega^{-\bullet}_{{2n}})$
has an induced $\g$-module structure from that on $\gl_r(\mathcal{W}_{2n})$ ($\g$ is a Lie subalgebra of $\gl_r(\mathcal{W}_{2n})$).

Let us recall the shuffle product on tensors, which we denote by $\times_{sh}$.  The shuffle product of a $p$-tensor with a $q$-tensor is a $(p+q)$-tensor defined by
\begin{align*}
(v_1\otimes \cdots \otimes v_p)\times_{sh} (v_{p+1}\otimes \cdots \otimes v_{p+q})
=\sum_{\sigma\in Sh(p,q)}v_{\sigma^{-1}(1)}\otimes \cdots \otimes v_{\sigma^{-1}(p+q)}.
\end{align*}
Here $Sh(p,q)$ is the subgroup of the permutation group $S_{p+q}$ given by $(p,q)$-shuffles
$$
Sh(p,q)=\{\sigma \in S_{p+q}| \sigma(1)<\cdots <\sigma (p)\ \text{and}\ \sigma(p+1)<\cdots<\sigma(p+q)\}.
$$

\begin{defn}\label{defn:interaction}
We define the interacting expectation map
$$\langle-\rangle_{int}\in C_{\Lie}^\*(\mathfrak{g};\Hom_{\mathbb{C}\LS{\hbar}}(C_{-\*}(\gl_r(\mathcal{W}_{2n})),\hOmega^{-\bullet}_{{2n}}))$$
by
\begin{align*}\langle \hat{\mathcal{O}}_0\otimes \hat{\mathcal{O}}_1\otimes \cdots \otimes \hat{\mathcal{O}}_m\rangle_{int}&:=\abracket{\sum_{k\geq 0} \hat{\mathcal{O}}_0\otimes \bracket{\bracket{\hat{\mathcal{O}}_1\otimes \cdots \otimes \hat{\mathcal{O}}_m}  \times_{sh} (\widehat\Theta/\hbar)^{\otimes k} }}_{free}\\
&=\sum_{k\geq 0}\sum_{ \substack{i_0,...,i_m\geq 0\\ i_0+\cdots+i_m=k}} \abracket{\widehat{\mathcal{O}}_0 \otimes (\widehat\Theta/\hbar)^{\otimes i_0}\otimes \hat{\mathcal{O}}_1\otimes (\widehat\Theta/\hbar)^{\otimes i_1}\otimes \cdots \otimes \hat{\mathcal{O}}_m\otimes(\widehat\Theta/\hbar)^{\otimes i_m}}_{free}.
\end{align*}
See Fig.\ref{int} for an illustration. In the above expression, we view $\hat \Theta$ as a Lie algebra 1-cochain valued in $\gl_r(\W_{2n})$, i.e., an element in $C^1(\g, \gl_r(\W_{2n}))$, using the natural embedding $\g\subset \gl_r(\W_{2n})$.   Therefore the following component of $\langle-\rangle_{int}$
$$
\hat{\mathcal{O}}_0\otimes \hat{\mathcal{O}}_1\otimes \cdots \otimes \hat{\mathcal{O}}_m\to \abracket{ \hat{\mathcal{O}}_0\otimes \bracket{\bracket{\hat{\mathcal{O}}_1\otimes \cdots \otimes \hat{\mathcal{O}}_m}  \times_{sh} (\widehat\Theta/\hbar)^{\otimes k} }}_{free}
$$
provides an element in $C_{\Lie}^k(\mathfrak{g};\Hom_{\mathbb{C}\LS{\hbar}}(C_{-m}(\gl_r(\mathcal{W}_{2n})),\hOmega^{-(m+k)}_{{2n}}))$. Koszul sign convention is always assumed to organize such a map into a Lie algebra cochain.

\end{defn}

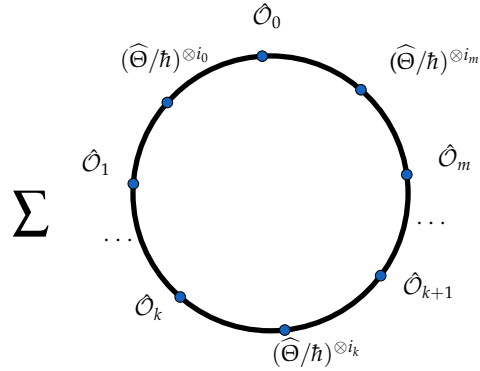
\begin{figure}

\definecolor{rvwvcq}{rgb}{0.08235294117647059,0.396078431372549,0.7529411764705882}
\begin{tikzpicture}[line cap=round,line join=round,>=triangle 45,x=1cm,y=1cm,scale=0.8,transform shape]
\clip(-6.87,-3.54) rectangle (7.47,4.3);
\draw [line width=2pt] (-0.01,0) circle (2.285607140345865cm);
\draw (-0.45,3.3) node[anchor=north west] {{$\hat{\mathcal{O}}_0$}};
\draw (-4.45,0.18) node[anchor=north west] {\Huge{$\sum$}};
\draw (1.83,2.6) node[anchor=north west] {{($\widehat{\Theta}/\hbar)^{\otimes i_m}$}};
\draw (2.63,0.94) node[anchor=north west] {{$\hat{\mathcal{O}}_{m}$}};
\draw (-2.65,2.62) node[anchor=north west] {{$(\widehat{\Theta}/\hbar)^{\otimes i_0}$}};
\draw (-3.29,0.86) node[anchor=north west] {{$\hat{\mathcal{O}}_1$}};
\draw (-2.93,-0.56) node[anchor=north west] {{$\cdot\cdot\cdot$}};
\draw (2.27,-0.28) node[anchor=north west] {{$\cdot\cdot\cdot$}};
\draw (2.05,-1.18) node[anchor=north west] {{$\hat{\mathcal{O}}_{k+1}$}};
\draw (-0.13,-2.26) node[anchor=north west] {{$(\widehat{\Theta}/\hbar)^{\otimes i_{k}}$}};
\draw (-2.41,-1.54) node[anchor=north west] {{$\hat{\mathcal{O}}_{k}$}};
\begin{scriptsize}
\draw [fill=rvwvcq] (-2.29,0.16) circle (2.5pt);
\draw [fill=rvwvcq] (-0.1476741265598887,2.281456954421006) circle (2.5pt);
\draw [fill=rvwvcq] (-1.7250258764177986,1.5108561292252032) circle (2.5pt);
\draw [fill=rvwvcq] (1.4922446916995062,1.7225739131487674) circle (2.5pt);
\draw [fill=rvwvcq] (2.25417072637637,0.31229941053467203) circle (2.5pt);
\draw [fill=rvwvcq] (1.8207956964465248,-1.3682788889231925) circle (2.5pt);
\draw [fill=rvwvcq] (0.2272102048344094,-2.273264462996424) circle (2.5pt);
\draw [fill=rvwvcq] (-1.5107707124277856,-1.7238582507616451) circle (2.5pt);
\end{scriptsize}
\end{tikzpicture}

  \caption{$\langle-\rangle_{int}$}\label{int}
\end{figure}

\begin{thm}\label{thm-interaction}
  $\langle-\rangle_{int}$ is closed under the differential $\partial_{\Lie}+b-\hbar\Delta$ and $B-\d$. In other words
  $$
  \partial_{\Lie}\langle-\rangle_{int}+\langle b(-)\rangle_{int}=\hbar \Delta \langle-\rangle_{int}, \quad \d\langle-\rangle_{int}=\langle B(-)\rangle_{int}.
  $$
\end{thm}
\begin{proof}The equation
  $\d \langle-\rangle_{int}=\langle B(-)\rangle_{int}$
  follows from the equation for $\langle-\rangle_{free}$. We next compute $\hbar\Delta\langle -\rangle_{int}$.  There will be three terms.   The first is about collision inside $\hat \OO$'s which gives rise to
  $$\langle b(-)\rangle_{int}.$$
The second term (Fig.\ref{ACTION}) gives $[\hat \Theta/\hbar,\hat \OO_k]_\star$. The third term (Fig.\ref{LIEACTION}) gives $-\frac{1}{2\hbar}[\widehat{\Theta},\widehat{\Theta}]_\star$ which equals $\pa \widehat{\Theta}$ by its Maurer-Cartan equation. Since  $\partial_{\Lie}=\partial+\frac{1}{\hbar}[\widehat{\Theta},-]_\star$, the sum of these two terms gives
  $$\partial_{\Lie}\langle-\rangle_{int}.$$
\end{proof}

\begin{figure}[htp]
\definecolor{ccqqqq}{rgb}{0.8,0,0}
\definecolor{xdxdff}{rgb}{0.49019607843137253,0.49019607843137253,1}
\begin{tikzpicture}[line cap=round,line join=round,>=triangle 45,x=1cm,y=1cm]
\clip(-7.32,-3.18) rectangle (7.2,1.84);
\draw [line width=2pt] (-5.36,0) circle (1.1260550608207396cm);
\draw [line width=2pt] (-1.42,0) circle (1.1280070921762846cm);
\draw [line width=2pt] (2.72,0) circle (1.1180339887498925cm);
\draw (-4.12,0.06) node[anchor=north west] {$\pm$};
\draw (-3.18,0.06) node[anchor=north west] {$\pm$};
\draw (-3.74,0.06) node[anchor=north west] {$\cdot\cdot\cdot$};
\draw (-6.46,-0.82) node[anchor=north west] {$\hat{\mathcal{O}}_k$};
\draw (-7.14,0.52) node[anchor=north west] {$\cdot\cdot\cdot$};
\draw (-1.62,-1.2) node[anchor=north west] {$\hat{\mathcal{O}}_k$};
\draw (1.96,-1.2) node[anchor=north west] {$\frac{1}{\hbar}[\widehat{\Theta},\hat{\mathcal{O}}_{k}]_{\star}$};
\draw (-0.42,-0.18) node[anchor=north west] {$\cdot\cdot\cdot$};
\draw (3.8,-0.22) node[anchor=north west] {$\cdot\cdot\cdot$};
\draw [->,line width=1pt] (-4.92,-1.26) -- (-5.58,-1.36);
\draw [->,line width=1pt] (-2.16,-1.16) -- (-1.58,-1.38);
\draw [->,line width=1pt] (0.24,-0.08) -- (1.18,-0.06);
\draw (-4.72,-0.8) node[anchor=north west] {$\widehat{\Theta}/\hbar$};
\draw (-3.14,-0.6) node[anchor=north west] {$\widehat{\Theta}/\hbar$};
\begin{scriptsize}
\draw [fill=xdxdff] (-5.36,1.1260550608207396) circle (2.5pt);
\draw [fill=xdxdff] (-6.341442044466652,0.5520611500124926) circle (2.5pt);
\draw [fill=xdxdff] (-5.7531159823220825,-1.0552060578119058) circle (2.5pt);
\draw [fill=xdxdff] (-4.712934505795452,-0.9215781281095079) circle (2.5pt);
\draw [fill=xdxdff] (-4.573180822944892,0.8055529669849912) circle (2.5pt);
\draw [fill=xdxdff] (-1.4582155781857178,1.1273595564787342) circle (2.5pt);
\draw [fill=xdxdff] (-2.349525134747477,0.639048530138892) circle (2.5pt);
\draw [fill=xdxdff] (-2.1867793633112433,-0.8273145762042381) circle (2.5pt);
\draw [fill=xdxdff] (-1.3562525832360732,-1.1262043628293772) circle (2.5pt);
\draw [fill=xdxdff] (-0.5550919068831923,0.7241091012140698) circle (2.5pt);
\draw [fill=xdxdff] (2.6807948634769896,1.1173463909057955) circle (2.5pt);
\draw [fill=xdxdff] (1.8230759917113275,0.6674783317497098) circle (2.5pt);
\draw [fill=ccqqqq] (2.743782574707725,-1.1177810112630586) circle (2.5pt);
\draw [fill=xdxdff] (3.6377999171377637,0.6384695075740964) circle (2.5pt);
\end{scriptsize}
\end{tikzpicture}

  \caption{The term $[\widehat\Theta,\hat{\mathcal{O}}_i]/\hbar$}\label{ACTION}
\end{figure}
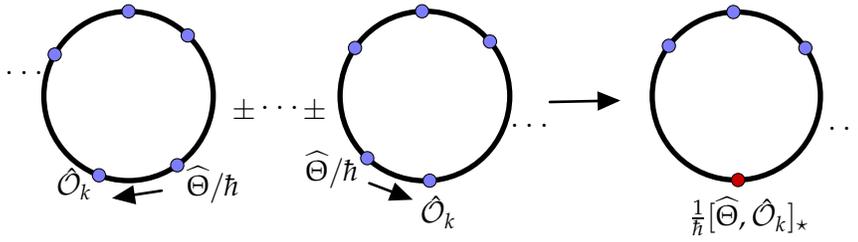

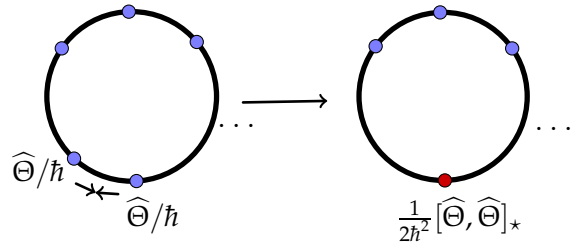
\begin{figure}[htp]
\definecolor{ccqqqq}{rgb}{0.8,0,0}
\definecolor{xdxdff}{rgb}{0.49019607843137253,0.49019607843137253,1}
\begin{tikzpicture}[line cap=round,line join=round,x=1cm,y=1cm]
\clip(-7.32,-3.18) rectangle (7.2,1.84);
\draw [line width=2pt] (-2.42,0) circle (1.1280070921762846cm);
\draw [line width=2pt] (1.72,0) circle (1.1180339887498925cm);
\draw (-2.62,-1.2) node[anchor=north west] {$\widehat{\Theta}/\hbar$};
\draw (0.96,-1.2) node[anchor=north west] {$\frac{1}{2\hbar^2}[\widehat{\Theta},\widehat{\Theta}]_{\star}$};
\draw (-1.42,-0.18) node[anchor=north west] {$\cdot\cdot\cdot$};
\draw (2.8,-0.22) node[anchor=north west] {$\cdot\cdot\cdot$};
\draw [->,line width=1pt] (-3.16,-1.16) -- (-2.9,-1.28);
\draw [->,line width=1pt] (-2.6,-1.3) -- (-2.9,-1.28);
\draw [->,line width=1pt] (-0.94,-0.08) -- (0.18,-0.06);
\draw (-4.14,-0.6) node[anchor=north west] {$\widehat{\Theta}/\hbar$};
\begin{scriptsize}
\draw [fill=xdxdff] (-2.4582155781857178,1.1273595564787342) circle (2.5pt);
\draw [fill=xdxdff] (-3.349525134747477,0.639048530138892) circle (2.5pt);
\draw [fill=xdxdff] (-3.1867793633112433,-0.8273145762042381) circle (2.5pt);
\draw [fill=xdxdff] (-2.3562525832360732,-1.1262043628293772) circle (2.5pt);
\draw [fill=xdxdff] (-1.5550919068831923,0.7241091012140698) circle (2.5pt);
\draw [fill=xdxdff] (1.6807948634769896,1.1173463909057955) circle (2.5pt);
\draw [fill=xdxdff] (0.8230759917113275,0.6674783317497098) circle (2.5pt);
\draw [fill=ccqqqq] (1.743782574707725,-1.1177810112630586) circle (2.5pt);
\draw [fill=xdxdff] (2.6377999171377637,0.6384695075740964) circle (2.5pt);
\end{scriptsize}
\end{tikzpicture}

  \caption{The term $[\widehat\Theta,\widehat{\Theta}]/2\hbar^2$}\label{LIEACTION}
\end{figure}

For applications, it would be useful to $\C[u,u^{-1}]$-linearly extend $\abracket{-}_{int}$ naturally to an element
$$
\abracket{-}_{int}^{S^1}\in C_{\Lie}^\*(\mathfrak{g};\Hom_{\mathbb{K}}(CC_{-\*}^{per}(\gl_r(\mathcal{W}_{2n})),\hOmega^{-\bullet}_{2n}[u,u^{-1}]).
$$
We will still call it $\abracket{-}_{int}$. An an element of this cyclic extended object, we have the following.

\begin{thm}\label{thm-interaction-S1}
$\abracket{-}_{int}$ is closed under the differential
$
(\partial_{\Lie}+b-\hbar\Delta)+u(B-\d).
$
\end{thm}
\begin{proof} This follows from Theorem \ref{thm-interaction}.

\end{proof}

\subsection{The Universal trace}\label{trace-map} Now we use expectation value map to construct the universal trace.

\begin{defn}\label{BV-integration} We define the Berezin integration map by
\begin{align*}
  \int_{BV}: \hOmega^{-\*}_{2n}[u,u^{-1}] & \to \mathbb{K}=\C\LS{\hbar}[u,u^{-1}]\\
          a     &\to \sigma(u^ne^{\hbar \iota_{\Pi}/u}a).
\end{align*}
Here $\sigma$ is the symbol map which sets $y^i,dy^i$'s to zero. $ \int_{BV}$ has cohomology degree $2n$.
\end{defn}

\begin{rem} The transformation $e^{\hbar \iota_{\Pi}/u}(-)$ at the cohomology level is called the character map in \cite{CATTANEO20111966}.
\end{rem}

$\int_{BV}$ can be viewed as an equivariant version of Berezin integral for BV algebras, where the super Lagrangian is  the pure fermionic one. $ \int_{BV}$ is  a $\deg=2n$ map of cochain complexes
$$
  \int_{BV}: (\hOmega^{-\*}_{2n}[u,u^{-1}], \hbar \Delta+ u\d ) \to (\mathbb{K},0).
$$

\begin{defn}We define the universal trace map to be the element
$$
\uTr := \int_{BV}\abracket{-}_{int}\in C^\*(\mathfrak{g};\Hom_{\mathbb{K}}(CC_{-\*}^{per}(\gl_r(\mathcal{W}_{2n})),\mathbb{K})).
$$
Given $a_1, \dots,a_k\in \g$, we will write
$$
\uTr[a_1,\dots, a_k](-)\in \Hom_{\mathbb{K}}(CC_{-\*}^{per}(\gl_r(\mathcal{W}_{2n})),\mathbb{K})
$$
for the corresponding evaluation map on cyclic tensors.
\end{defn}

\begin{thm}\label{TRACE}
  The universal trace lies in the $(\g, \mathfrak{h})$-Lie algebra cochain complex
  $$
  \uTr \in C^\*(\mathfrak{g},\mathfrak{h};\Hom_{\mathbb{K}}(CC^{per}_{-\bullet}(\gl_r(\mathcal{W}_{2n})),\mathbb{K}))
  $$
  and is closed under the differential
  $$
    (\pa_{\Lie}+b+uB) \uTr=0.
  $$
\end{thm}
\begin{proof} The equation $(\pa_{\Lie}+b+uB) \uTr=0$ follows from Theorem \ref{thm-interaction-S1} and the fact that $\int_{BV}$ is a cochain map.  We only need to show that $\uTr$ is $\mathfrak{h}$-invariant and vanishes when there is some argument taking value in $\mathfrak{h}$.

   Note that for $f\in\mathfrak{h}$, $\d f$ is linear in $y^i$ so $\partial_P$ can be only applied once. It will contribute to the expectation value map $\abracket{\ldots, f ,\ldots}_{free}$ a factor of
   $$\int^{1}_{0}(u-\frac{1}{2})du=0.$$
 Thus
   $$\uTr[...,f,...](-)=\int_{BV} \langle \ldots, f, \ldots\rangle_{free}=0.$$

  Finally, $\uTr$ is $\mathfrak{h}-$invariant since the operations $\Tr,\iota_{{\Pi}},\partial_P,\sigma$ are all $\mathfrak{h}-$invariant.
\end{proof}

\subsection{Gauss-Manin connection and Index}\label{sec: GM}
In this section, we use $\uTr$ to implement the idea of exact semi-classical approximation to prove the algebraic index theorem (Theorem \ref{thm-index}).

As we mentioned in the introduction, we need to explore the $\hbar$-variation in terms of certain Gauss-Manin connection. In our case, the relevant one is Getzler's Gauss-Manin connection  \cite{getzler1993cartan} on periodic cyclic homologies. We briefly recall this construction.

Assume $\hbar$ is the deformation parameter and we are given a family of unital (homotopic) associative multiplication $m_{\hbar}$ on $A_{\hbar}$ parametrized by $\hbar$. Getzler's Gauss-Manin connection is defined on the periodic cyclic chains $CC_{-\*}^{per}(A_\hbar)$ by

$$
\nabla^{GM}_{\partial_{\hbar}}=\partial_{\hbar}+\frac{1}{u}\iota_{\frac{\partial m_\hbar}{\partial \hbar}}.
$$
Here $m_{\hbar}$ is viewed as an element in the Hochschild cochain $C^\*(A,A)$, $\iota$ is the noncommutative analogue of contraction, and $\mathcal L$ is the noncommutative analogue of Lie derivative.  The following homotopic Cartan formula holds
$$
[b+uB,\iota_\phi]+\iota_{[m_{\hbar},\phi]}=u\mathcal{L}_\phi, \quad \forall\phi\in C^{\bullet}(A,A).
$$
Note that the Hochschild differential $b$ depends on $\hbar$ through $m_{\hbar}$, but the Connes operator $B$ is independent of $\hbar$. We refer to  \cite{getzler1993cartan} for detailed explanation.

For our application, we only need a simplified situation when the deformation is inner
$$
\frac{\partial m_\hbar}{\partial \hbar}=[m_\hbar,\varphi_\hbar], \quad \text{for some} \ \varphi_\hbar\in C^{\bullet}(A,A).
$$
Then the homotopy Cartan formula implies that
$$
\nabla^{GM}_{\partial_{\hbar}}=\partial_{\hbar}+\mathcal{L}_{\varphi_\hbar}-\frac{1}{u}[b+uB,\iota_{\varphi_\hbar}].
$$

In our case, we apply the above construction to the Weyl algebra $\W_{2n}$. It would be convenient to work with the logarithmic derivative $\hbar\pa_\hbar$. Observe that the following transformations
$$
 y^i\to \lambda y^i, \quad \hbar \to \lambda^2 \hbar, \quad \lambda\in \C^*
$$
are automorphism of $\W_{2n}$. This realizes the $\hbar\pa_{\hbar}$variation by an explicit inner deformation.

Introduce the vector field
$$
\E=\frac{1}{2}\sum_i y^i\frac{\partial}{\partial y^i}
$$
which gives a $\hbar$-linear map on $\W_{2n}$ and can be viewed as an element of $C^1(\W_{2n},\W_{2n})$. Let $m_{\hbar}$ denote the Moyal product on $\W_{2n}$. Then the above automorphism says that $m_\hbar$ is invariant under simultaneous rescaling of $y$ and $\hbar$. The infinitesimal version is translated into the equation
$$
    \hbar\pa_{\hbar} m_\hbar+[\E,m_\hbar]=0 \quad \text{or equivalently}\quad \hbar\pa_{\hbar} m_\hbar=[m_\hbar,\E].
$$
It follows that the Gauss-Manin connection has the form
$$
       \nabla_{\hbar\pa_\hbar}^{GM}=\hbar\pa_\hbar+\mathcal L_{\E} + \text{ $(b+uB)$-homotopy}.
$$
This suggests that we should look at rescaling property under
$$
\hbar\partial_{\hbar}+{1\over 2}\sum y^i\frac{\partial}{\partial{y^i}}.
$$
This is indeed how index theorem was derived in \cite{BVQandindex}. In the rest of this section, we will basically show that the argument in \cite{BVQandindex} can be lifted to a universal calculation in Lie algebra cohomology.

Let us first set up some notations motivated from the above discussion.

\begin{defn} We define the $\hbar$-connection $\nabla$ on the $\C\LS{\hbar}$-module $\W_{2n}$ and $\hat \Omega_{2n}^{-\*}$ by
$$
  \nabla_{\hbar\pa_\hbar}={\hbar\pa_\hbar}+\E.
$$
Here the action of $\E$ on $\hat \Omega_{2n}^{-\*}$ is understood by Lie derivative with respect to the vector field $\frac{1}{2}\sum_i y^i\frac{\partial}{\partial y^i}$.
\end{defn}

Note that we have induced $\hbar$-connections on tensors of $\gl_r(\W_{2n})$, $\Hom_{\mathbb{C}\LS{\hbar}}(C_{-\*}(\gl_r(\mathcal{W}_{2n})),\hOmega^{-\bullet}_{{2n}})$, $\Hom_{\mathbb{K}}(CC_{-\*}^{per}(\gl_r(\mathcal{W}_{2n})),\mathbb{K})$, etc.  All of them will still be denoted by $\nabla_{\hbar\pa_\hbar}$.

\begin{defn} Let $V$ be a $\C\LS{\hbar}$-module where the connection $\nabla_{\hbar\pa_\hbar}$ is defined. Assume $V$ carries a structure of $\g$-module. We define the linear map
$$
\nabla_{\hbar\pa_\hbar}:  C^\*(\mathfrak{g};V)\to C^\*(\mathfrak{g};V)
$$
by extending that on $V$. In other words, given $\varphi \in C^k(\mathfrak{g};V)$,  the cochain $\nabla_{\hbar\pa_\hbar} \varphi\in C^k(\mathfrak{g};V)$ is
$$
(\nabla_{\hbar\pa_\hbar} \varphi) (a_1,\dots, a_k):= \nabla_{\hbar\pa_\hbar} \bracket{\varphi (a_1,\dots, a_k)}, \quad a_i \in \g.
$$
Similarly we define $\nabla_{\hbar\pa_\hbar}$ on $C^\*(\mathfrak{g}, \mathfrak{h};V)$.
\end{defn}

\begin{rem} Although $\g$ is  a Lie algebra over $\C\PS{\hbar}$, we do not allow $\nabla_{\hbar\pa_\hbar}$ to acts on the $\g$-factor. The reason is that cochains in $C^\bullet(\g; V)$ is only $\C$-linear, but not $\C\PS{\hbar}$-linear. For examples, the projections and curvatures defined in Section \ref{sec:Char} are only $\C$-linear maps.

\end{rem}

In particular,  we have now a well-defined operator
$$
\nabla_{\hbar\pa_\hbar}: C^\*(\mathfrak{g},\mathfrak{h};\Hom_{\mathbb{K}}(CC_{-\*}^{per}(\gl_r(\mathcal{W}_{2n})),\mathbb{K})) \to C^\*(\mathfrak{g},\mathfrak{h};\Hom_{\mathbb{K}}(CC_{-\*}^{per}(\gl_r(\mathcal{W}_{2n})),\mathbb{K})).
$$
Note that this operator does not commute with $\pa_{\Lie}$. In fact, using $\pa_{\Lie}=\pa+[\hat\Theta,-]$, we find
$$
[\nabla_{\hbar\pa_\hbar}, \pa_{\Lie}]=[\nabla_{\hbar\pa_\hbar}, \pa+{1\over \hbar }[\hat\Theta,-]_\star]=\bbracket{\nabla_{\hbar\pa_\hbar}(\hat\Theta/\hbar),-}_\star.
$$
Here $\hat\Theta$ is viewed as an element in $C^1(\g; \gl_r(\W_{2n}))$ so $\nabla_{\hbar\pa_\hbar}$ can be applied.   To derive the last equality, we have used the fact that the Moyal commutator $[-,-]_\star$ is compatible with $\nabla_{\hbar\pa_\hbar}$. The above commutator relation will play an important role in deriving the algebra index below.

\begin{lem}The free expectation value map
$
\abracket{-}_{free}\in \Hom_{\mathbb{C}\LS{\hbar}}(C_{-\*}(\gl_r(\mathcal{W}_{2n})),\Omega^{-\bullet}_{{2n}})
$ is flat with respect to $\nabla_{\hbar\pa_\hbar}$. In other words, for any $\OO\in C_{-\*}(\gl_r(\mathcal{W}_{2n}))$,
$$
\nabla_{\hbar\pa_\hbar}\abracket{ \OO}_{free} =\abracket{ \mathcal \nabla_{\hbar\pa_\hbar} (\OO)}_{free}.
$$
Similarly, $\int_{BV}$ is flat with respect to $\nabla_{\hbar\pa_\hbar}$.
\end{lem}
This lemma follows from a direct check. This also implies that $\abracket{-}_{free}$ is homotopic flat with respect to Getzler's Gauss-Manin connection.

With the above preparation,  we now come to the computation of the index. We will adopt the same notations as in Section \ref{sec:Char} in the rest of this section.

Recall the $\mathfrak{h}-$equivariant projection $\pr$ in Section \ref{sec:Char}. Let
$$\widehat{\gamma}:=\widehat\Theta-\pr\in C^1_{\Lie}(\mathfrak{g},\mathfrak{h};\mathfrak{g}).$$

By Theorem \ref{TRACE},  we can replace $\widehat{\Theta}$ by $\widehat{\gamma}$ in the expression of $\uTr$ since $\pr$ lies in $\mathfrak{h}$. We always assume this replacement in the following discussions. As in Definition \ref{defn:interaction}, when we insert $\widehat \gamma$ to define $\uTr$, it is viewed as a Lie algebra 1-cochain valued in $\gl_r(\W_{2n})$, i.e. an element of $C^1(\g,\mathfrak{h}; \gl_r(\W_{2n}))$, using the natural embedding $\g\subset \gl_r(\W_{2n})$.

The next proposition is the crucial observation in computing the algebraic index.
\begin{prop}\label{con:EACT}
In the cochain complex $C^\*(\mathfrak{g},\mathfrak{h};\mathbb{K})$
$$\nabla_{\hbar\partial_{\hbar}}(e^{R_3/u\hbar}\uTr(1))=\partial_{\Lie}\text{-exact term}.$$
\end{prop}
\begin{proof} Since $\abracket{-}_{free}$ and $\int_{BV}$ are $\nabla_{\hbar\pa_\hbar}$-flat, we have
\begin{align*}
\nabla_{\hbar\partial_{\hbar}}(\uTr(1)) & =\int_{BV}\nabla_{\hbar\partial_{\hbar}}\langle 1\rangle_{int}=\int_{BV}\langle 1\otimes \nabla_{\hbar\partial_{\hbar}}(\widehat\gamma/\hbar)\rangle_{int}
=\int_{BV}\langle B(\nabla_{\hbar\partial_{\hbar}}(\widehat\gamma/\hbar))\rangle_{int}\\ &= \uTr(B(\nabla_{\hbar\partial_{\hbar}}(\widehat\gamma/\hbar)))=u^{-1}\uTr((b+uB)(\nabla_{\hbar\partial_{\hbar}}(\widehat\gamma/\hbar)))\\
&=u^{-1}\uTr((\partial_{\Lie}+uB+b)(\nabla_{\hbar\partial_{\hbar}}(\widehat\gamma/\hbar)))-u^{-1}\uTr(\partial_{\Lie}(\nabla_{\hbar\partial_{\hbar}}(\widehat\gamma/\hbar))).
\end{align*}
By Theorem \ref{TRACE}, the first term is
$$
\partial_{\Lie}\bbracket{u^{-1}\uTr((\nabla_{\hbar\partial_{\hbar}}(\widehat\gamma/\hbar)))}
$$
which is $\partial_{\Lie}\text{-exact}$.  To compute the second term, we observe
$$
\nabla_{\hbar\partial_{\hbar}}(\pr/\hbar)=\nabla_{\hbar\partial_{\hbar}}(\pr_3/\hbar)
$$
which is valued in the center of $\gl_r(\W_{2n})$. This is because $\pr_1/\hbar, \pr_2/\hbar$ are valued in $\hbar^{-1}\sp_{2n}, \gl_r$ which are annihilated are $\nabla_{\hbar\partial_{\hbar}}$. It follows that
\begin{align*}
   \partial_{\Lie}(\nabla_{\hbar\partial_{\hbar}}(\widehat\gamma/\hbar))&=\partial(\nabla_{\hbar\partial_{\hbar}}(\widehat\gamma/\hbar))+\bbracket{\hat\Theta/\hbar, \nabla_{\hbar\partial_{\hbar}}(\widehat\gamma/\hbar)}_{\star}\\
   &=\partial(\nabla_{\hbar\partial_{\hbar}}(\widehat\Theta/\hbar))+\bbracket{\hat\Theta/\hbar, \nabla_{\hbar\partial_{\hbar}}(\widehat\Theta/\hbar)}_{\star}-\partial(\nabla_{\hbar\partial_{\hbar}}(\pr_3/\hbar))\\
      &=\nabla_{\hbar\partial_{\hbar}}\bracket{\partial(\widehat\Theta/\hbar)+{1\over 2}\bbracket{\hat\Theta/\hbar, \widehat\Theta/\hbar}_{\star}}-\partial(\nabla_{\hbar\partial_{\hbar}}(\pr_3/\hbar))\\
      &=-\nabla_{\hbar\partial_{\hbar}}(\partial(\pr_3/\hbar))=\nabla_{\hbar\partial_{\hbar}}(R_3/\hbar).
\end{align*}
Combining the above computations, we find
$$
\nabla_{\hbar\partial_{\hbar}}(\uTr(1))=-\nabla_{\hbar\partial_{\hbar}}(R_3/u\hbar) \uTr(1)+\partial_{\Lie}\bbracket{u^{-1}\uTr((\nabla_{\hbar\partial_{\hbar}}(\widehat\gamma/\hbar)))}.
$$
Since $R_3$ is $\pa_{\Lie}$-closed,
$$
\nabla_{\hbar\partial_{\hbar}}(e^{R_3/u\hbar}\uTr(1))=\partial_{\Lie}\bracket{e^{R_3/u\hbar}\bbracket{u^{-1}\uTr((\nabla_{\hbar\partial_{\hbar}}(\widehat\gamma/\hbar)))}}.
$$
\end{proof}
\begin{rem}
  This argument is the major difference between \cite{PFLAUM20101958,WILLWACHER2015277,feigin1989riemann} and our method. They use large N techniques or explicit evaluations on certain Lie subalgebra to reduce the computations to tree and one loop diagrams.
\end{rem}

\noindent \textbf{Semi-classical analysis and algebraic index}. Recall that given an even cochain $A=\sum_{p\ even}A_p,  A_p\in C_{Lie}^p(\g,\mathfrak{h};\mathbb{K})$, we denote $A_u:=\sum_{p}u^{-p/2}A_p$.

 The following proposition generalizes the semi-classical analysis of \cite[Lemma 2.59]{BVQandindex}.
\begin{prop}\label{ONELOOP} In the cochain complex $C^\*(\mathfrak{g},\mathfrak{h};\mathbb{K})$
$$\uTr(1)=u^ne^{-R_3/u\hbar}(\hat{A}(\mathfrak{sp}_{2n})_u\cdot Ch(\gl_r)_{u}+O(\hbar)).$$
Precisely, this means that $e^{R_3/u\hbar} \uTr(1)$ is a cochain in  $C^\*(\mathfrak{g},\mathfrak{h};\mathbb{K})$ whose value in $\mathbb K$ has only non-negative powers in $\hbar$ with leading $\hbar$-order $u^n\hat{A}(\mathfrak{sp}_{2n})_u\cdot Ch(\gl_r)_{u}$.
\end{prop}
\begin{proof}

This lemma is proved by Feynman diagram computations (see e.g.  \cite{Costello_2011}). Recall
$$
\uTr(1)=\sigma(u^ne^{\hbar \iota_{\Pi}/u}(\sum_{m=0}^{2n}\tr\int_{S^1_{cyc}[m+1]}d\theta_0\cdots d\theta_{m}e^{\hbar \partial_P}(1\otimes \underbrace{d_{2n}\widehat{\gamma}/\hbar\otimes\cdots\otimes d_{2n}\widehat{\gamma}/\hbar}_{\text{m}\ \widehat{\gamma}/\hbar' s}))).
$$
By the standard technique of Feynman diagrams (see e.g. \cite{bessis1980quantum}),  this can be expressed as
$$
\uTr(1)=u^n \bracket{\sum_{\substack{\Gamma\ \\}} {\hbar^{g(\Gamma)-k(\Gamma)}W_{\Gamma}\over |\text{Aut}(\Gamma)|}}.
$$
Let us explain the notation and terminology.
The sum of $\Gamma$ is over all diagrams (possibly disconnected) where vertices are given by $\hat \gamma$ and propagator is $\iota_{\hat{\Pi}}/u+P$. Precisely, let us view $\hat \gamma\in C^1(\g,\mathfrak{h}; \gl_r(\W_{2n}))$ and decompose
$$
\hat \gamma=\sum_{i\geq 0}\hat \gamma_i \hbar^i, \quad  \hat\gamma_i=\sum_{m\geq 0} \hat\gamma_i^{(m)}.
$$
$\hat \gamma_i^{(m)}$ collects those terms in $\hat \gamma$ whose value in $\gl_r(\W_{2n})$ is homogeneous of degree $m$ in $y$'s and degree $i$ in $\hbar$. Then in the above Feynman diagram expansion, $d_{2n}(\widehat \gamma_i^{(m)})$ contributes a $m$-valency vertex of loop number $i$ (with coefficient valued in $C^1(\g, \mathfrak{h}; \gl_r(\C))$). Since
$$
d_{2n}(\widehat \gamma_i^{(m)})=\sum_k dy^k {\pa\over \pa y^k}(\widehat \gamma_i^{(m)}),
$$
there will be two types of edges attached to the vertex,  by $y$'s or $dy$'s. We use blue color to indicate the $y$-edge and use red color to indicate the $dy$-edge. Therefore the $m$-valency vertex $d_{2n}(\widehat \gamma_i^{(m)})$ contains 1 red edge and $m-1$ blue edges. See the figure below.

\begin{figure}[H]

\definecolor{ccccff}{rgb}{0.8,0.8,1}
\definecolor{ffwwzz}{rgb}{1,0.4,0.6}
\definecolor{qqwuqq}{rgb}{0,0.39215686274509803,0}
\begin{tikzpicture}[line cap=round,line join=round,>=triangle 45,x=0.5cm,y=0.5cm]
\clip(-14.344921328686345,-4.22815692666691) rectangle (13.809418909855031,9.271859144440677);
\draw [line width=2pt,color=ffwwzz] (4.31874568416901,5.181736373538365)-- (4.398165543798182,3.037400163550744);
\draw [line width=2pt,color=ccccff] (4.398165543798182,3.037400163550744)-- (2.4126690530688886,2.640300865404888);
\draw [line width=2pt,color=ccccff] (4.398165543798182,3.037400163550744)-- (3.8819364562085656,1.0121937430068801);
\draw [line width=2pt,color=ccccff] (4.398165543798182,3.037400163550744)-- (5.629173368050344,1.2901632517089792);
\draw (5.271783999719071,4.380286022611292) node[anchor=north west] {$\cdots$};
\draw (2.7700584214001616,7.279110899076038) node[anchor=north west] {$dy^{i_{m}}$};
\draw (0.7051420710416967,4.224347005207094) node[anchor=north west] {$y^{i_1}$};
\draw (2.6906385617709896,1.047552620040249) node[anchor=north west] {$y^{i_{2}}$};
\draw (6.145402455639959,1.5432015672590326) node[anchor=north west] {$y^{i_{3}}$};
\draw (-12.312463153507112,4.062606584094608) node[anchor=north west] {$E_{ab}\hbar^{l}y^{i_{1}}\cdots y^{i_{m-1}}dy^{i_m}$};
\begin{scriptsize}
\draw [fill=qqwuqq] (4.398165543798182,3.037400163550744) circle (2.5pt);
\end{scriptsize}
\end{tikzpicture}
  \caption{This is the loop number $l$ vertex of valency $m$. The vertex is matrix valued. Here $E_{ab}$ is the fundamental matrix in $\gl_r(\C)$, whose only nontrivial entry is 1 at the $a$-th row and $b$-th column. }
  \end{figure}
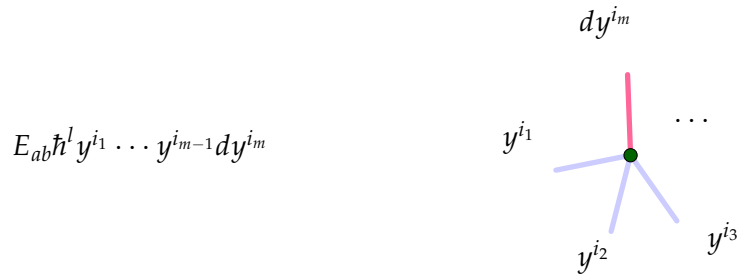

Note that our vertices are matrix valued, therefore it is important to keep their orders. This is a subtle difference with the usual Feynman diagram technique. For each diagram $\Gamma$,  let $V(\Gamma)$ be the set of ordered vertices of $\Gamma$. We use a bijection
 $$
 \chi:\{1,...,|V(\Gamma)|\}\rightarrow V(\Gamma),
 $$
 to label the order.  Let
$$
\ell :V(\Gamma)\rightarrow \mathbb{Z}_{\geq 0}
$$
be the map which assigns the loop number of each vertex. Let $b_1(\Gamma)$ be the first Betti number of $\Gamma$. We say $\Gamma$ is a $g$-loop diagram if
$$
g=g(\Gamma)= b_1(\Gamma)+\sum_{v\in V(\Gamma)} \ell(v).
$$
A diagram $\Gamma$ of loop number $g=0$ is called a tree diagram.  We denote
$$
k(\Gamma)=\sharp\{\text{connected components of $\Gamma$}\}.
$$

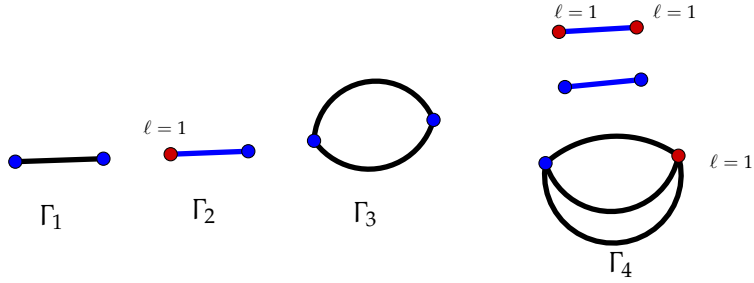
\begin{figure}[H]

\definecolor{ccqqqq}{rgb}{0.8,0,0}
\definecolor{qqqqff}{rgb}{0,0,1}
\begin{tikzpicture}[line cap=round,line join=round,>=triangle 45,x=0.5cm,y=0.5cm]
\clip(-10.742020626832202,-3.622815692666691) rectangle (13.412319611709174,9.271859144440677);
\draw [line width=2pt] (-10.572477996300686,2.6800107952194736)-- (-8.22959213724012,2.759430654848645);
\draw [line width=2pt,color=qqqqff] (-6.442645295583756,2.8785604442924018)-- (-4.377728945225291,2.957980303921573);
\draw (-10.215088627969413,1.8063923392985914) node[anchor=north west] {$\Gamma_1$};
\draw (-6.164675786881655,1.9652320585569336) node[anchor=north west] {$\Gamma_2$};
\draw (-1.876003366906382,1.8858121989277625) node[anchor=north west] {$\Gamma_3$};
\draw [line width=2pt,color=qqqqff] (3.8819364562085688,6.134774689088418)-- (5.946852806567033,6.253904478532174);
\draw [line width=2pt,color=qqqqff] (4.040776175466911,4.665507285948752)-- (6.06598259601079,4.864056935021679);
\draw [shift={(-0.9828817674902689,3.17555890241257)},line width=2pt]  plot[domain=0.3831285695108012:3.1049554169837195,variable=\t]({1*1.6487166676882155*cos(\t r)+0*1.6487166676882155*sin(\t r)},{0*1.6487166676882155*cos(\t r)+1*1.6487166676882155*sin(\t r)});
\draw [shift={(-1.173196877740176,4.263073818126329)},line width=2pt]  plot[domain=3.755543095488054:6.015726198186052,variable=\t]({1*1.7828889178357024*cos(\t r)+0*1.7828889178357024*sin(\t r)},{0*1.7828889178357024*cos(\t r)+1*1.7828889178357024*sin(\t r)});
\draw [shift={(5.267092639385909,3.1765002792214774)},line width=2pt]  plot[domain=3.4401081258406845:6.096911400331047,variable=\t]({1*1.8231771198815856*cos(\t r)+0*1.8231771198815856*sin(\t r)},{0*1.8231771198815856*cos(\t r)+1*1.8231771198815856*sin(\t r)});
\draw [shift={(5.4170614274633175,0.5070558514435932)},line width=2pt]  plot[domain=0.9573661894053154:2.29646802958683,variable=\t]({1*2.851726637517952*cos(\t r)+0*2.851726637517952*sin(\t r)},{0*2.851726637517952*cos(\t r)+1*2.851726637517952*sin(\t r)});
\draw [shift={(5.314674441925687,2.32954419401343)},line width=2pt]  plot[domain=-3.313474481234637:0.28412339304719564,variable=\t]({1*1.8169000117030647*cos(\t r)+0*1.8169000117030647*sin(\t r)},{0*1.8169000117030647*cos(\t r)+1*1.8169000117030647*sin(\t r)});
\draw (4.9143946313878,0.5519036728214657) node[anchor=north west] {$\Gamma_4$};
\draw (-7.44370873902665,4.040576092796707) node[anchor=north west] {\tiny{$\ell=1$}};
\draw (7.614669858779639,3.14927819835914) node[anchor=north west] {\tiny{$\ell=1$}};
\draw (3.481804770614665,7.1246222325989125) node[anchor=north west] {\tiny{$\ell=1$}};
\draw (6.1056925258253765,7.120271179817696) node[anchor=north west] {\tiny{$\ell=1$}};
\begin{scriptsize}
\draw [fill=qqqqff] (-10.572477996300686,2.6800107952194736) circle (2.5pt);
\draw [fill=qqqqff] (-8.22959213724012,2.759430654848645) circle (2.5pt);
\draw [fill=ccqqqq] (-6.442645295583756,2.8785604442924018) circle (2.5pt);
\draw [fill=qqqqff] (-4.377728945225291,2.957980303921573) circle (2.5pt);
\draw [fill=ccqqqq] (3.8819364562085688,6.134774689088418) circle (2.5pt);
\draw [fill=ccqqqq] (5.946852806567033,6.253904478532174) circle (2.5pt);
\draw [fill=qqqqff] (4.040776175466911,4.665507285948752) circle (2.5pt);
\draw [fill=qqqqff] (6.06598259601079,4.864056935021679) circle (2.5pt);
\draw [fill=qqqqff] (-2.6304920333835136,3.2359498126236717) circle (2.5pt);
\draw [fill=qqqqff] (0.5463023517833555,3.7918888300278697) circle (2.5pt);
\draw [fill=qqqqff] (3.5245470878772953,2.6403008654048885) circle (2.5pt);
\draw [fill=ccqqqq] (7.058730841375437,2.8388505144778162) circle (2.5pt);
\end{scriptsize}
\end{tikzpicture}
  \caption{Here are some examples of graphs: $\Gamma_1$ is a tree diagram; $\Gamma_2, \Gamma_3$ are  1-loop diagrams; $\Gamma_4$ is a disconnected graph, where $k(\Gamma_4)=3, g(\Gamma_4)=5$.}\label{TREEONELOOP}
\end{figure}

There will be two types of edges connecting vertices in $\Gamma$, which we call red and blue propagators. The red propagator connects the red $dy$-edges, while the blue propagator connects the blue $y$-edges. The Feynman rule of the red propagator is given by $\iota_{\Pi}/u$, and that of the blue propagator is given by $\pa_P$. See Figure \ref{Feynmanrule}.

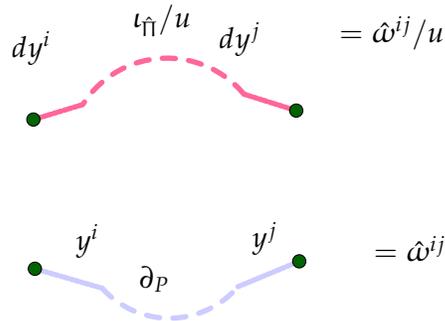
\begin{figure}[H]

\definecolor{ccccff}{rgb}{0.8,0.8,1}
\definecolor{ffwwzz}{rgb}{1,0.4,0.6}
\definecolor{qqwuqq}{rgb}{0,0.39215686274509803,0}
\begin{tikzpicture}[line cap=round,line join=round,>=triangle 45,x=0.5cm,y=0.5cm]
\clip(-14.344921328686345,-5.622815692666691) rectangle (13.809418909855031,9.271859144440677);
\draw [line width=2pt,dash pattern=on 1pt off 1pt,color=ffwwzz] (-3.74237006819192,2.124071777815276)-- (-2.4319423843105863,2.5211710759611314);
\draw [line width=2pt,dash pattern=on 1pt off 1pt,color=ffwwzz] (1.8567300356646868,2.7991405846632307)-- (3.2465775791751916,2.3623313567027893);
\draw [line width=2pt,dash pattern=on 1pt off 1pt,color=ccccff] (-3.702660138377334,-1.8469212036432805)-- (-1.91571,-2.36315);
\draw [line width=2pt,dash pattern=on 1pt off 1pt,color=ccccff] (1.697890316406343,-2.323440361418307)-- (3.325997438804363,-1.6483715545703528);
\draw [shift={(-0.18662797478103318,1.1022064659511712)},line width=2pt,dash pattern=on 5pt off 5pt,color=ffwwzz]  plot[domain=0.6930421752666642:2.577999041154312,variable=\t]({1*2.6561056760042066*cos(\t r)+0*2.6561056760042066*sin(\t r)},{0*2.6561056760042066*cos(\t r)+1*2.6561056760042066*sin(\t r)});
\draw [shift={(-0.12630011257076135,-0.7607703777751121)},line width=2pt,dash pattern=on 5pt off 5pt,color=ccccff]  plot[domain=3.87190454308131:5.574850413934222,variable=\t]({1*2.4020008740529413*cos(\t r)+0*2.4020008740529413*sin(\t r)},{0*2.4020008740529413*cos(\t r)+1*2.4020008740529413*sin(\t r)});
\draw (-1.3597742793167682,5.452454127605102) node[anchor=north west] {$\iota_{\hat{\Pi}}/u$};
\draw (4.120196035096081,5.142026443723779) node[anchor=north west] {$=\hat{\omega}^{ij}/u$};
\draw (-4.5849806998606473,4.744927145577923) node[anchor=north west] {$dy^i$};
\draw (0.92731610413753576,4.983186724465437) node[anchor=north west] {$dy^j$};
\draw (-2.908461542085617,-0.3887129007819069) node[anchor=north west] {$y^i$};
\draw (1.777310176035515,-0.3298731815235647) node[anchor=north west] {$y^j$};
\draw (-1.2009345600584247,-1.49678358994737126) node[anchor=north west] {$\partial_P$};
\draw (5.033524420831555,-0.6945143044901956) node[anchor=north west] {$=\hat{\omega}^{ij}$};
\begin{scriptsize}
\draw [fill=qqwuqq] (-3.74237006819192,2.124071777815276) circle (2.5pt);
\draw [fill=qqwuqq] (3.2465775791751916,2.3623313567027893) circle (2.5pt);
\draw [fill=qqwuqq] (-3.702660138377334,-1.8469212036432805) circle (2.5pt);
\draw [fill=qqwuqq] (3.325997438804363,-1.6483715545703528) circle (2.5pt);
\end{scriptsize}
\end{tikzpicture}
  \caption{Feynman rules for red and blue propagators.}\label{Feynmanrule}
\end{figure}

 The Feynman integral $W_{\Gamma}\in C^{|V(\Gamma)|}(\g, \mathfrak{h};\C)$ is defined as follows. Given $\xi_i\in \g/\mathfrak{h}$,

 \begin{align*}
  W_{\Gamma}(\xi_1\wedge\cdots\wedge\xi_{m})
  :=\sum_{\varepsilon\in S_{m}\ \text{and}\ \chi}\text{sign}(\varepsilon)W_{\Gamma^{\chi}}( \xi_{\varepsilon(1)}\otimes\cdots\otimes \xi_{\varepsilon(m)}), \quad m=|V(\Gamma)|.
\end{align*}
Here $\Gamma^{\chi}$ means the graph $\Gamma$ equipped with an ordering map $\chi$, and
$$
W_{\Gamma^{\chi}}(\xi_1\otimes\cdots\otimes \xi_{m}):=$$
$$
\sigma(\tr\int_{S^1_{cyc}[m+1]}d\theta_0\cdots d\theta_{m}\prod_{e\in E(\Gamma)}((\iota_{\hat{\Pi}}/u)_{e}+(\pa_P)_{e})(1\otimes d_{2n}{\widehat{\gamma}_{\upsilon(\chi(1))}}(\xi_{1})\otimes\cdots\otimes d_{2n}{\widehat{\gamma}_{\upsilon(\chi(m))}}(\xi_{m}))).
$$
Here $(\iota_{\hat{\Pi}}/u)_{e}$ means applying the operator $\iota_{\hat{\Pi}}/u$ to the two vertices indexed by the two ends of the edge $e$. The definition of $(\pa_P)_{e}$ is similar. 

\begin{figure}[H]
  \centering
\definecolor{ccccff}{rgb}{0.8,0.8,1}
\definecolor{ccqqqq}{rgb}{0.8,0,0}
\definecolor{zzttff}{rgb}{0.6,0.2,1}
\definecolor{ffwwzz}{rgb}{1,0.4,0.6}
\definecolor{cqcqcq}{rgb}{0.93,0.93,0.93}
\definecolor{uuuuuu}{rgb}{0.26666666666666666,0.26666666666666666,0.26666666666666666}
\begin{tikzpicture}[line cap=round,line join=round,>=triangle 45,x=0.5cm,y=0.5cm]
\clip(-7.654772264337302,-8.677404908666313) rectangle (11.951218868530606,11.195420365168422);
\draw [line width=2.2pt,color=cqcqcq] (0,0) circle (1.423cm);
\draw [shift={(1.8876840259623044,1.4566124389126767)},line width=2pt,dash pattern=on 5pt off 5pt,color=ccccff]  plot[domain=-1.2465225883591877:2.5609374881552633,variable=\t]({1*2.517593468607931*cos(\t r)+0*2.517593468607931*sin(\t r)},{0*2.517593468607931*cos(\t r)+1*2.517593468607931*sin(\t r)});
\draw [shift={(1.2227893158436987,0.9435531069451457)},line width=2pt,dash pattern=on 5pt off 5pt,color=ccccff]  plot[domain=-0.9064282957434706:2.220843195539546,variable=\t]({1*2.3794074325140757*cos(\t r)+0*2.3794074325140757*sin(\t r)},{0*2.3794074325140757*cos(\t r)+1*2.3794074325140757*sin(\t r)});
\draw [shift={(2.304915382423598,-2.655160804883435)},line width=2pt,dash pattern=on 5pt off 5pt,color=ffwwzz]  plot[domain=-3.299441297475185:1.5876531732381831,variable=\t]({1*4.423627676182236*cos(\t r)+0*4.423627676182236*sin(\t r)},{0*4.423627676182236*cos(\t r)+1*4.423627676182236*sin(\t r)});
\draw [shift={(-2.584278876641646,-2.23951206565799)},line width=2pt,dash pattern=on 5pt off 5pt,color=ffwwzz]  plot[domain=1.5197109219933875:6.191569026218227,variable=\t]({1*3.7883557511393096*cos(\t r+0*3.7883557511393096*sin(\t r)},{0*3.7883557511393096*cos(\t r)+1*3.7883557511393096*sin(\t r)});
\draw (-2.271556542433589,4.088627632175724) node[anchor=north west] {\tiny{$d_{2n}\widehat{\gamma}_0$}};
\draw (-4.610817054270099,2.718997376257476) node[anchor=north west] {\tiny{$d_{2n}\widehat{\gamma}_0$}};
\draw (2.796075404463954,-0.85512812640879904) node[anchor=north west] {\tiny{$d_{2n}\widehat{\gamma}_0$}};
\draw (-0.45627194465139226,2.064551992380105) node[anchor=north west] {\tiny{$d_{2n}\widehat{\gamma}_1$}};
\draw (-4.523903891548935,-1.8116715450482149) node[anchor=north west] {\tiny{$d_{2n}\widehat{\gamma}_1$}};
\draw (-1.1410870726105198,-01.49209115200062375) node[anchor=north west] {\tiny{$d_{2n}\widehat{\gamma}_1$}};
\draw [shift={(3.0866289545016556,2.392932009408751)},line width=2pt,dash pattern=on 5pt off 5pt,color=ffwwzz]  plot[domain=-1.6888388165797394:3.0077805664165647,variable=\t]({1*3.3337151494009567*cos(\t r)+0*3.3337151494009567*sin(\t r)},{0*3.3337151494009567*cos(\t r)+1*3.3337151494009567*sin(\t r)});
\draw (-6.528299412555655,-3.8248581080683595) node[anchor=north west] {\tiny{$\iota_{\hat{\Pi}}/u$}};
\draw (3.972199216150965,6.082134076406964) node[anchor=north west] {\tiny{$\iota_{\hat{\Pi}}/u$}};
\draw (3.4243471137836634,-5.0575253383947825) node[anchor=north west] {\tiny{$\iota_{\hat{\Pi}}/u$}};
\draw (1.872099490409641,4.073343034393535) node[anchor=north west] {\tiny{$\pa_P$}};
\draw (3.196075404463954,4.621195136760834) node[anchor=north west] {\tiny{$\pa_P$}};
\begin{scriptsize}
\draw [fill=zzttff] (-0.21728449357682808,2.8376933217455917) circle (2pt);
\draw [fill=ccqqqq] (2.230350501390146,1.7678383920418728) circle (2pt);
\draw [fill=ccqqqq] (-2.0637166428382705,-1.959793243559332) circle (2pt);
\draw [fill=ccqqqq] (1.2050923212049016,-2.578268495871105) circle (2.5pt);
\draw [fill=zzttff] (2.694022168163878,-0.9175840543896053) circle (2.5pt);
\draw [fill=zzttff] (-2.390833354916236,1.543901489230593) circle (2pt);
\end{scriptsize}
\end{tikzpicture}

  \caption{An example of $W_{\Gamma^{\chi}}$. }\label{Diagram example}
\end{figure}
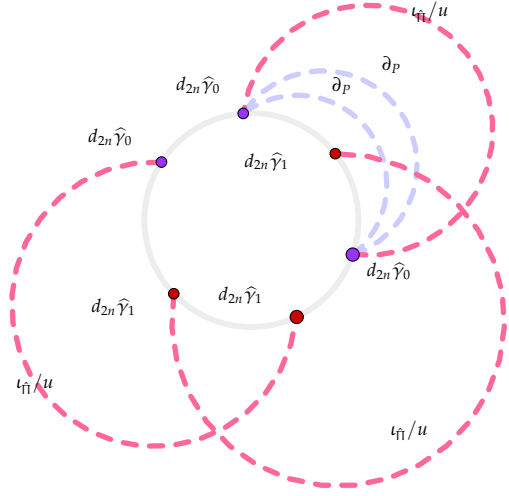

This gives the precise meaning of the Feynman diagram expansion
$$
\uTr(1)=u^n \bracket{\sum_{\substack{\Gamma\ \\}} {\hbar^{g(\Gamma)-k(\Gamma)}W_{\Gamma}\over |\text{Aut}(\Gamma)|}}.
$$
Note that by the symbol map $\sigma$ in the last step, all diagrams  do not contain external edges. In other words, all edges of vertices should be contracted by propagators.

The vertices of a tree diagram has only loop number $0$, and we observe that $\hat\gamma_0$ is valued in the scalar matrices $\C\PS{y^i}\cdot\text{Id}$. Since they are scalar matrices, we are free to move them around inside $\tr$ so their orders are not relevant. By collecting all contributions from tree diagrams, we can write
$$
\uTr(1)=u^n e^{\bracket{{1\over \hbar}\sum\limits^{\text{con. tree}}_{\Gamma_0}{W_{\Gamma_0}\over |\text{Aut}(\Gamma_0)|}}}\bracket{\sum^\prime_{\substack{\Gamma}} {\hbar^{g(\Gamma)-k(\Gamma)}W_{\Gamma}\over |\text{Aut}(\Gamma)|}}.
$$
Here $\sum\limits^{\text{con. tree}}_{\Gamma_0}$ is the sum of all connected tree diagrams, and $\sum\limits^\prime_{\substack{\Gamma}}$ is sum of all diagrams where the loop number of each connected component is at least $1$.

Observe that for each $\Gamma$ in the sum $\sum\limits^\prime_{\substack{\Gamma}}$, we have $g(\Gamma)-k(\Gamma)\geq 0$. Equality holds if and if each connected component of $\Gamma$ has loop number $1$. Therefore
$$
\uTr(1)=u^n e^{\bracket{{1\over \hbar}\sum\limits^{\text{con. tree}}_{\Gamma_0}{W_{\Gamma_0}\over |\text{Aut}(\Gamma_0)|}}}\bracket{\sum^{\text{1-loop}}_{\substack{\Gamma_1}} {W_{\Gamma_1}\over |\text{Aut}(\Gamma_1)|}+O(\hbar)}.
$$
Here $\sum\limits^{\text{1-loop}}_{\substack{\Gamma_1}}$  is sum of diagrams where the loop number of each connected component is precisely $1$.

The proposition follows by computing the tree diagrams ($g=0$) and 1-loop diagrams ($g=1$).

We first collect some details on the vertices and the curvatures. Recall the projection map for $f+\hbar A\in \mathfrak{g}$, where $f\in \mathbb{C}[[y^i]],A\in \gl_r(\mathcal{W}^+_{2n})$:
$$\pr_1((f+\hbar A))=\frac{1}{2}\partial_{y^i}\partial_{y^j}f(0)y^iy^j, \quad \pr_2((f+\hbar A))=hA_1(0).$$
 Here we write $\hbar A=\hbar A_1+\hbar^2 A_2+...$. The curvature has three parts:
\begin{align*}R_{1}((f+\hbar A),(g+\hbar B))&=\frac{1}{6}\hat\omega^{ij}(\partial_if(0)\partial_j\partial_p\partial_qg(0)+\partial_ig(0)\partial_j\partial_p\partial_qf(0))y^py^q\\
R_2((f+\hbar A),(g+\hbar B))&=\hbar \hat\omega^{ij} (\pa_if(0) \pa_jB_1(0)- \pa_iA_1(0) \pa_jg(0))\\
R_3((f+\hbar A),(g+\hbar B))&=\widehat\omega^{ij}\partial_if(0)\partial_j g(0)+O(\hbar^2).
\end{align*}

Let us also write down  a few leading terms in the expansion $\widehat{\gamma}=\sum_{i\geq 0}\widehat{\gamma}_i$:
\begin{align*}\widehat\gamma_0((f+\hbar A))&=\widehat{\gamma}(f)=\partial_if(0)y^i+\frac{1}{6}\partial_i\partial_j\partial_kf(0)y^i y^j y^k+O(y^4)\\
\widehat\gamma_1((f+\hbar A))&=\hbar A_1=\hbar\partial_i A_1(0)y^i+\hbar O(y^2) \\
\widehat{\gamma}_2(f+\hbar A)&=\hbar^2 A_2\\
\cdots\\
\widehat{\gamma}_l(f+\hbar A)&=\hbar^l A_l, l>0.
\end{align*}

Now we are ready to compute tree and 1-loop diagrams.

\noindent \textbf{Tree computation}: The only connected tree  without external edges contains two vertices with a red propagator, see Fig.\ref{Omega}. This term contributes $-\widehat{\omega}_0/u$. Here $\widehat{\omega}_0 \in C^2(\g, \mathfrak{h};\C)$ is the 2-cocycle
$$
  \widehat{\omega}_0((f+\hbar A),(g+\hbar B))=\widehat\omega^{ij}\partial_if(0)\partial_j g(0).
$$
In other words, $\hat \omega_0= R_3|_{\hbar=0}$ which descends to the symplectic 2-form under the Gelfand-Fuks map.

\begin{figure}[H]
  \centering
  \definecolor{ccccff}{rgb}{0.8,0.8,1}
  \definecolor{ffwwzz}{rgb}{1,0.4,0.6}
\definecolor{ccqqqq}{rgb}{0.8,0,0}
\definecolor{xdxdff}{rgb}{0.49019607843137253,0.49019607843137253,1}
\begin{tikzpicture}[line cap=round,line join=round,>=triangle 45,x=1cm,y=1cm]
\clip(-5.78,-2.34) rectangle (4.4,2.68);
\draw [line width=2pt,dash pattern=on 5pt off 5pt,color=ffwwzz] (-1.24,0)-- (1,0);
\draw (-3.8,0.5) node[anchor=north west] {$d_{2n}(\widehat{\gamma}^{(1)}_0)$};
\draw (1.58,0.5) node[anchor=north west] {$d_{2n}(\widehat{\gamma}^{(1)}_0)$};
\begin{scriptsize}
\draw [fill=xdxdff] (-1.24,0) circle (2.5pt);
\draw [fill=xdxdff] (1,0) circle (2.5pt);
\end{scriptsize}
\end{tikzpicture}

  \caption{This diagram gives $\hat{\omega}_0/\hbar u$}\label{Omega}
\end{figure}

\noindent \textbf{One-loop computation}: There are two types of 1-loop diagrams. Let us denote
\begin{itemize}
\item Type I:  Each connected component of $\Gamma_1$ has $b_1=1$ with only vertices of loop number $0$.
\item Type II: Each connected component of $\Gamma_1$ is a tree with one vertex of loop number $1$.
\end{itemize}
We have
$$
\sum^{\text{1-loop}}_{\substack{\Gamma_1}} {W_{\Gamma_1}\over |\text{Aut}(\Gamma_1)|}=\bracket{\sum\limits^{\text{1-loop}}_{\substack{\Gamma_1}: \text{type I}} {{W_{\Gamma_1}}\over |\text{Aut}(\Gamma_1)|}}\bracket{\sum\limits^{\text{1-loop}}_{\substack{\Gamma_1: \text{type II}}} {W_{\Gamma_1}\over |\text{Aut}(\Gamma_1)|}}.
$$
The reason we can separate these two contributions is because all the vertices of type I diagrams are scalar valued so we are free to move them inside the trace. This is similar to our tree diagram discussion above. We can further express the type I contributions in terms of connected diagrams
$$
\sum^{\text{1-loop}}_{\substack{\Gamma_1}} {W_{\Gamma_1}\over |\text{Aut}(\Gamma_1)|}=e^{\bracket{\sum\limits^{\text{con. 1-loop}}_{\substack{\Gamma_1}: \text{type I}} {{W_{\Gamma_1}}\over |\text{Aut}(\Gamma_1)|}}}\bracket{\sum^{\text{1-loop}}_{\substack{\Gamma_1: \text{type II}}} {W_{\Gamma_1}\over |\text{Aut}(\Gamma_1)|}}.
$$
Here $\sum\limits^{\text{con. 1-loop}}_{\substack{\Gamma_1}}$ is the sum of all connected type I 1-loop diagrams.

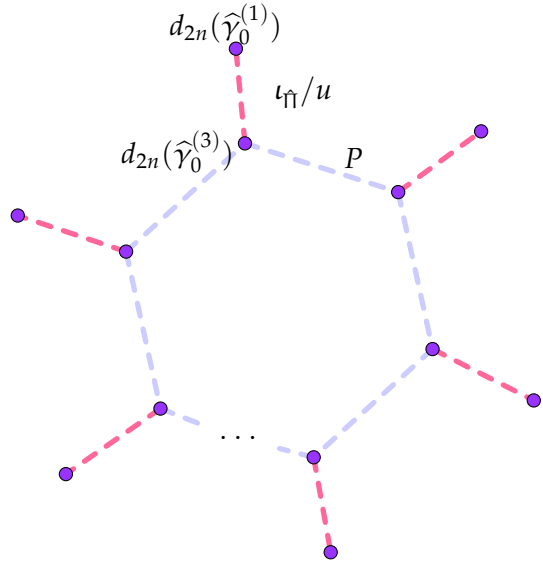
\begin{figure}[H]
  \centering
  \definecolor{ccccff}{rgb}{0.8,0.8,1}
\definecolor{ffwwzz}{rgb}{1,0.4,0.6}
\definecolor{ccqqqq}{rgb}{0.8,0,0}
\definecolor{zzttff}{rgb}{0.6,0.2,1}
\begin{tikzpicture}[line cap=round,line join=round,>=triangle 45,x=1cm,y=1cm]
\clip(-4.3,-3.72) rectangle (9.88,6.3);
\draw [line width=2pt,dash pattern=on 5pt off 5pt,color=ffwwzz] (2.28,3.5)-- (2.16,4.76);
\draw [line width=2pt,dash pattern=on 5pt off 5pt,color=ffwwzz] (4.317076581449593,2.8516798620205863)-- (5.42,3.66);
\draw [line width=2pt,dash pattern=on 5pt off 5pt,color=ffwwzz] (3.194153162899183,-0.6766402759588261)-- (3.42,-1.94);
\draw [line width=2pt,dash pattern=on 5pt off 5pt,color=ffwwzz] (4.774153162899183,0.7633597240411731)-- (6.12,0.08);
\draw [line width=2pt,dash pattern=on 5pt off 5pt,color=ffwwzz] (1.1570765814495916,-0.028320137979412907)-- (-0.1,-0.9);
\draw [line width=2pt,dash pattern=on 5pt off 5pt,color=ffwwzz] (0.7,2.06)-- (-0.74,2.54);
\draw [line width=2pt,dash pattern=on 5pt off 5pt,color=ccccff] (0.7,2.06)-- (2.28,3.5);
\draw [line width=2pt,dash pattern=on 5pt off 5pt,color=ccccff] (2.28,3.5)-- (4.317076581449593,2.8516798620205863);
\draw [line width=2pt,dash pattern=on 5pt off 5pt,color=ccccff] (4.317076581449593,2.8516798620205863)-- (4.774153162899183,0.7633597240411731);
\draw [line width=2pt,dash pattern=on 5pt off 5pt,color=ccccff] (4.774153162899183,0.7633597240411731)-- (3.194153162899183,-0.6766402759588261);
\draw [line width=2pt,dash pattern=on 5pt off 5pt,color=ccccff] (0.7,2.06)-- (1.1570765814495916,-0.028320137979412907);
\draw (1.14,5.5) node[anchor=north west] {$d_{2n}(\widehat{\gamma}^{(1)}_0)$};
\draw (1.8,-0.2) node[anchor=north west] {$\cdot\cdot\cdot$};
\draw [line width=2pt,dash pattern=on 5pt off 5pt,color=ccccff] (1.1570765814495916,-0.028320137979412907)-- (1.7,-0.24);
\draw [line width=2pt,dash pattern=on 5pt off 5pt,color=ccccff] (3.194153162899183,-0.6766402759588261)-- (2.74,-0.56);
\draw (3.46,3.58) node[anchor=north west] {$P$};
\draw (0.5,3.8) node[anchor=north west] {$d_{2n}(\widehat{\gamma}^{(3)}_0)$};
\draw (2.54,4.52) node[anchor=north west] {$\iota_{\hat{\Pi}}/u$};
\begin{scriptsize}
\draw [fill=zzttff] (2.28,3.5) circle (2.5pt);
\draw [fill=zzttff] (0.7,2.06) circle (2.5pt);
\draw [fill=zzttff] (1.1570765814495916,-0.028320137979412907) circle (2.5pt);
\draw [fill=zzttff] (3.194153162899183,-0.6766402759588261) circle (2.5pt);
\draw [fill=zzttff] (4.774153162899183,0.7633597240411731) circle (2.5pt);
\draw [fill=zzttff] (4.317076581449593,2.8516798620205863) circle (2.5pt);
\draw [fill=zzttff] (2.16,4.76) circle (2.5pt);
\draw [fill=zzttff] (5.42,3.66) circle (2.5pt);
\draw [fill=zzttff] (3.42,-1.94) circle (2.5pt);
\draw [fill=zzttff] (6.12,0.08) circle (2.5pt);
\draw [fill=zzttff] (-0.1,-0.9) circle (2.5pt);
\draw [fill=zzttff] (-0.74,2.54) circle (2.5pt);
\end{scriptsize}
\end{tikzpicture}
\caption{These diagrams give the $\hat{A}(\mathfrak{sp}_{2n})_u$}\label{AHAT}
\end{figure}
 A connected diagram of type I without external edges contains $k$ vertices of valency 3 that are connected by blue propagators to form a wheel, and they are further connected by red propagators to $k$ vertices of valency 1. See Figure \ref{AHAT}. The 3-valency vertex on the wheel is represented by $d_{2n}(\widehat{\gamma}^{(3)}_0)$ and the 1-valency vertex is represented by $d_{2n}(\widehat{\gamma}^{(1)}_0)$.

 The  sum of connected diagrams of type I  contributes
 $$\sum\limits^{\text{con. 1-loop}}_{\substack{\Gamma_1}: \text{type I}} {{W_{\Gamma_1}}\over |\text{Aut}(\Gamma_1)|}=\sum_{k\geq 0}C(k)\cdot\frac{1}{k!}\sum'_{\epsilon}\text{sign}(\epsilon)F_k(R_{1}(\xi_{\epsilon(1)},\xi_{\epsilon(2)}),...,R_{1}(\xi_{\epsilon(2k-1)},\xi_{\epsilon(2k)}))$$
where $F_k=k!\text{ch}_k=\text{tr}(X^k)\in (\Sym^k\mathfrak{sp}_{2n}^{\vee})^{\mathfrak{sp}_{2n}},$ and the prime means sum is over $\epsilon\in S_{2k}$ such that $\epsilon(2i-1)<\epsilon(2i)$. The coefficient $C(k)$ of this factor can be compute by integrals over $S^1[k]$
$$C(k)=\frac{u^{-k}}{k}\int_{S^1[k]}\pi^*_{12}(P^{S^1})\pi^*_{23}(P^{S^1})\cdot\cdot\cdot \pi^*_{k1}.(P^{S^1})d\theta^1\wedge...\wedge d\theta^k$$
When k is odd it is 0 and when k is even it is $\frac{u^{-k}}{k}\cdot\frac{2\zeta(k)}{{(2\pi i)}^k}$. Here $\zeta(k)$ is Riemann's zeta function. The total contribution from connected type I 1-loop diagram is

$$\sum_{k\geq 0}\frac{u^{-2k}}{2k}\frac{2\zeta(2k)}{(2\pi i)^{2k}}(2k)!ch_{2k}=\sum_{k\geq 0}\frac{2(2k-1)!\zeta(2k)}{(2\pi i)^{2k}}u^{-2k}ch_{2k}.$$
This is precisely $\log \hat{A}(\mathfrak{sp}_{2n})_u$ \cite{Hirzebruch1994Manifolds}. This computation is similar to that in \cite{Grady:2011jc,BVQandindex,CALAQUE20101839,Calaque}.

\begin{figure}[H]
  \centering
  \definecolor{ffwwzz}{rgb}{1,0.4,0.6}
\definecolor{ccqqqq}{rgb}{0.8,0,0}
\definecolor{xdxdff}{rgb}{0.49019607843137253,0.49019607843137253,1}
\begin{tikzpicture}[line cap=round,line join=round,>=triangle 45,x=1cm,y=1cm]
\clip(-5.78,-2.34) rectangle (4.4,2.68);
\draw [line width=2pt,dash pattern=on 5pt off 5pt,color=ffwwzz] (-1.24,0)-- (1,0);
\draw (-3.8,0.5) node[anchor=north west] {$d_{2n}(\widehat{\gamma}^{(1)}_0)$};
\draw (1.58,0.5) node[anchor=north west] {$ d_{2n}(\widehat{\gamma}^{(1)}_1)$};
\begin{scriptsize}
\draw [fill=xdxdff] (-1.24,0) circle (2.5pt);
\draw [fill=ccqqqq] (1,0) circle (2.5pt);
\end{scriptsize}
\end{tikzpicture}

  \caption{This diagram gives $R_2/u$}\label{Cherncharacter}
\end{figure}

A type II 1-loop diagram is computed by a trace of product of contributions from connected type II 1-loop diagrams. There is only one connected type II 1-loop diagram:  it contains two 1-valency vertices, one with loop number 0 and one with loop number 1, and contributes $R_2/u$. See Figure \ref{Cherncharacter}.

\begin{figure}[H]

\definecolor{ffqqtt}{rgb}{1,0,0.2}
\definecolor{ccqqqq}{rgb}{0.8,0,0}
\definecolor{zzttff}{rgb}{0.6,0.2,1}
\definecolor{cqcqcq}{rgb}{0.95,0.95,0.95}
\begin{tikzpicture}[line cap=round,line join=round,>=triangle 45,x=0.5cm,y=0.5cm]
\clip(-9.473836604763807,-6.559166208307436) rectangle (9.295723554263773,6.703950349764144);
\draw [line width=2pt,color=cqcqcq] (0,0) circle (1.502002805539927cm);
\draw [shift={(-1.3386701297502595,1.6001596533435798)},line width=2pt,dash pattern=on 5pt off 5pt,color=ffqqtt]  plot[domain=1.4386824968139298:3.096219025642099,variable=\t]({1*1.170674862228037*cos(\t r)+0*1.170674862228037*sin(\t r)},{0*1.170674862228037*cos(\t r)+1*1.170674862228037*sin(\t r)});
\draw [shift={(1.1959574900160403,1.935474857395687)},line width=2pt,dash pattern=on 5pt off 5pt,color=ffqqtt]  plot[domain=0.11133721473733463:1.9232800215379344,variable=\t]({1*1.0084017560964478*cos(\t r)+0*1.0084017560964478*sin(\t r)},{0*1.0084017560964478*cos(\t r)+1*1.0084017560964478*sin(\t r)});
\draw [shift={(-1.8744889484372422,-0.46467796663922983)},line width=2pt,dash pattern=on 5pt off 5pt,color=ffqqtt]  plot[domain=2.6347849098149374:4.134394949061946,variable=\t]({1*1.2870297999914175*cos(\t r)+0*1.2870297999914175*sin(\t r)},{0*1.2870297999914175*cos(\t r)+1*1.2870297999914175*sin(\t r)});
\draw [shift={(-0.4660441444638941,-1.962884225449742)},line width=2pt,dash pattern=on 5pt off 5pt,color=ffqqtt]  plot[domain=3.7616777102651597:5.196876510647548,variable=\t]({1*1.1752749755024883*cos(\t r)+0*1.1752749755024883*sin(\t r)},{0*1.1752749755024883*cos(\t r)+1*1.1752749755024883*sin(\t r)});
\draw [shift={(1.7734811864781361,-1.406335010207244)},line width=2pt,dash pattern=on 5pt off 5pt,color=ffqqtt]  plot[domain=-1.5317969914503085:0.1909055166026657,variable=\t]({1*0.9903209458470174*cos(\t r)+0*0.9903209458470174*sin(\t r)},{0*0.9903209458470174*cos(\t r)+1*0.9903209458470174*sin(\t r)});
\draw (3.1009745031883793,1.488712900781905) node[anchor=north west] {$\cdots$};
\draw (-5.661683342563566,0.826880737205479) node[anchor=north west] {$\sum \text{tr}$};
\draw (-1.27789659250839,3.07987802895) node[anchor=north west] {\tiny$d_{2n}\widehat{\gamma}^{(1)}_1$};
\draw (-4.27673012613439568,2.468224502875016) node[anchor=north west] {\tiny$d_{2n}\widehat{\gamma}^{(1)}_0$};
\draw (-3.0673012613439568,3.5006826780542406) node[anchor=north west] {\tiny$\iota_{\hat\Pi}/u$};
\draw (-3.155606711463445,5.4801942801473515) node[anchor=north west] {\Large{$\overbrace{}^{R_2/u}$}};
\begin{scriptsize}
\draw [fill=zzttff] (0.8478270904118058,2.8818782209620872) circle (2.5pt);
\draw [fill=ccqqqq] (2.1981156380726126,2.0475156891590314) circle (2.5pt);
\draw [fill=ccqqqq] (-1.1844573119580823,2.7606328498807637) circle (2.5pt);
\draw [fill=zzttff] (-2.5081401258431133,1.6532591946901263) circle (2.5pt);
\draw [fill=ccqqqq] (-2.9997371012470304,0.16003180650675478) circle (2.5pt);
\draw [fill=zzttff] (-2.577651138259624,-1.5426430128583273) circle (2.5pt);
\draw [fill=ccqqqq] (-1.422517041050551,-2.6458416663913464) circle (2.5pt);
\draw [fill=zzttff] (0.08134596843702753,-3.0029012120132745) circle (2.5pt);
\draw [fill=ccqqqq] (1.812093255581033,-2.3959029380944226) circle (2.5pt);
\draw [fill=zzttff] (2.7458107911401535,-1.2184235531825687) circle (2.5pt);
\end{scriptsize}
\end{tikzpicture}
  \caption{$\text{tr}(e^{R_2/u})$}\label{CHERNCHAR}
\end{figure}
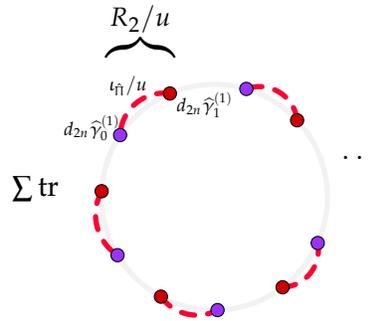

The sum of diagrams of type
II 1-loop diagram is (see Figure \ref{CHERNCHAR})
 $$\sum^{\text{1-loop}}_{\substack{\Gamma_1: \text{type II}}} {W_{\Gamma_1}\over |\text{Aut}(\Gamma_1)|}=\text{tr}(e^{R_2/u}).$$

Putting all the computations together, we get
$$
\uTr(1)=u^ne^{-R_3/u\hbar}(\hat{A}(\mathfrak{sp}_2n)_u\cdot \text{Ch}(\gl_r)_u+O(\hbar))
$$

\end{proof}
Now we can easily compute the  index.
\begin{thm}\label{index thm}As a cohomology class in $H^{\bullet}(\mathfrak{g},\mathfrak{h};\mathbb{K})$,
  $$\uTr(1)=u^ne^{-R_3/u\hbar}(\hat{A}(\mathfrak{sp}_{2n})_u\cdot Ch(\gl_r)_{u})\in H^{\bullet}(\mathfrak{g},\mathfrak{h};\mathbb{K}).$$
\end{thm}
\begin{proof} By Proposition \ref{ONELOOP}, we can expand
$$
e^{R_3/u\hbar}\uTr(1)=u^n(\hat{A}(\mathfrak{sp}_{2n})_u\cdot Ch(\gl_r)_{u}+O(\hbar)).
$$
By Proposition \ref{con:EACT}, terms in $O(\hbar)$ have nontrivial weights in $\nabla_{\hbar\pa_\hbar}$ and hence are $\pa_{\Lie}$-exact. Theorem follows by passing to cohomology.
\end{proof}

Now we discuss how to descent the above index theorem to geometric situation.
\begin{defn}\label{UNIVERSALTRACE}
Let $W_D$ be the quantum algebra defined in Section \ref{normalized}, which is the space of flat sections of Fedosov connection. We define the trace map $\Tr: W_D\rightarrow \mathbb{C}\LS{\hbar}$ by
  $$
  \Tr(\mathcal{O})=\int_M\desc(\uTr(\mathcal{O})).
  $$
 It does not depend on $u$ since $ \Tr(\mathcal{O})$ has cohomology degree $0$.
\end{defn}

It is easily to verify the following normalization property
\begin{prop}
 $$ \Tr(f)=\frac{(-1)^n}{\hbar^n}(\int_M\text{tr}(f)\frac{\omega^n}{n!}+O(\hbar)),\quad \forall f\in \Gamma(M,\End(E)).$$
\end{prop}
\begin{proof}
  It follows from the Feynman diagram computations. The $\hbar$-leading term of $\Tr(f)$ comes from the tree diagrams which only involves the vertex $d_{2n}(\hat\omega_{ij}y^i dx^j)$, and this gives
  $$
  \int_M \int_{BV}\text{tr}(f)e^{d_{2n}(\hat\omega_{ij}y^i dx^j)/\hbar}=\frac{(-1)^n}{\hbar^n}\int_M\text{tr}(f)\frac{\omega^n}{n!}.
  $$
\end{proof}

The index for the quantum algebra is given by
$$
\Tr(1)=\int_M\desc(\uTr(1)).
$$
Then Theorem \ref{index thm} and Proposition \ref{prop-char-desc} imply the following algebraic index theorem (see \cite{fedosov-bundle}).
\begin{thm}
Let $M$ be a compact smooth symplectic manifold of dimension $2n$, $E$ be a complex vector bundle over $M$. Let $W_D$ is the quantum algebra associated to the deformation quantization class $\omega_{\hbar}$. Let $\Tr$ be the  trace map on $W_D$  obtained from the universal trace $\uTr$ by descent construction. Then
 $$ \Tr(1)=\int_{M}e^{-\omega_{\hbar}/\hbar}\hat{A}(M)\Ch(E).$$
\end{thm}

\bibliographystyle{plain}

\end{document}